\documentclass{amsart}
\usepackage{graphicx}
\usepackage{amsmath}
\usepackage{amscd}
\usepackage{mathtools}
\usepackage{enumerate}
\usepackage{tikz-cd}
\usepackage[all,cmtip]{xy}
\usepackage{tikz}
\usepackage{amssymb}
\usepackage{ascmac} 
\usepackage{color}

\newtheorem{dfn}{Definition}[section]
\newtheorem{thm}[dfn]{Theorem}
\newtheorem{pro}[dfn]{Proposition}
\newtheorem{lem}[dfn]{Lemma}
\newtheorem{cro}[dfn]{Corollary}
\newtheorem{mar}[dfn]{Remark}

\newtheorem{ble}[dfn]{Problem}

\title[finite abelian groups of K3 surfaces]{finite abelian groups of K3 surfaces with smooth quotient}
\author{Taro Hayashi}
\address{
	(Taro Hayashi)
	Faculty of Agriculture,
	Kindai University,
	Nakamaticho 3327-204, Nara, Nara 631-8505, Japan
}
\email{haya4taro@gmail.com}
\date{\today}
\begin{document}
	
\maketitle
\begin{abstract}
The quotient space of a $K3$ surface by a finite group is an Enriques surface or a rational surface if it is smooth. 
Finite groups where the quotient space are Enriques surfaces are known.
In this paper, by analyzing effective divisors on smooth rational surfaces, we will study finite groups which act faithfully on $K3$ surfaces such that the quotient space are smooth. 
In particular, we will completely determine effective divisors on Hirzebruch surfaces such that there is a finite Abelian cover from a $K3$ surface to a Hirzebrunch surface such that the branch divisor is that effective divisor.  
Furthermore, we will decide the Galois group and give the way to construct that Abelian cover from an effective divisor on a Hirzebruch surface.
Subsequently, we study the same theme for Enriques surfaces.
\end{abstract}

\section{Introduction}
In this paper, we work over ${\mathbb C}$. A $K3$ surface $X$ is a smooth surface with $h^{1}({\mathcal O}_{X})=0$, and ${\mathcal O}_{X}(K_{X})\cong {\mathcal O}_{X}$, where $K_{X}$ is the canonical divisor of $X$. In particular, a $K3$ surface is simply connected.
Finite groups acting faithfully on $K3$ surfaces are well studied.
Let $\omega$ be a non-degenerated two holomorphic form. 
An automorphism $f$ of a $K3$ surface is called symplectic if $f^{\ast}\omega=\omega$. A finite subgroup $G$ of automorphisms of a $K3$ surface is called symplectic if $G$ is generated by symplectic automorphisms. 
The minimal resolution $X_{m}$ of the quotient space $X/G$ is one of a $K3$ surface, an Enriques surface, and a rational surface.
The surface $X_{m}$ is a $K3$ surface if and only if $G$ is a symplectic group. 
Symplectic groups are classified (see [\ref{bio:13},\ref{bio:12},\ref{bio:17}]).
If the quotient space of $X/G$ is smooth, then it is an Enriques surface or a rational surface.
The quotient space $X/G$ is an Enriques surface if and only if $G$ is isomorphic to ${\mathbb Z}/2{\mathbb Z}$ as a group and the fixed locus of $G$ is an empty set.
It is not well-known what kind of rational surface is realized as the quotient space of a $K3$ surface by a finite subgroup of Aut$(X)$.
In this paper, we will consider the case where $X/G$ is a smooth rational surface. 
The minimal model of smooth rational surfaces is the projective plane ${\mathbb P}^{2}$ or a Hirzebruch surfaces ${\mathbb F}_{n}$ where $n\not=1$, and ${\mathbb F}_{1}$ is isomorphic to ${\mathbb P}^2$ blow-up at a point.
In other words, all smooth rational surfaces which are not minimal are ${\mathbb F}_{1}$ or given by blowups of ${\mathbb F}_{n}$ for $0\leq n$.
Therefore, if $X/G$ is not ${\mathbb P}^{2}$, then
 there is a birational morphism $f:X/G\rightarrow {\mathbb F}_{n}$.
Our first main results are to analyze the quotient space $X/G$ and $G$ when $X/G$ is smooth.
\begin{thm}\label{thm:42}
Let $X$ be a $K3$ surface and $G$ be a finite subgroup of Aut$(X)$ such that $X/G$ is smooth.
For a birational morphism $f:X/G\rightarrow{\mathbb F}_{n}$ from the quotient space $X/G$ to a Hirzebruch surface ${\mathbb F}_{n}$, we get that $n=0,1,2,3,4,6,8$, or $12$.
Furthermore, if $n=6,8,12$, then $f$ is an isomorphism.
\end{thm}
Let $X$ be a $K3$ surface, and $\omega$ be a non-degenerated holomorphic two form of $X$.
For a finite group $G$ of Aut$(X)$, 
We write $G_s$ as a set of symplectic automorphisms of $G$. 
Then there is a short exact sequence: $1\rightarrow G_s\rightarrow G\rightarrow^{\varphi} C_n\rightarrow1$, where $C_n$ is a cyclic group of order $n$, and $\varphi(g):=\xi_g\in \mathbb C^{\ast}$ such that $g^{\ast}\omega=\xi_g\omega$ in H$^{2,0}(X)$ for $g\in G$.
\begin{thm}\label{thm:45}
Let $X$ be a $K3$ surface, $G$ be a finite subgroup of Aut$(X)$ such that $X/G$ is smooth.
Then the above exact sequence is split, i.e. there is a purely non-symplectic automorphism $g\in G$ such that $G$ is the semidirect product $G_{s}\rtimes\langle g \rangle$
of $G_{s}$ and $\langle g \rangle$. 
\end{thm}
Next, we will classify finite abelian groups which act faithfully on $K3$ surfaces and the quotient space is smooth.
\begin{dfn}
We will use the following notations.\\
${\mathcal AG}:=$
$
\left\{
\begin{aligned}
&{\mathbb Z}/2{\mathbb Z}^{\oplus a},\ {\mathbb Z}/3{\mathbb Z}^{\oplus b},\ {\mathbb Z}/4{\mathbb Z}^{\oplus c},{\mathbb Z}/2{\mathbb Z}^{\oplus d}\oplus{\mathbb Z}/3{\mathbb Z}^{\oplus e},{\mathbb Z}/2{\mathbb Z}^{\oplus f}\oplus{\mathbb Z}/4{\mathbb Z}^{\oplus g},\\
&{\mathbb Z}/2{\mathbb Z}\oplus{\mathbb Z}/3{\mathbb Z}^{\oplus h}\oplus{\mathbb Z}/4{\mathbb Z},{\mathbb Z}/2{\mathbb Z}\oplus{\mathbb Z}/4{\mathbb Z}\oplus{\mathbb Z}/8{\mathbb Z}\\
&:1\leq a\leq 5,\ 1\leq b,c\leq3,\\
&(d,e)=(1,1),(1,2),(1,3),(2,1),(2,2),(3,1)(3,2),\\
&(f,g)=(1,1),(1,2),(2,1),(3,1),\ h=1,2 
\end{aligned}
\right\}
$\\
${\mathcal AG}_{\infty}:=$
$
\left\{
\begin{aligned}
&{\mathbb Z}/2{\mathbb Z}^{\oplus a},\ {\mathbb Z}/4{\mathbb Z}^{\oplus c},\ {\mathbb Z}/2{\mathbb Z}^{\oplus d}\oplus{\mathbb Z}/3{\mathbb Z}^{e},\ {\mathbb Z}/2{\mathbb Z}^{\oplus 2}\oplus{\mathbb Z}/4{\mathbb Z}\\
&:a=1,2,3,4,5,\ c=1\ {\rm or}\ 3,\ (d,e)=(1,1),(1,2),\ {\rm or}\ (3,2)
\end{aligned}
\right\}
$\\
${\mathcal AG}_{0}:=$
$
\left\{
\begin{aligned}
&{\mathbb Z}/2{\mathbb Z}^{\oplus a},\ {\mathbb Z}/3{\mathbb Z}^{\oplus b},\ {\mathbb Z}/2{\mathbb Z}^{\oplus f}\oplus{\mathbb Z}/4{\mathbb Z}^{\oplus g}\\
&:a=1,2,3,4,5,\ b=1,2,3,\ (f,g)=(1,1),(1,2),(2,1),(3,1)
\end{aligned}
\right\}
$\\
${\mathcal AG}_{1}:=$
$
\left\{
\begin{aligned}
&{\mathbb Z}/2{\mathbb Z}^{\oplus a},\ {\mathbb Z}/4{\mathbb Z}^{\oplus 2},\ {\mathbb Z}/2{\mathbb Z}\oplus{\mathbb Z}/3{\mathbb Z}^{\oplus e},\ {\mathbb Z}/2{\mathbb Z}^{\oplus f}\oplus{\mathbb Z}/4{\mathbb Z},\\
&{\mathbb Z}/2{\mathbb Z}\oplus{\mathbb Z}/3{\mathbb Z}^{\oplus 2}\oplus{\mathbb Z}/4{\mathbb Z},\ {\mathbb Z}/2{\mathbb Z}\oplus{\mathbb Z}/4{\mathbb Z}\oplus{\mathbb Z}/8{\mathbb Z}\\
&:a=1,2,3,4,5,\ e=1,2,3,\ f=1,2,3
\end{aligned}
\right\}
$\\
${\mathcal AG}_{2}:=$
$
\left\{
\begin{aligned}
&{\mathbb Z}/2{\mathbb Z}^{\oplus a},\ {\mathbb Z}/3{\mathbb Z}^{b},\ {\mathbb Z}/2{\mathbb Z}^{2}\oplus{\mathbb Z}/3{\mathbb Z}^{\oplus 2},\ {\mathbb Z}/2{\mathbb Z}^{\oplus f}\oplus{\mathbb Z}/4{\mathbb Z}^{\oplus g}\\
&:a=1,2,3,4,\ b=1,2,3,\ (f,g)=(1,1),(1,2),(2,1),(3,1)
\end{aligned}
\right\}
$\\
${\mathcal AG}_{3}:=$
$
\left\{
\begin{aligned}
&{\mathbb Z}/2{\mathbb Z}^{\oplus d}\oplus{\mathbb Z}/3{\mathbb Z}^{\oplus e},\ {\mathbb Z}/2{\mathbb Z}\oplus{\mathbb Z}/3{\mathbb Z}\oplus{\mathbb Z}/4{\mathbb Z}\\
&:(d,e)=(1,1),(1,2),(3,1)
\end{aligned}
\right\}
$\\
${\mathcal AG}_4:=$
$
\left\{
\begin{aligned}
&{\mathbb Z}/2{\mathbb Z}^{\oplus a},\ {\mathbb Z}/4{\mathbb Z},\ {\mathbb Z}/2{\mathbb Z}\oplus{\mathbb Z}/3{\mathbb Z}^{\oplus 2},\ {\mathbb Z}/2{\mathbb Z}^{\oplus f}\oplus{\mathbb Z}/4{\mathbb Z}\\
&:a=1,2,3,\ f=1,2
\end{aligned}
\right\}
$\\
${\mathcal AG}_6:=$
$
\left\{
\begin{aligned}
{\mathbb Z}/3{\mathbb Z}^{\oplus b},\ {\mathbb Z}/2{\mathbb Z}^{\oplus 2}\oplus{\mathbb Z}/3{\mathbb Z}:b=1,2
\end{aligned}
\right\}
$\\
${\mathcal AG}_{8}:=$
$
\left\{
\begin{aligned}
{\mathbb Z}/2{\mathbb Z}\oplus{\mathbb Z}/4{\mathbb Z}
\end{aligned}
\right\}
$\\
${\mathcal AG}_{12}:=$
$
\left\{
\begin{aligned}
{\mathbb Z}/2{\mathbb Z}\oplus{\mathbb Z}/3{\mathbb Z}
\end{aligned}
\right\}
$
\end{dfn}
Notice that ${\mathcal AG}=\bigcup_{n=0,1,2,3,4,6,8,12,\infty}{\mathcal AG}_{n}$.
In [\ref{bio:1}],  Uluda$\breve{{\rm g}}$ classified finite abelian groups for the case $X/G$ is ${\mathbb P}^{2}$.
Furthermore, he gave the way to construct the pair $(X,G)$ where $X$ is a $K3$ surface and $G$ is a finite subgroup of Aut$(X)$ such that $X/G\cong{\mathbb P}^{2}$.
We have the following.
\begin{thm}\label{thm:111}([\ref{bio:1}])
Let $X$ be a $K3$ surface and $G$ be a finite abelian subgroup of ${\rm Aut}(X)$ such that the quotient space $X/G$ is isomorphic to ${\mathbb P}^{2}$. 
Then $G$ is one of ${\mathcal AG}_{\infty}$ as a group.
Conversely, for every $G\in{\mathcal AG}_{\infty}$, there is a $K3$ surface $X'$ and a finite abelian subgroup $G'$ of ${\rm Aut}(X')$ such that $X'/G'\cong{\mathbb P}^{2}$ and $G'\cong G$ as a group. 
\end{thm}
By analyzing the irreducible components of the branch locus of the quotient map $p:X\rightarrow X/G$, 
we will study a pair $(X,G)$ consisting of a K3 surface $X$ and a finite abelian subgroup $G$ of Aut$(X)$ such that the quotient space $X/G$ is smooth. 
More precisely, the preimage of the branch locus of $p$ is $\bigcup_{g\in G\backslash\{{\rm id}_{X}\}}$Fix$(g)$ where Fix$(g):=\{x \in X : g(x) =x\}$. 
Recall that for an automorphism $f$ of finite order of a $K3$ surface, if Fix$(f)$ contains a curve, then $f$ is non-symplectic. 
The fixed locus of a non-symplectic automorphism is well-known, e.g. [\ref{bio:2},\ref{bio:9},\ref{bio:7}].
By analyzing the fixed locus of non-symplectic automorphisms of $G$ from the branch divisor of the quotient map,
we will reconstruct $G$ from the branch divisor of the quotient map.
In Section 4, we will investigate the relationship between a branch divisor and exceptional divisors of blowups.
Based on the above results,
we will obtain our second main result.
\begin{thm}\label{thm:1}
Let $X$ be a $K3$ surface and $G$ be a finite abelian subgroup of ${\rm Aut}(X)$ such that the quotient space $X/G$ is smooth. 
Then $G$ is one of ${\mathcal AG}$ as a group.
Conversely, for every $G\in{\mathcal AG}$, there is a $K3$ surface $X'$ and a finite abelian subgroup $G'$ of ${\rm Aut}(X')$ such that $X'/G'$ is smooth and $G'\cong G$ as a group. 
\end{thm}
Furthermore, in Section 3, for a Hirzebruch surface ${\mathbb F}_{n}$ and an effective divisor $B$ on ${\mathbb F}_{n}$, we will give a necessary and sufficient condition for the existence of a finite Abelian cover $f:X\rightarrow {\mathbb F}_{n}$ such that $X$ is a $K3$ surface and the branch divisor of $f$ is $B$. 
In other words, we will solve a part of the Fenchel's problem for Hirzebruch surfaces.
In addition, we will decide the Galois group and give the way to construct $f:X\rightarrow {\mathbb F}_{n}$ from the pair ${\mathbb F}_{n}$ and $B$.
\begin{thm}\label{thm:2} 
Let $X$ be a $K3$ surface and $G$ be a finite abelian subgroup of ${\rm Aut}(X)$ such that the quotient space $X/G$ is isomorphic to ${\mathbb F}_{n}$.
Then $G$ is one of ${\mathcal AG}_{n}$ as a group. 
Conversely, for every $G\in{\mathcal AG}_{n}$, there is a $K3$ surface $X'$ and a finite abelian subgroup $G'$ of ${\rm Aut}(X')$ such that 
$X'/G'$ is isomorphic to $\mathbb F_n$ and $G'\cong G$ as a group. 
\end{thm}
Subsequently, we will get a similar result for Enriques surfaces.
\begin{dfn}
We use the following notations.\\
${\mathcal AG}(E):=$
$
\left\{
\begin{aligned}
&{\mathbb Z}/2{\mathbb Z}^{\oplus a},\ {\mathbb Z}/4{\mathbb Z}^{\oplus 2},\ {\mathbb Z}/2{\mathbb Z}^{\oplus f}\oplus{\mathbb Z}/4{\mathbb Z},\ {\mathbb Z}/4{\mathbb Z}\oplus{\mathbb Z}/8{\mathbb Z}\\
&:a=2,3,4\ f=1,2 
\end{aligned}
\right\}
$\\
${\mathcal AG}_{\infty}(E):=$
$
\left\{
\begin{aligned}
{\mathbb Z}/2{\mathbb Z}^{\oplus a}:\ a=2,3,4
\end{aligned}
\right\}
$\\
${\mathcal AG}_{0}(E):=$
$
\left\{
\begin{aligned}
&{\mathbb Z}/2{\mathbb Z}^{\oplus a},\ {\mathbb Z}/4{\mathbb Z}^{\oplus 2},\ {\mathbb Z}/2{\mathbb Z}^{\oplus f}\oplus{\mathbb Z}/4{\mathbb Z}\\
&:a=2,3,4,\ f=1,2 
\end{aligned}
\right\}
$\\
${\mathcal AG}_{1}(E):=$
$
\left\{
\begin{aligned}
&{\mathbb Z}/2{\mathbb Z}^{\oplus a},\ {\mathbb Z}/2{\mathbb Z}^{\oplus f}\oplus{\mathbb Z}/4{\mathbb Z},\ {\mathbb Z}/4{\mathbb Z}\oplus{\mathbb Z}/8{\mathbb Z}\\
&:a=2,3,4,\ f=1,2
\end{aligned}
\right\}
$\\
${\mathcal AG}_{2}(E):=$
$
\left\{
\begin{aligned}
{\mathbb Z}/2{\mathbb Z}^{\oplus a},\ {\mathbb Z}/4{\mathbb Z}^{\oplus 2},\ {\mathbb Z}/2{\mathbb Z}^{\oplus 2}\oplus{\mathbb Z}/4{\mathbb Z}:\ a=2,3
\end{aligned}
\right\}
$\\
${\mathcal AG}_{4}(E):=$
$
\left\{
\begin{aligned}
{\mathbb Z}/2{\mathbb Z}\oplus{\mathbb Z}/4{\mathbb Z}
\end{aligned}
\right\}
$
\end{dfn}
Then ${\mathcal AG}(E)=\bigcup_{n=0,1,2,4,\infty}{\mathcal AG}_{n}(E)$.
Let $E$ be an Enriques surface $E$ and $H$ be a finite abelian subgroup of ${\rm Aut}(E)$ such that $E/H$ is smooth. 
Let $X$ be the $K3$-cover of $E$, and $G:=\{ s\in{\rm Aut}(X):s\ {\rm is\ a\ lift\ of\ some}\ h\in H\}$. Then $G$ is a finite abelian subgroup of Aut$(X)$, $G$ has a non-symplectic involution whose fixed locus is empty, and $X/G=E/H$. 
The case of $E/H\cong{\mathbb P}^{2}$ was studied in [\ref{bio:6}].
By analyzing the groups of Theorem \ref{thm:1}, we get the following theorems:
\begin{thm}\label{thm:21}
Let $E$ be an Enriques surface and $H$ be a finite subgroup of Aut$(E)$ such that the quotient space $E/H$ is smooth.
If  there is a birational morphism from $E/H$ to a Hirzebruch surface ${\mathbb F}_{n}$, then $0\leq n\leq 4$. 	
In particular, if the quotient space $E/H$ is a Hirzebruch surface ${\mathbb F}_{n}$,  
then $n=0,1,2,4$.
\end{thm}
\begin{thm}\label{thm:24} 
Let $E$ be an Enriques surface and $H$ be a finite abelian subgroup of Aut$(E)$ such that the quotient space $E/H$ is isomorphic to ${\mathbb F}_{n}$.
Then $H$ is one of ${\mathcal AG}_{n}(E)$ as a group. 
Conversely, for every $H'\in{\mathcal AG}_{n}(E)$, there is an Enriques surface $E'$ and a finite abelian subgroup $H'$ of ${\rm Aut}(E')$ such that $E'/H'$ is smooth and $H'\cong H$ as a group. 
\end{thm}
\begin{thm}\label{thm:23}
Let $E$ be an Enriques surface and $H$ be a finite abelian subgroup of Aut$(E)$ such that the quotient space $E/H$ is smooth. 
Then $H$ is one of ${\mathcal AG}(E)$ as a group.
Conversely, for every $H\in{\mathcal AG}(E)$, there is an Enriques surface $E'$ and a finite abelian subgroup $H'$ of ${\rm Aut}(E')$ such that $E'/H'$ is smooth and $H'\cong H$ as a group. 
\end{thm}
Section $2$ is preliminaries.
In Section $3.1$, we will give examples for pairs $(X',G')$ described in Theorem \ref{thm:1}. 
In other words, we will show that for each $G\in{\mathcal AG}_{n}$ where $n=0,1,2,3,4,6,8,12$,
there is a pair $(X',G')$ where $X'$ is a K3 surface and $G'$ is a finite abelian subgroup of ${\rm Aut}(X')$ such that $G\cong G'$ as a group and  $X'/G'\cong{\mathbb F}_{n}$. 
Furthermore, we will give the way to construct $(X',G')$, and 
we will show that the way to construct $(X',G')$ is uniquely determined up to isomorphism from the branch divisor of the quotient map $p:X'\rightarrow X'/G'$.
In Section 3.2, we will describe branch divisors and abelian groups for the case where the quotient space is a Hirzebruch surface.
In Section 4, first, we will show Theorem \ref{thm:42} and \ref{thm:45}. 
Next, we will show that for a pair $(X,G)$ where $X$ is a $K3$ surface and $G$ is a finite abelian subgroup, if $X/G$ is smooth, then $G$ is isomorphic to one of ${\mathcal AG}$ as a group.
In Section 5, we will show Theorem \ref{thm:21}, \ref{thm:24}, and \ref{thm:23}.
In Section 6, based on $[\ref{bio:5}]$, we will describe the existence of a $K3$ surface $X$  and a finite group $G$ which is not necessarily an abelian group such that $X/G$ is smooth, and $X/G$ is neither $\mathbb P^2$ nor an Enriques surface.
\section{Preliminaries} 
We recall the properties of the Galois cover. 
\begin{dfn}
Let $f:X\rightarrow M$ be a branched covering, where $M$ is a complex manifold and $X$ is a normal complex space. 
We call $f:X\rightarrow M$ the Galois cover if there is a subgroup $G$ of Aut$(X)$ such that $X/G \cong M$ and  $f:X\rightarrow M$ is isomorphic to the quotient map $p:X\rightarrow X/G\cong M$.  
We call $G$ the Galois group of $f:X\rightarrow M$. 
Furthermore, if $G$ is an abelian group, then we call $f:X\rightarrow M$ the Abelian cover.	
\end{dfn}
\begin{dfn}
Let $f:X\rightarrow M$ be a finite branched covering, where $M$ is a complex manifold and $X$ is a normal complex space and $\varDelta$ be the branch locus of $f$.
Let $B_{1},\ldots ,B_{s}$ be irreducible hypersurfaces of $M$ and positive integers $b_{1},\ldots,b_{s}$, where $b_{i}\geq 2$ for $i=1,\ldots,s$. 
If $\varDelta=B_{1}\cup\ldots\cup B_{s}$ and for every j and for any irreducible component $D$ of $f^{-1}(B_{j})$ the ramification index at $D$ is $b_{j}$,  
then  we call an effective divisor $B:=\sum_{i=1}^{s}b_{i}B_{i}$ the branch divisor of $f$. 
\end{dfn}
Let $X$ be a normal projective variety and $G$ be a finite subgroup of ${\rm Aut}(X)$.  
Let $Y:=X/G$ be the quotient space and $p:X\rightarrow Y$ be the quotient map. 
The branch locus, denoted by $\varDelta$ is a subset of $Y$ given by $\varDelta:=\{y\in Y|\ |p^{-1}(y)|<|G|\}$.
It is known that $\varDelta$ is an algebraic subset of dimension ${\rm dim}\,(X)-1$ if $Y$ is smooth $[\ref{bio:15}]$. 
Let $\{B_{i}\}_{i=1}^{r}$ be the irreducible components of $\varDelta$ whose dimension is $1$.
Let $D$ be an irreducible component of $D$ of $p^{-1}(B_{j})$ and $G_{D}:=\{g\in G:g_{|D}={\rm id}_{D}\}$. Then 
the ramification index at $D$ is $b_{j}:=|G_{D}|$, and the positive integer $b_{j}$ is independent of an irreducible component of $p^{-1}(B_{j})$.
Then $b_{1}B_{1}+\cdots+b_{r}B_{r}$ is the branch divisor of $G$.
We state the facts (Theorem \ref{thm:3} and \ref{thm:4}) of the Galois cover theory which we need.
\begin{thm}\label{thm:3}
([\ref{bio:8}])
For a complex manifold $M$ and an effective divisor $B$ on $M$,
if there is a branched covering map $f:X\rightarrow M$ where $X$ is a simply connected complex manifold $X$ and the branch divisor of $f$ is $B$, then there is a subgroup $G$ of Aut$(X)$ such that $X/G \cong M$ and  $f:X\rightarrow M$ is isomorphic to the quotient map $p:X\rightarrow X/G\cong M$. 
Furthermore, a pair $(X,G)$ is a unique up to isomorphism.
\end{thm}
\begin{thm}\label{thm:4}
([\ref{bio:8}])
For a complex manifold $M$ and an effective divisor $B:=\sum_{i=1}^{n}b_{i}B{i}$ on $M$, where $B_{i}$ is an irreducible hypersurface for $i=1,\ldots,n$. 
Let $f:X\rightarrow M$ be a branched cover whose branch divisor is $B$ and where $X$ is a simply connected complex manifold.
Then for a branched cover $g:Y\rightarrow M$ whose branch divisor is $\sum_{j=1}^{m}b'_{j}B_{j}$ and $b'_{j}$ is divisible by $b_{i}$ and $m\leq n$, there is a branched cover $h:X\rightarrow Y$ such that $f=g\circ h$.
\end{thm}
Let $X$ be a $K3$ surface and $G$ be a finite subgroup of Aut$(X)$ such that $X/G$ is smooth.
Since $K3$ surfaces are simply connected, $G$ is determined by the branch divisor of the quotient map $p:X\rightarrow X/G$ from Theorem \ref{thm:3}. 
In order to classify finite abelin groups $G$ which acts on $K3$ surfaces and the quotient space is smooth, 
we will search a smooth rational surface $S$ and an effective divisor $B$ on $S$ such that there is a $K3$ surface and a finite subgroup $G$ of Aut$(X)$ such that $X/G\cong S$ and the branch divisor of the quotient map $p:X\rightarrow X/G$ is $B$.
There is the problem which is called Fenchel's problem.
\begin{ble}
Let $M$ be a projective manifold.
Give a necessary and sufficient condition on an effective divisor $D$ on $M$ for the existence of a finite Galois (resp. Abelian) cover $\pi:X\rightarrow M$ whose branch divisor is $D$.
\end{ble}
The Fenchel's problem was originally for compact Riemann surfaces and was answered by Bundgaard-Nielsen [\ref{bio:3}] and Fox [\ref{bio:4}].  
\begin{thm}\label{thm:10}
([\ref{bio:3}],[\ref{bio:4}]) 
Let $k\geq 1$ and let $D:=\sum_{i=1}^km_ix_i$ be a divisor on a compact Riemann surface $M$ where $x_i\in M$ and $m_i\in\mathbb Z$ for $i=1,\ldots,k$. 
Then there is a finite Galois cover $p:X\rightarrow M$ such that the branch divisor of $p$ is $D$ except for\\
i) $M={\mathbb P}^{1}$ and $k=1$, and\\
ii) $M={\mathbb P}^{1}$, $k=2$, and $m_{1}\not=m_{2}$.

Furthermore, for the case $M={\mathbb P}^{1}$ there exists a finite Abelian cover ${\mathbb P}^{1}\rightarrow{\mathbb P}^{1}$ whose branch divisor is $D$ if and only if \\
i) $k=2$ and $m_{1}=m_{2}$ or \\
ii) $k=3$ and $m_{1}=m_{2}=m_{3}=2$.
\end{thm}
In order to study the cover of the Galois cover $X\rightarrow X/G$, the following theorem is useful.
\begin{thm}\label{thm:6}
Let $X$ be a smooth projective variety, $G$ be a finite subgroup of ${\rm Aut}(X)$ such that $X/G$ is smooth.
Let $p:X\rightarrow X/G$ be the quotient map, and $B:=b_{1}B_{1}+\ldots+b_{r}B_{r}$ be the branch divisor of $p$.
Then 
\[ K_{X}=p^{\ast}K_{X/G}+\sum_{i=1}^{r}\frac{b_{i}-1}{b_{i}}p^{\ast}B_{i}\]
where $K_{X}$ $($resp. $K_{X/G})$ is the canonical divisor of $X$ $($resp. $X/G)$.
\end{thm}
Let $X$ be a $K3$ surface and $G$ be a finite subgroup of Aut$(X)$ such that $X/G$ is smooth, and $B$ be the branch divisor of the quotient map
$p:X\rightarrow X/G$.
The canonical line bundle of a $K3$ surface is trivial. 
By Theorem \ref{thm:6}, the branch divisor is restricted in the Picard group of the smooth rational surface $X/G$, i.e. $B$ must satisfy \[K_{X/G}+\sum_{i=1}^{r}\frac{b_{i}-1}{b_{i}}B_{i}=0\ {\rm in\ Pic}_{\mathbb Q}(X/G).\]
In Section 3.1, we will show that for a Hirzebruch surface ${\mathbb F}_{n}$, if ${\mathbb F}_{n}$ has an effective divisor $B=\sum_{i=1}^{k}b_{i}B_{i}$ where $B_{i}$ is an irreducible curve and $b_{i}\geq 2$ for $i=1,\ldots,k$ such that $\sum_{i=1}^{k}\frac{b_{i}-1}{b_{i}}B_{i}+K_{S}=0$ in Pic$_{\mathbb Q}({\mathbb F}_{n})$, then $0\leq n\leq 12$.
In Section 4,  we will show Theorem \ref{thm:42} by using Theorem \ref{thm:6}.

The following theorem is important for checking the structure of $G$ from the branch divisor.
\begin{thm}\label{thm:5}(See [\ref{bio:16}])
For a $K3$ surface $X$ and a finite subgroup $G$ of Aut$(X)$ such that $X/G$ is smooth. 
Let $B:=\sum_{i=1}^{k}b_{i}B_{i}$ be the branch divisor of the quotient map $p:X\rightarrow X/G$. 
We put $p^{\ast}B_{i}=\Sigma_{j=1}^{l}b_{i}C_{i,j}$ where $C_{i,j}$ is an irreducible curve for $j=1,\ldots,l$.
Let $G_{C_{i,j}}:=\{g\in G:\ g_{|C_{i,j}}={\rm id}_{C_{i,j}}\}$, $G_{i}$ be a subgroup of $G$, which is generated by $G_{C_{i,1}},\ldots, G_{C_{i,l}}$, and $I\subset\{1,\ldots,k\}$ be a subset.
Then, the following holds.\\
$i)$ If $(X/G)\backslash\cup_{i\in I}B_{i}$ is simply connected, then $G$ is generated by $\{G_{j}\}_{j\in\{1,\ldots,k\}\backslash I}$.\\
$ii)$ $G_{C_{i,j}}\cong{\mathbb Z}/b_{i}{\mathbb Z}$ and $G_{C_{i,j}}$ is generated by a purely non-symplectic automorphism of order $b_{i}$.\\
$iii)$ If $G$ is abelian, then there is an automorphism $g\in G$ such that $\cup_{j=1}^{l}C_{i,j}\subset{\rm fix}(g)$, and hence $C_{i,j}$ are pairwise disjoint.\\
$iv)$ If the self intersection number $(B_{i}\cdot B_{i})$ of $B_{i}$ is positive, then $l=1$, and hence $G_{i}$ is generated by a purely non-symplectic automorphism of order $b_{i}$.
\end{thm}
\begin{proof}
We will show $i)$.	
We assume that $(X/G)\backslash\cup_{i\in I}B_{i}$ is simply connected.
Let $H$ be the subgroup of $G$ which is generated by $\{G_{j}\}_{j\in\{1,\ldots,k\}\backslash I}$, and $X_{0}:=X\backslash \cup_{i\in I}p^{-1}(B_{i})$.
Then $G$ and $H$ act on $X_{0}$.
We assume that $G\not=H$.
Let $Y:=X_{0}/H$ be the quotient space, and $G':=G/H$.
Then $G'$ acts faithfully on $Y$, $Y/G'\cong (X/G)\backslash\cup_{i\in I}B_{i}$, and the branch locus of $Y\rightarrow Y/G'$ is a finite set. 
Since $(X/G)\backslash\cup_{i\in I}B_{i}$ is smooth and simply connected, this is a contradiction.
Therefore,  $G$ is generated by $\{G_{j}\}_{j\in\{1,\ldots,k\}\backslash I}$. 
	
Since $X$ is a $K3$ surface, an automorphism whose fixed locus contains a curve can only be purely non-symplectic. 
Therefore, by the definition of the ramification index $b_{i}$, we get ii).
	
We will show $iii)$ and $iv)$.	
Since $B_{i}$ is contained in the branch locus, we get $p^{-1}(B_{i})=\bigcup_{j=1}^{l}C_{i,j}\subset\bigcup_{g\in G}{\rm fix}(g)$. 
Since $G$ is finite, for each $j$, there is $s_{j}\in G$ such that $C_{i,j}\subset{\rm fix}(s_{j})$. 
Since $B_{i}$ is irreducible, we get that $p(C_{i,j})=p(C_{i,k})$ for $1\leq j<k\leq l$. 
Therefore, there is $t\in G$ such that $t(C_{i,j})=C_{i,k}$. 
Since $C_{i,j}\subset{\rm fix}(s_{j})$ and $t(C_{i,j})=C_{i,k}$, we obtain that $C_{i,k}\subset {\rm fix}(t\circ s_{j}\circ t^{-1})$. 
Since $G$ is abelian, we have $s_{j}=t\circ s_{j}\circ t^{-1}$. 
We get $iii)$.  
If the self intersection number $(B_{i}\cdot B_{i})$ of $B_{i}$ is positive, then by Hodge index theorem, we get $l=1$.
By $ii)$,  $G_{i}\cong{\mathbb Z}/b_{i}{\mathbb Z}$ is generated by a purely non-symplectic automorphism of order $b_{i}$.
\end{proof}
Let $X$ be a $K3$ surface and $G$ be a finite abelian subgroup of ${\rm Aut}(X)$ such that $X/G$ is smooth and $B:=\sum_{i=1}^{k}b_{i}B_{i}$ be the branch divisor of the quotient map $p:X\rightarrow X/G$. 
If $k=1$, then by Theorem \ref{thm:5} $G=G_{B_{1}}\cong{\mathbb Z}/b_{1}{\mathbb Z}$. 
We assume that $k=2$. By Theorem \ref{thm:5}, $G$ is generated by $G_{B_{1}}\cong\mathbb Z/b_{1}\mathbb Z$ and $G_{B_{2}}\cong\mathbb Z/b_{2}\mathbb Z$.
Moreover, we assume that the intersection $B_{1}\cap B_{2}$ of $B_{1}$ and $B_{2}$ is not an empty set.
Since $B_{1}\cap B_{2}\not=\emptyset$, $p^{-1}(B_{1})\cap p^{-1}(B_{2})\not=\emptyset$. 
Since the fixed locus of an automorphism is a pairwise disjoint set of points and curves, we get $G_{B_{1}}\cap G_{B_{2}}=\{{\rm id}_{X}\}$. 
Therefore, $G=G_{B_{1}}\oplus G_{B_{2}}$, but in the case of $k\geq3$ it is not necessarily $G=\bigoplus_{i=1}^{k}G_{B_{i}}$ even if $B_{i}\cap B_{j}\not=\emptyset$ for $1\leq i<j\leq k$.

For an irreducible component $B_{i}$ of $B$ we write $p^{\ast}B_{i}=\sum_{j=1}^{l}b_{i}C_{j}$ where $C_{j}$ is a smooth curve for $j=1,\ldots,l$.
Since the degree of $p$ is $|G|$,
by $iv)$ of Theorem \ref{thm:5}, we get that $|G|(B_{i}\cdot B_{i})=b^{2}_{i}l(C_{j}\cdot C_{j})$ for $j=1,\ldots, l$.
If the self intersection number $(B_{i})^{2}$ of $B_{i}$ is positive, then by $iv)$ of Theorem \ref{thm:5}, we get that $l=1$ and the genus of $C_{1}$ is $2$ or more.
If $(B_{i})^{2}$ is zero, then $C_{1},\ldots,C_{l}$ are elliptic curves.
If $(B_{i})^{2}$  is negative, then $C_{1},\ldots,C_{l}$ are rational curves.
Recall that there is $g\in G$ such that $g$ is a non-symplectic automorphism of order $b_{i}$ and $C_{1},\ldots,C_{l}$ are contained in Fix$(g)$. 
There are many results on the number of curves, the genus of curves, and the number of isolated points of the fixed locus of a non-symplectic automorphism.  
We use them to search $B$ such that there is a Galois cover $f:X\rightarrow S$ such that $X$ is a $K3$ surface and the branch divisor of $f$ is $B$ and we use them to restore $G$ from $B$. 
Here $S$ is a smooth rational surface and $B$ is an effective divisor on $S$.
\section{Abelian groups of K3 surfaces with Hirzebruch surfaces} 
Here, we give the list of a numerical class of an effective divisor $B=\sum_{i=1}^{k}b_{i}B_{i}$ on ${\mathbb F}_{n}$ such that $B_{i}$ is a smooth curve for each $i=1,\ldots,k$ and $K_{{\mathbb F}_{n}}+\sum_{i=1}^{k}\frac{b_{i}-1}{b_{i}}B_{i}=0$ in Pic$_{\mathbb Q}({\mathbb F}_{n})$. 
\begin{dfn}
For a Hirzebruch surface ${\mathbb F}_{n}$ where $n\in{\mathbb Z}_{\geq 0}$, 
we take two irreducible curves $C$ and $F$ such that ${\rm Pic}({\mathbb F}_{n})={\mathbb Z}C\oplus{\mathbb Z}F$, $(C\cdot F)=1$,
$(F\cdot F)=0$, $(C\cdot C)=-n$, and $K_{{\mathbb F}_{n}}=-2C-(n+2)F$ in ${\rm Pic}({\mathbb F}_{n})={\mathbb Z}C\oplus{\mathbb Z}F$.
Notice that for $n=0$, $C=pr_{1}^{\ast}{\mathcal O}_{{\mathbb P}^{1}}(1)$ and $F=pr_{2}^{\ast}{\mathcal O}_{{\mathbb P}^{1}}(1)$, and 
for $n\geq1$, $C$ is the unique curve on $\mathbb F_{n}$ such that the self intersection number is negative, and $F$ is the fibre class of the conic bundle of $\mathbb F_{n}$.	
\end{dfn}
\begin{lem}\label{thm:7}
Let ${\mathbb F}_{n}$ be a Hirzebruch surface where $n\not=0$, and $C'\subset{\mathbb F}_{n}$ be an irreducible curve. 
Then one of the following holds:\\
1) $C'=C$.\\
2) $C'=F$\ in\  ${\rm Pic}({\mathbb F}_{n})$.\\
3) $C'=aC+bF$ where $a\geq1$ and $b\geq na$.
\end{lem}	
\begin{dfn}
Let $X$ be a $K3$ surface and $G$ be a finite subgroup of ${\rm Aut}(X)$ such that $X/G\cong{\mathbb F}_{n}$. 
Let $B:=\sum_{i=1}^{l}b_{i}B_{i}$ be the branch divisor of the quotient map $p:X\rightarrow X/G$.  
For each $B_{i}$, there are integers $\alpha_{i},\beta_{i}$ such that
$B_{i}=\alpha_{i}C+\beta_{i}F$ in ${\rm Pic}({\mathbb F}_{n})$.
We call 
\[ \sum_{i=1}^lb_i(\alpha_iC+\beta_{i}F)\]
as the numerical class of $B$.
\end{dfn}
\begin{pro}\label{thm:8}
Let $X$ be a $K3$ surface and $G$ be a finite subgroup of Aut$(X)$ such that $X/G\cong{\mathbb F}_{n}$.  
Then $0\leq n\leq 12$. 
\end{pro}
\begin{proof}
We assume that $X/G\cong{\mathbb F}_{n}$ where $n\geq1$.
Let $B$ be the branch divisor of the quotient map $p:X\rightarrow X/G$.
We write $B:=\sum_{i=1}^{k}b_{i}B_{i}+\sum_{j=1}^{l}b'_{j}B'_{j}$ such that $B_{i}\not=F$ and $B'_{j}=F$ in Pic$({\mathbb F}_{n})$ for $i=1,\ldots,k$ and $j=1,\ldots,l$. 
Since the canonical line bundle of a $K3$ surface is trivial and Pic$(\mathbb F_{n})$ is torsion free, 
by Theorem \ref{thm:6}, we get that 
\[ 0=K_{\mathbb F_n}+\sum_{i=1}^{k}\frac{b_{i}-1}{b_{i}}B_{i}+\sum_{j=1}^{l}\frac{b'_{j}-1}{b'_{j}}B'_{j}\ \ {\rm in\ \ Pic}(\mathbb F_{n}). \]
Since $B_{i}$ is an irreducible curve for $i=1,\ldots,k$, 
there are integers $c_{i},d_{i}$ such that
$B_{i}=c_{i}C+d_{i}F$ in Pic$({\mathbb F}_{n})$ and $(c_{i},d_{i})=(1,0)$ or $d_{i}\geq nc_{i}>0$. 
By $K_{{\mathbb F}_{n}}=-2C-(n+2)F$ in Pic$({\mathbb F}_{n})={\mathbb Z}C\oplus{\mathbb Z}F$, we get that 
\[ \begin{cases}
\; 2=\sum_{i=1}^{k}\frac{b_{i}-1}{b_{i}}c_{i} \\
\; n+2=\sum_{i=1}^{k}\frac{b_{i}-1}{b_{i}}d_{i}+\sum_{j=1}^{l}\frac{b'_{j}-1}{b'_{j}}.
\end{cases} \]
Since $b_{i}\geq2$, $\frac{1}{2}\leq\frac{b_{i}-1}{b_{i}}<1$.
Since $2=\sum_{i=1}^{k}\frac{b_{i}-1}{b_{i}}c_{i}$, $\sum_{i=1}^{k}c_{i}=3$ or $4$.
By a simple calculation, we get that i) $\sum_{i=1}^{k}c_{i}=4$ if and only if $b_{1}=\cdots=b_{k}=2$, and ii) if $\sum_{i=1}^{k}c_{i}=3$, 
then $(b_{1},\ldots,b_{k};c_{1},\ldots,c_{k})$ where $c_{1}\leq\cdots \leq c_{k}$ is one of $(3;3)$, $(2,4;1,2)$, $(3,3;1,2)$, $(2,3,6;1,1,1)$, $(2,4,4;1,1,1)$, and $(3,3,3;1,1,1)$.

We assume that $(c_{i},d_{i})\not=(1,0)$ for $i=1,\ldots,k$, i.e. $C$ is not an irreducible component of $B$. 
Since $d_{i}\geq nc_{i}$ for $i=1,\ldots,k$, by $2=\sum_{i=1}^{k}\frac{b_{i}-1}{b_{i}}c_{i}$ and $n+2=\sum_{i=1}^{k}\frac{b_{i}-1}{b_{i}}d_{i}+\sum_{j=1}^{l}\frac{b'_{j}-1}{b'_{j}}$,
we get that $n+2\geq 2n+\sum_{j=1}^{l}\frac{b'_{j}-1}{b_{j}'}$. Since $\frac{b'_{i}-1}{b'_{i}}\geq 0$, we get $0\leq n\leq2$.
	
We assume that $(c_{i},d_{i})=(1,0)$ for some $1\leq i\leq k$, i.e. $C$ is an irreducible component of $B$.
For simplify, we assume that $i=1$. 
In the same way as above, we get that $n+2\geq n(2-\frac{b_{1}-1}{b_{1}})$. 
Since  $2\leq b_{1}\leq6$, we obtain $0\leq12\leq n$.
\end{proof}
Notice that by simple calculations, there are not a $K3$ surface $X$ and a finite subgroup $G$ of Aut$(X)$ such that $X/G\cong{\mathbb F}_{l}$ for $l=10,11$.

In section 6, we will give the list of a numerical class of an effective divisor $B=\sum_{i=1}^{k}b_{i}B_{i}$ on ${\mathbb F}_{n}$ such that $B_{i}$ is a smooth curve for each $i=1,\ldots,k$ and $K_{{\mathbb F}_{n}}+\sum_{i=1}^{k}\frac{b_{i}-1}{b_{i}}B_{i}=0$ in Pic$({\mathbb F}_{n})$. 
\subsection{Abelian covers of a Hirzebruch surface by a K3 surface} 
Let $X$ be a $K3$ surface, $G$ be a finite abelian subgroup of Aut$(X)$ such that $X/G$ is a Hirzebruch surface ${\mathbb F}_{n}$, and $B$ be the branch divisor of the quotient map $p:X\rightarrow X/G$. 
In this section, we will decide the numerical class of $B$. 
Notice that since $G$ is abelian and the quotient space $X/G$ is smooth, the support of $B$ and that of $p^{\ast}B$ are simple normal crossing. 

Furthermore, we will show that the structure as a group of $G$ depends only on the numerical class of $B$ by Theorem \ref{thm:5},
and we will give the way to construct $X$ and $G$ which depends only on the numerical class of $B$ by Theorem \ref{thm:3} and the cyclic cover.
As a result the following will follow.
For each $G\in{\mathcal AG}_{n}$ where $n=0,1,2,3,4,6,8,12$,
there is a pair $(X,G')$ where $X$ is a $K3$ surface and $G'$ is a finite abelian subgroup of ${\rm Aut}(X)$ such that $G\cong G'$ as a group and  $X/G'\cong{\mathbb F}_{n}$.
In $[\ref{bio:18}]$, the case where $G\cong\mathbb Z/2\mathbb Z$ is studied.
\begin{thm}\label{thm:12}
(see [\ref{bio:0}, Chapter I, Section 17]) 
Let $M$ be a smooth projective variety, and $D$ be a smooth effective divisor on $M$. 
Then if the class ${\mathcal O}_{M}(D)/n\in$Pic$(M)$, then there is the Galois  cover $f:X\rightarrow M$ whose branch divisor is $nD$ and the Galois group is isomorphic to ${\mathbb Z}/n{\mathbb Z}$ as a group. 
\end{thm}
For $n\geq 0$, a Hirzebruch surface ${\mathbb F}_{n}$ is isomorphic to a variety ${\mathcal F}_{n}$ in ${\mathbb P}^{1}\times{\mathbb P}^{2}$
\[ {\mathcal F}_{n}:=\{([X_{0}:X_{1}],[Y_{0}:Y_{1}:Y_{2}])\in{\mathbb P}^{1}\times{\mathbb P}^{2}:\ X^{n}_{0}Y_{0}=X_{1}^{n}Y_{1}\}.\]
From here, we assume that ${\mathbb F}_{n}={\mathcal F}_{n}$.  
The first projection gives the fibre space structure $f:{\mathbb F}_{n}\rightarrow{\mathbb P}^{1}$ such that the numerical class of the fibre of $f$ is $F$, and 
\[ C= \{([X_{0}:X_{1}],[Y_{0}:Y_{1}:Y_{2}])\in{\mathbb F}_{n}:\ Y_{0}=Y_{1}=0\}, \]
is the unique irreducible curve on $\mathbb F_{n}$ such that the self intersection number is negative.
Let $a$ and $b$ be positive integers such that $b\geq na$.
Furthermore,
we put
\[ F(X_{0},X_{1},Y_{0},Y_{1},Y_{2}):=\sum_{0\leq i\leq b-na,0\leq j,k\leq a,j+k\leq a}t_{i,j,k}X^{i}_{0}X^{b-na-i}_{1}Y^{j}_{0}Y^{k}_{1}Y^{a-j-k}_{2}  \]
where $t_{i,j,k}\in{\mathbb C}$, and
\[ B_{F}:=\{([X_{0}:X_{1}],[Y_{0}:Y_{1}:Y_{2}])\in{\mathbb F}_{n}:\ F(X_{0},X_{1},Y_{0},Y_{1},Y_{2})=0\}.  \]
If $B_{F}$ is an irreducible curve of ${\mathbb F}_{n}$, then $B_{F}=aC+bF$ in Pic$({\mathbb F}_{n})$.

Let $g_{1}$ and $g_{m}$ be automorphisms of ${\mathbb P}^{1}$ which are induced by matrixes
\[ g_{1}:=\begin{pmatrix}
0 & 1 \\
1 & 0 
\end{pmatrix},\ 
g_{m}:=\begin{pmatrix}
1 & 0 \\
0 & \zeta_{m} 
\end{pmatrix}, \] 
where $\zeta_{m}$ is an $m$-th root of unity $m\geq2$.
Then  $\langle g_{1},g_{2}\rangle\cong{\mathbb Z}/2{\mathbb Z}^{\oplus 2}$, and $\langle g_{m}\rangle\cong{\mathbb Z}/m{\mathbb Z}$ for $m\geq 2$. 
Here for a subset $S$ of group $G$, $\langle S \rangle$ is the subgroup of $G$ which is generated by $S$.
Then
\[{\mathbb P}^{1}\cong{\mathbb P}^{1}/\langle g_{1},g_{2}\rangle\ {\rm and}\ {\mathbb P}^{1}\cong{\mathbb P}^{1}/\langle g_{m}\rangle,\] 
and the quotient maps are isomorphic to 
\[ {\mathbb P}^{1}\ni [z_{0}:z_{1}]\mapsto[(z^{2}_{0}+z_{1}^{2})^{2}:(z_{0}^{2}-z^{2}_{1})^{2}]\in{\mathbb P}^{1}\ {\rm and}\ {\mathbb P}^{1}\ni [z_{0}:z_{1}]\mapsto[z^{m}_{0}:z^{m}_{1}]\in{\mathbb P}^{1}\]
for $m\geq2$, and the branch divisors are 
\[ 2x_{0}+2x_{1}+2x_{2}\  {\rm and}\ mx_{0}+mx_{1},\]
where $x_0:=[1:0]$, $x_1:=[0:1]$, and $x_2:=[1:1]$.

The above Galois covers ${\mathbb P}^{1}\rightarrow{\mathbb P}^{1}/\langle g_{1},g_{2}\rangle\cong{\mathbb P}^{1}$ and  ${\mathbb P}^{1}\rightarrow{\mathbb P}^{1}/\langle g_{m}\rangle\cong{\mathbb P}^{1}$ naturally induce the Galois covers of ${\mathbb P}^{1}\times{\mathbb P}^{1}$ and ${\mathbb F}_{n}$ whose Galois groups are induced by $g_{m}$ for $m\geq2$.
We will explain in a bit more detail for ${\mathbb F}_{n}$.
For ${\mathbb P}^{1}\rightarrow{\mathbb P}^{1}/\langle g_{1},g_{2}\rangle$,
let ${\mathbb P}^{1}\times_{{\mathbb P}^{1}}{\mathbb F}_{n}$ be the fibre product of ${\mathbb P}^{1}\rightarrow{\mathbb P}^{1}/\langle g_{1},g_{2}\rangle$ and $f:{\mathbb F}_{n}\rightarrow{\mathbb P}^{1}$. 
Let $p:{\mathbb P}^{1}\times_{{\mathbb P}^{1}}{\mathbb F}_{n}\rightarrow {\mathbb F}_{n}$ be the natural projection of the fibre product.
Then 
\[{\mathbb P}^{1}\times_{{\mathbb P}^{1}}{\mathbb F}_{n}\cong{\mathbb F}_{4n}\] 
and  $p:{\mathbb P}^{1}\times_{{\mathbb P}^{1}}{\mathbb F}_{n}\rightarrow {\mathbb F}_{n}$
is the Galois cover such that the branch divisor of $p$ is 
\[ 2F+2F+2F\ {\rm in\ Pic}({\mathbb F}_{n}),\]
 and the Galois group is isomorphic to ${\mathbb Z}/2{\mathbb Z}^{\oplus 2}$ as a group, which is induced by $\langle g_{1},g_{2}\rangle$.
Let $C_{m}$ is the irreducible curve on ${\mathbb F}_{m}$ such that the self intersection number is negative and $F_{m}$ is the numerical class of the fibre ${\mathbb F}_{m}\rightarrow{\mathbb P}^{1}$ for $m\geq1$. 
Then
\[ p^{\ast}C_{n}=C_{4n}\  {\rm and}\ p^{\ast}F_{n}=4F_{4n}\ {\rm in\ Pic}({\mathbb F}_{4n}).\]
For ${\mathbb P}^{1}\rightarrow{\mathbb P}^{1}/\langle g_{m}\rangle$,
let ${\mathbb P}^{1}\times_{{\mathbb P}^{1}}{\mathbb F}_{n}$ be the fibre product of ${\mathbb P}^{1}\rightarrow{\mathbb P}^{1}/\langle g_{m}\rangle$ and $f:{\mathbb F}_{n}\rightarrow{\mathbb P}^{1}$.
Let  $p:{\mathbb P}^{1}\times_{{\mathbb P}^{1}}{\mathbb F}_{n}\rightarrow {\mathbb F}_{n}$ be the natural projection of the fibre product.
Then 
\[{\mathbb P}^{1}\times_{{\mathbb P}^{1}}{\mathbb F}_{n}\cong{\mathbb F}_{mn},\] 
$p:{\mathbb P}^{1}\times_{{\mathbb P}^{1}}{\mathbb F}_{n}\rightarrow {\mathbb F}_{n}$
is the Galois cover such that the branch divisor of $p$ is 
\[ mF+mF\ {\rm in\ Pic}({\mathbb F}_{n}),\]
and the Galois group is isomorphic to ${\mathbb Z}/m{\mathbb Z}$ as a group, which is induced by $\langle g_{m}\rangle$, and 
\[ p^{\ast}C_{n}=C_{mn}\  {\rm and}\ p^{\ast}F_{n}=mF_{mn}\ {\rm in\ Pic}({\mathbb F}_{mn}).\]
\begin{dfn}
From here, we use the notation.
$B^{k}_{i,j}$ (or simply $B_{i,j}$) is a smooth curve on ${\mathbb F}_{n}$ such that  
$B^{k}_{i,j}=iC+jF$ in ${\rm Pic}({\mathbb F}_{n})$ for $n\geq0$, where $k\in{\mathbb N}$. 
\end{dfn}
\begin{pro}\label{pro:1}
For each numerical classes (\ref{1},\ref{14},\ref{50}) of the list in section 6, 
there are a $K3$ surface $X$ and a finite abelian subgroup $G$ of Aut$(X)$ such that $X/G\cong{\mathbb P}^{1}\times{\mathbb P}^{1}$ and the numerical class of the branch divisor $B$ of the quotient map $p:X\rightarrow X/G$ is (\ref{1},\ref{14},\ref{50}).   

Furthermore, for a pair $(X,G)$ of a $K3$ surface $X$ and a finite abelian subgroup $G$ of Aut$(X)$, if $X/G\cong{\mathbb P}^{1}\times{\mathbb P}^{1}$ and the numerical class of the branch divisor $B$ of the quotient map $p:X\rightarrow X/G$ is (\ref{1},\ref{14},\ref{50}), then $G$ is isomorphic to ${\mathbb Z}/3{\mathbb Z}$, ${\mathbb Z}/3{\mathbb Z}^{\oplus2}$, ${\mathbb Z}/3{\mathbb Z}^{\oplus3}$, in order, as a group. 
\end{pro}
\begin{proof}
Let $B_{3,3}$ be a smooth curve on ${\mathbb P}^{1}\times{\mathbb P}^{1}$. 
Then the numerical class of $3B_{3,3}$ is (\ref{1}).
By Theorem \ref{thm:12}, there is the Galois cover $p:X\rightarrow{\mathbb P}^{1}\times{\mathbb P}^{1}$ such that the branch divisor is $3B_{3,3}$ and the Galois group is ${\mathbb Z}/3{\mathbb Z}$ as a group.
By Theorem \ref{thm:6}, the canonical divisor of $X$ is a numerically trivial. 
By [\ref{bio:11}], $X$ is not a bi-ellitptic surface.
 By [\ref{bio:6}], $X$ is not an abelian surface.  
If $X$ is an Enriques surface, then there is the Galois cover $q:X'\rightarrow{\mathbb P}^{1}\times{\mathbb P}^{1}$ such that $X'$ is a $K3$ surface, the Galois group is ${\mathbb Z}/2{\mathbb Z}\oplus{\mathbb Z}/3{\mathbb Z}$ as a group, and the branch divisor is $3B_{3,3}$. 
By Theorem \ref{thm:5}, this is a contradiction. 
Therefore, $X$ is a $K3$ surface.

In addition, let $(X',G')$ be a pair of a $K3$ surface $X'$ and a finite abelian subgroup $G'$ of Aut$(X')$ such that $X'/G'\cong{\mathbb P}^{1}\times{\mathbb P}^{1}$ and the numerical class of the branch divisor $B'$ of the quotient map $p':X'\rightarrow X'/G'$ is (\ref{1}).
By Theorem \ref{thm:5}, $G'\cong{\mathbb Z}/3{\mathbb Z}$ as a group.
Since the support of $B'$ is smooth, there is a smooth curve $B'_{3,3}$ such that  $B'=3B'_{3,3}$.
Then by the above discussion, there is the Galois cover $f:X\rightarrow{\mathbb P}^{1}\times{\mathbb P}^{1}$ such that $X$ is a $K3$ surface, 
the branch divisor is $B'$, and the Galois group $G$ is ${\mathbb Z}/3{\mathbb Z}$ as a group.
Since a $K3$ surface is simply connected, by Theorem \ref{thm:3}, the pair $(X',G')$ is isomorphic to the pair $(X,G)$.\\

Let $B^{1}_{1,0}$, $B^{2}_{1,0}$, and $B_{1,3}$ be smooth curves on ${\mathbb P}^{1}\times{\mathbb P}^{1}$ such that $B^{1}_{1,0}+B^{2}_{1,0}+B_{1,3}$ is simple normal crossing. 
Then the numerical class of $3B^{1}_{1,0}+3B^{2}_{1,0}+3B_{1,3}$ is (\ref{14}).
Let $p:{\mathbb P}^{1}\times{\mathbb P}^{1}\rightarrow{\mathbb P}^{1}\times{\mathbb P}^{1}$ be the Galois cover such that the branch divisor is $3B^{1}_{1,0}+3B^{2}_{1,0}$, and the Galois group is ${\mathbb Z}/3{\mathbb Z}$ as a group, which is induced by the Galois cover ${\mathbb P}^{1}\ni [z_{0}:z_{1}]\mapsto[z^{3}_{0}:z^{3}_{1}]\in{\mathbb P}^{1}$.
Since $B^{1}_{1,0}+B^{2}_{1,0}+B_{1,3}$ is simple normal crossing, 
$p^{\ast}B_{1,3}$ is a reduced divisor on ${\mathbb P}^{1}\times{\mathbb P}^{1}$ such that whose support is a union of pairwise disjoint smooth curves, and  $p^{\ast}B_{1,3}=(3,3)$ in Pic$({\mathbb P}^{1}\times{\mathbb P}^{1})$.
As for the case of (\ref{1}), 
there is the Galois cover $q:X\rightarrow{\mathbb P}^{1}\times{\mathbb P}^{1}$ such that $X$ is a $K3$ surface, the Galois group is ${\mathbb Z}/3{\mathbb Z}$ as a group, and the branch divisor is $3p^{\ast}B_{1,3}$.
Then the branched cover $p\circ q:X\rightarrow{\mathbb P}^{1}\times{\mathbb P}^{1}$ has $3B^{1}_{1,0}+3B^{2}_{1,0}+3B_{1,3}$ as the branch divisor.
Since $X$ is simply connected, by Theorem \ref{thm:3}, $p\circ q$ is the Galois cover. 
Since the degree of $p\circ q$ is 9, by Theorem \ref{thm:5}, the Galois group of $p\circ q$ is ${\mathbb Z}/3{\mathbb Z}^{\oplus 2}$ as a group.

Conversely, for a $K3$ surface $X$ and a finite abelian subgroup $G$ of Aut$(X)$ such that $X/G\cong{\mathbb P}^{1}\times{\mathbb P}^{1}$ and the numerical class of the branch divisor $B$ of the quotient map $p:X\rightarrow X/G$ is (\ref{14}). 
By the above discussion, $G$  isomorphic to ${\mathbb Z}/3{\mathbb Z}^{\oplus 2}$ as a group, and $X\rightarrow X/G$ is given by
the composition of the Galois cover $X\rightarrow{\mathbb P}^{1}\times{\mathbb P}^{1}$ whose numerical class of the branch divisor is (\ref{1}) and the Galois cover $p:{\mathbb P}^{1}\times{\mathbb P}^{1}\rightarrow{\mathbb P}^{1}\times{\mathbb P}^{1}$ which is induced by the Galois cover ${\mathbb P}^{1}\ni [z_{0}:z_{1}]\mapsto[z^{3}_{0}:z^{3}_{1}]\in{\mathbb P}^{1}$.\\

As for the case of (\ref{14}), we get the claim for (\ref{50}). 
In this case, the Galois group is ${\mathbb Z}/3{\mathbb Z}^{\oplus 3}$ as a group. 
Furthermore, let $X$ be a $K3$ surface and $G$ be a finite abelian subgroup of Aut$(X)$ such that $X/G\cong{\mathbb P}^{1}\times{\mathbb P}^{1}$ and the numerical class of the branch divisor $B$ of $G$ is (\ref{50}).
As for the case of (\ref{14}),
$X\rightarrow X/G$ is given by
the composition of the Galois cover $X\rightarrow{\mathbb P}^{1}\times{\mathbb P}^{1}$ whose numerical class of the branch divisor is (\ref{1}) and the Galois cover $p:{\mathbb P}^{1}\times{\mathbb P}^{1}\rightarrow{\mathbb P}^{1}\times{\mathbb P}^{1}$ which is isomorphic to the Galois cover $p:{\mathbb P}^{1}\times{\mathbb P}^{1}\ni ([z_{0}:z_{1}],[w_{0}:w_{1}])\mapsto([z^{3}_{0}:z^{3}_{1}],[w^{3}_{0}:w^{3}_{1}])\in{\mathbb P}^{1}\times{\mathbb P}^{1}$.\\

For (\ref{1}), we obtain an example if we use a curve $B_{3,3}$ in $\mathbb P^{1}\times\mathbb P^{1}$ given by the equation
\[ B_{3,3}:z^{3}_{0}w_{0}^{3}+z^{3}_{0}w^{3}_{1}+z^{3}_{1}w_{0}^{3}+2z_{1}^{3}w_{1}^{3}=0.\]
For (\ref{14}), we obtain an example if we use curves $B^{1}_{1,0}, B^{2}_{1,0},B_{1,3}$ in $\mathbb P^{1}\times\mathbb P^{1}$ given by the equations
\[B^{1}_{1,0}:z_{0}=0,\ B^{2}_{1,0}:z_{1}=0,\ B_{1,3}:z_{0}w_{0}^{3}+z_{0}w^{3}_{1}+z_{1}w_{0}^{3}+2z_{1}w_{1}^{3}=0.\]
For (\ref{50}), we obtain an example if we use curves $B^{1}_{1,0},B^{2}_{1,0},B_{1,1},B^{1}_{0,1},B^{2}_{0,1}$  in $\mathbb P^{1}\times\mathbb P^{1}$ given by the equations 
\[B^{1}_{1,0}:z_{0}=0,\ B^{2}_{1,0}:z_{1}=0,\ B_{1,1}:z_{0}w_{0}+z_{0}w_{1}+z_{1}w_{0}+2z_{1}w_{1}=0, \]
\[ B^{1}_{0,1}:w_{0}=0,\ B^{2}_{0,1}:w_{1}=0. \]
\end{proof}
\begin{cro}\label{pro:9}
For each numerical classes (\ref{167},\ref{80}),  
(\ref{256},\ref{217},\ref{170},\ref{77}) of the list in section 6,
there are a $K3$ surface $X$ and a finite abelian subgroup $G$ of Aut$(X)$ such that  $X/G\cong{\mathbb F}_{n}$ and the numerical class of the branch divisor $B$ of the quotient map $p:X\rightarrow X/G$ is 
(\ref{167},\ref{80}), 
(\ref{256},\ref{217},\ref{170},\ref{77}).

Furthermore, for a pair $(X,G)$ of a $K3$ surface $X$ and a finite abelian subgroup $G$ of Aut$(X)$, if $X/G\cong{\mathbb F}_{n}$ and the numerical class of the branch divisor $B$ of the quotient map $p:X\rightarrow X/G$ is (\ref{167},\ref{80}), (\ref{256},\ref{217},\ref{170},\ref{77}), then $G$ is 
${\mathbb Z}/3{\mathbb Z}$, ${\mathbb Z}/2{\mathbb Z}\oplus{\mathbb Z}/3{\mathbb Z}$,  
${\mathbb Z}/3{\mathbb Z}$, ${\mathbb Z}/2{\mathbb Z}\oplus{\mathbb Z}/3{\mathbb Z}$, ${\mathbb Z}/3{\mathbb Z}^{\oplus 2}$, ${\mathbb Z}/2{\mathbb Z}\oplus{\mathbb Z}/3{\mathbb Z}^{\oplus 2}$, in order, as a group. 
\end{cro}
\begin{proof}
In the same way as Proposition \ref{pro:1}, we get this corollary.
More specifically,
let $X$ be a $K3$ surface, $G$ be a finite abelian subgroup of Aut$(X)$ such that $X/G\cong{\mathbb F}_{n}$, and $B$ be the branch divisor of the quotient map $p:X\rightarrow X/G$.
Then we get the following.

i) If the numerical class of $B$ is one of (\ref{167},\ref{256}), then $X\rightarrow X/G$ is given by Theorem \ref{thm:12}.

ii) If the numerical class of $B$ is (\ref{80}), then  $X\rightarrow X/G$ is given by 
the composition of the Galois cover $X'\rightarrow{\mathbb F}_{2} $ whose numerical class of the branch divisor is (\ref{167}) and the Galois cover ${\mathbb F}_{2} \rightarrow{\mathbb F}_{1}$ which is induced by the Galois cover ${\mathbb P}^{1}\rightarrow{\mathbb P}^{1}$ of degree $2$.

iii) If the numerical class of $B$ is one of (\ref{217},\ref{170},\ref{77}), then $X\rightarrow X/G$ is given by 
the composition of the Galois cover $X'\rightarrow{\mathbb F}_{6}$ whose numerical class of the branch divisor is (\ref{256}) and the Galois cover ${\mathbb F}_{6} \rightarrow{\mathbb F}_{m}$ which is induced by the Galois cover ${\mathbb P}^{1}\rightarrow{\mathbb P}^{1}$ of degree $\frac{6}{m}$.\\

For (\ref{167}), we obtain an example if we use a curve $B_{3,6}$ in $\mathbb F_{2}$ given by the equation
\[ B_{3,6}:Y^3_{0}+Y^3_{1}+Y^3_{2}=0.\]
For (\ref{80}), we obtain an example if we  use curves $B_{3,3},B^{1}_{0,1},B^{2}_{0,1}$ in $\mathbb F_{1}$ given by the equations
\[ B_{3,3}:Y^3_{0}+Y^3_{1}+Y^3_{2}=0,\ B^{1}_{0,1}:X_{0}=0,\ B^{2}_{0,1}:X_{1}=0.\]
For (\ref{256}), we obtain an example if we use a section $C$ and a curve $B_{2,12}$ in $\mathbb F_{6}$ given by the equation
\[ B_{2,12}:Y^{2}_{0}+Y^{2}_{1}+Y^{2}_{2}=0.\]
For (\ref{217}), we obtain an example if we use a section $C$ and curves $B_{2,6},B^{1}_{0,1},B^{2}_{0,1}$ in $\mathbb F_{3}$ given by the equations
\[ B_{2,6}:Y^{2}_{0}+Y^{2}_{1}+Y^{2}_{2}=0,\ B^{1}_{0,1}:X_{0}=0,\ B^{2}_{0,1}:X_{1}=0.\]
For (\ref{170}), we obtain an example if we use a section $C$ and curves $B_{2,4},B^{1}_{0,1},B^{2}_{0,1}$ in $\mathbb F_{2}$ given by the equations
\[ B_{2,4}:Y^{2}_{0}+Y^{2}_{1}+Y^{2}_{2}=0,\ B^{1}_{0,1}:X_{0}=0,\ B^{2}_{0,1}:X_{1}=0.\]
For (\ref{77}), we obtain an example if we use a section $C$ and curves $B_{2,2},B^{1}_{0,1},B^{2}_{0,1}$ in $\mathbb F_{1}$ given by the equations
\[ B_{2,2}:Y^{2}_{0}+Y^{2}_{1}+Y^{2}_{2}=0,\ B^{1}_{0,1}:X_{0}=0,\ B^{2}_{0,1}:X_{1}=0.\]
\end{proof}
\begin{pro}\label{pro:2}
For each numerical classes (\ref{2},\ref{16},\ref{53},\ref{15},\ref{51},\ref{52},\ref{33},\ref{66},\ref{63},\ref{62}) of the list in section 6, 
there are a $K3$ surface $X$ and a finite abelian subgroup $G$ of Aut$(X)$ such that  $X/G\cong{\mathbb F}_{n}$ and the numerical class of the branch divisor $B$ of the quotient map $p:X\rightarrow X/G$ is 
(\ref{2},\ref{16},\ref{53},\ref{15},\ref{51},\ref{52},\ref{33},\ref{66},\ref{63},\ref{62}).

Furthermore, for a pair $(X,G)$ of a $K3$ surface $X$ and a finite abelian subgroup $G$ of Aut$(X)$, if $X/G\cong{\mathbb F}_{n}$ and the numerical class of the branch divisor $B$ of the quotient map $p:X\rightarrow X/G$ is (\ref{2},\ref{16},\ref{53},\ref{15},\ref{51},\ref{52},\ref{33},\ref{66},\ref{63},\ref{62}), then $G$ is ${\mathbb Z}/2{\mathbb Z}$, ${\mathbb Z}/2{\mathbb Z}^{\oplus2}$, ${\mathbb Z}/2{\mathbb Z}^{\oplus3}$, ${\mathbb Z}/2{\mathbb Z}\oplus{\mathbb Z}/4{\mathbb Z}$, ${\mathbb Z}/2{\mathbb Z}\oplus{\mathbb Z}/4{\mathbb Z}^{\oplus 2}$, ${\mathbb Z}/2{\mathbb Z}^{\oplus2}\oplus{\mathbb Z}/4{\mathbb Z}$, ${\mathbb Z}/2{\mathbb Z}^{\oplus3}$, ${\mathbb Z}/2{\mathbb Z}^{\oplus 5}$, ${\mathbb Z}/2{\mathbb Z}^{\oplus 4}$, ${\mathbb Z}/2{\mathbb Z}^{3}\oplus{\mathbb Z}/4{\mathbb Z}$, in order, as a group. 
\end{pro}
\begin{proof}
In the same way as Proposition \ref{pro:1}, we get this proposition.
More specifically,
let $X$ be a $K3$ surface, $G$ be a finite abelian subgroup of Aut$(X)$ such that $X/G\cong{\mathbb F}_{n}$, and $B$ be the branch divisor of the quotient map $p:X\rightarrow X/G$. 
Then we get the following.

i) If the numerical class of $B$ is (\ref{2}), then $X\rightarrow X/G$ is given by Theorem \ref{thm:12}.

ii) If the numerical class of $B$ is one of (\ref{16},\ref{53},\ref{15},\ref{51},\ref{52},\ref{33},\ref{66},\ref{63},\ref{62}), then  $X\rightarrow X/G$ is given by 
the composition of the Galois cover $X\rightarrow{\mathbb P}^{1}\times{\mathbb P}^{1} $ whose numerical class of the branch divisor is (\ref{2}) and the Galois cover ${\mathbb P}^{1}\times{\mathbb P}^{1}\rightarrow{\mathbb P}^{1}\times{\mathbb P}^{1}$ which is induced by the Galois cover ${\mathbb P}^{1}\rightarrow{\mathbb P}^{1}$.\\

For (\ref{2}), we obtain an example if we use a curve $B_{4,4}$ in $\mathbb P^{1}\times\mathbb P^{1}$ given by the equation
\[ B_{4,4}:(z^{4}_{0}+z^{4}_{1})(w_{0}^{4}+w^{4}_{1})+2z_{0}^{2}z_{1}^{2}w_{0}^{2}w_{1}^{2}=0.\]
For (\ref{16}), we obtain an example if we use  curves $B^{1}_{1,0},B^{2}_{1,0},B_{2,4}$ in $\mathbb P^{1}\times\mathbb P^{1}$ given by the equations
\[ B^{1}_{1,0}:z_{0}=0,\ B^{2}_{1,0}:z_{1}=0,\ B_{2,4}:(z^{2}_{0}+z^{2}_{1})(w_{0}^{4}+w^{4}_{1})+2z_{0}z_{1}w_{0}^{2}w_{1}^{2}=0.\]
For (\ref{53}), we obtain an example if we use curves $B^{1}_{1,0},B^{2}_{1,0},B_{2,2},B^{1}_{0,1},B^{2}_{0,1}$ in $\mathbb P^{1}\times\mathbb P^{1}$ given by the equations
\[ B^{1}_{1,0}:z_{0}=0,\ B^{2}_{1,0}:z_{1}=0,\ B_{2,2}:(z^{2}_{0}+z^{2}_{1})(w_{0}^{2}+w^{2}_{1})+2z_{0}z_{1}w_{0}w_{1}=0,\]
\[ B^{1}_{0,1}:w_{0}=0,\ B^{2}_{0,1}:w_{1}=0.\]
For (\ref{15}), we obtain an example if we use curves $B^{1}_{1,0},B^{2}_{1,0},B_{2,4}$ in $\mathbb P^{1}\times\mathbb P^{1}$ given by the equations
\[ B^{1}_{1,0}:z_{0}=0,\ B^{2}_{1,0}:z_{1}=0,\ B_{2,4}:(z_{0}+z_{1})(w_{0}^{4}+w^{4}_{1})+(z_{0}-z_{1})w_{0}^{2}w_{1}^{2}=0.\]
For (\ref{51}), we obtain an example if we use curves $B^{1}_{1,0},B^{2}_{1,0},B_{1,1},B^{1}_{0,1},B^{2}_{0,1}$ in $\mathbb P^{1}\times\mathbb P^{1}$ given by the equations
\[ B^{1}_{1,0}:z_{0}=0,\ B^{2}_{1,0}:z_{1}=0,\ B_{1,1}:(z_{0}+z_{1})(w_{0}+w_{1})+2(z_{0}-z_{1})(w_{0}-w_{1})=0, \]
\[ B^{1}_{0,1}:w_{0}=0,\ B^{2}_{0,1}:w_{1}=0.\]
For (\ref{52}), we obtain an example if we use curves $B^{1}_{1,0},B^{2}_{1,0},B_{1,2},B^{1}_{0,1},B^{2}_{0,1}$ in $\mathbb P^{1}\times\mathbb P^{1}$ given by the equations
\[ B^{1}_{1,0}:z_{0}=0,\ B^{2}_{1,0}:z_{1}=0,\ B_{1,2}(z_{0}+z_{1})(w^{2}_{0}+w^{2}_{1})+(z_{0}-z_{1})w_{0}w_{1},\]
\[ B^{1}_{0,1}:w_{0}=0,\ B^{2}_{0,1}:w_{1}=0.\]
For (\ref{33}), we obtain an example if we use curves $B^{1}_{1,0},B^{2}_{1,0},B^{3}_{1,0},B_{1,4}$ in $\mathbb P^{1}\times\mathbb P^{1}$ given by the equations
\[ B^{1}_{1,0}:z_{0}=0,\ B^{2}_{1,0}:z_{1}=0,\ B^{3}_{1,0}:z_{0}-z_{1}=0,\] 
\[ B_{1,4}:(z_{0}+z_{1})(w^{4}_{0}+w^{4}_{1})+2(z_{0}-z_{1})(w^{4}_{0}-w^{4}_{1})=0.\]
For (\ref{66}), we obtain an example if we use curves $B^{1}_{1,0},B^{2}_{1,0},B^{3}_{1,0},B_{1,1},B^{1}_{0,1},B^{2}_{0,1},B^{3}_{0,1}$ in $\mathbb P^{1}\times\mathbb P^{1}$ given by the equations
\[ B^{1}_{1,0}:z_{0}=0,\ B^{2}_{1,0}:z_{1}=0,\ B^{3}_{1,0}:z_{0}-z_{1}=0,\] 
\[ B_{1,1}:(z_{0}-2z_{1})w_{0}+(2z_{0}+z_{1})w_{1}=0,\]
\[ B^{1}_{0,1}:w_{0}=0,\ B^{2}_{0,1}:w_{1}=0,\ B^{3}_{0,1}:w_{0}-w_{1}=0,\] 
For (\ref{63}), we obtain an example if we use curves $B^{1}_{1,0},B^{2}_{1,0},B^{3}_{1,0},B_{1,2},B^{1}_{0,1},B^{2}_{0,1}$ in $\mathbb P^{1}\times\mathbb P^{1}$ given by the equations
\[ B^{1}_{1,0}:z_{0}=0,\ B^{2}_{1,0}:z_{1}=0,\ B^{3}_{1,0}:z_{0}-z_{1}=0,\] 
\[ B_{1,2}:(z_{0}-2z_{1})w^{2}_{0}+(2z_{0}+z_{1})w^{2}_{1}=0,\ B^{1}_{0,1}:w_{0}=0,\ B^{2}_{0,1}:w_{1}=0.\]
For (\ref{62}), we obtain an example if we use curves $B^{1}_{1,0},B^{2}_{1,0},B^{3}_{1,0},B_{1,1},B^{1}_{0,1},B^{2}_{0,1}$ in $\mathbb P^{1}\times\mathbb P^{1}$ given by the equations
\[ B^{1}_{1,0}:z_{0}=0,\ B^{2}_{1,0}:z_{1}=0,\ B^{3}_{1,0}:z_{0}-z_{1}=0,\] 
\[ B_{1,2}:(z_{0}-2z_{1})w_{0}+(2z_{0}+z_{1})w_{1}=0,\ B^{1}_{0,1}:w_{0}=0,\ B^{2}_{0,1}:w_{1}=0.\]
\end{proof}
\begin{cro}\label{pro:10}
For each numerical classes (\ref{74}),
(\ref{169},\ref{81}),
(\ref{236},\ref{171},\ref{79},\ref{78}) of the list in section 6, 
there are a $K3$ surface $X$ and a finite abelian subgroup $G$ of Aut$(X)$ such that  $X/G\cong{\mathbb F}_{n}$ and the numerical class of the branch divisor $B$ of the quotient map $p:X\rightarrow X/G$ is 
(\ref{74}), (\ref{169},\ref{81}),
(\ref{236},\ref{171},\ref{79},\ref{78}).
 
Furthermore, for a pair $(X,G)$ of a $K3$ surface $X$ and a finite abelian subgroup $G$ of Aut$(X)$, if $X/G\cong{\mathbb F}_{n}$ and the numerical class of the branch divisor $B$ of the quotient map $p:X\rightarrow X/G$ is (\ref{74}),
(\ref{169},\ref{81}),
(\ref{236},\ref{171},\ref{79},\ref{78}), then $G$ is ${\mathbb Z}/2{\mathbb Z}$, 
${\mathbb Z}/2{\mathbb Z}$, ${\mathbb Z}/2{\mathbb Z}^{\oplus 2}$,
${\mathbb Z}/2{\mathbb Z}$, ${\mathbb Z}/2{\mathbb Z}^{\oplus 2}$, ${\mathbb Z}/2{\mathbb Z}\oplus {\mathbb Z}/4{\mathbb Z}$, ${\mathbb Z}/2{\mathbb Z}^{\oplus 3}$, in order, as a group. 
\end{cro}
\begin{proof}
In the same way as Proposition \ref{pro:1}, we get this corollary.
More specifically,
let $X$ be a $K3$ surface, $G$ be a finite abelian subgroup of Aut$(X)$ such that $X/G\cong{\mathbb F}_{n}$, and $B$ be the branch divisor of the quotient map $p:X\rightarrow X/G$.  
Then  we get the following.

i) If the numerical class of $B$ is one of (\ref{74},\ref{169},\ref{236}), then $X\rightarrow X/G$ is given by Theorem \ref{thm:12}.

ii) If the numerical class of $B$ is one of (\ref{81}), then  $X\rightarrow X/G$ is given by 
the composition of the Galois cover $X\rightarrow{\mathbb F}_{2}$ whose numerical class of the branch divisor is (\ref{169}) and the Galois cover ${\mathbb F}_{2}\rightarrow{\mathbb F}_{1}$ which is induced by the Galois cover ${\mathbb P}^{1}\rightarrow{\mathbb P}^{1}$ of degree 2.

iii) If the numerical class of $B$ is one of (\ref{171},\ref{79},\ref{78}), then $X\rightarrow X/G$ is given by 
the composition of the Galois cover $X\rightarrow{\mathbb F}_{4}$ whose numerical class of the branch divisor is (\ref{236})
and the Galois cover ${\mathbb F}_{4} \rightarrow{\mathbb F}_{m}$ which is induced by the Galois cover ${\mathbb P}^{1}\rightarrow{\mathbb P}^{1}$ of degree $\frac{4}{m}$.\\

For (\ref{74}), we obtain an example if we use a curve $B_{4,6}$ in $\mathbb F_{1}$ given by the equation
\[ B_{4,6}:X^{2}_{0}Y^{4}_{1}+X^{2}_{1}Y^{4}_{0}+X_{0}X_{1}Y^{4}_{2}=0.\]
For (\ref{169}), we obtain an example if we use a curve $B_{4,8}$ in $\mathbb F_{2}$ given by the equation
\[ B_{4,8}:Y^4_{0}+Y^4_{1}+Y^4_2=0.\]
For (\ref{81}), we obtain an example if we use curves $B_{4,4},B^{1}_{0,1},B^{2}_{0,1}$ in $\mathbb F_{1}$ given by the equations
\[ B_{4,4}:Y^4_0+Y^4_1+Y^4_2=0,\ B^{1}_{0,1}:X_{0}=0,\ B^{2}_{0,1}:X_{1}=0.\]
For (\ref{236}), we obtain an example if we use a section $C$ and a curve $B_{3,12}$ in $\mathbb F_{4}$ given by the equation
\[ B_{3,12}:Y^{3}_{0}+Y^{3}_{1}+Y^{3}_{2}=0.\]
For (\ref{171}), we obtain an example if we use a section $C$ and curves $B_{3,6},B^{1}_{0,1},B^{2}_{0,1}$ in $\mathbb F_{2}$ given by the equations 
\[ B_{3,6}:Y^{3}_{0}+Y^{3}_{1}+Y^{3}_{2}=0,\ B^{1}_{0,1}:X_{0}=0,\ B^{2}_{0,1}:X_{1}=0.\]
For (\ref{79}), we obtain an example if we use a section $C$ and curves $B_{3,3},B^{1}_{0,1},B^{2}_{0,1}$ in $\mathbb F_{1}$ given by the equations 
\[ B_{3,3}:Y^{3}_{0}+Y^{3}_{1}+Y^{3}_{2}=0,\ B^{1}_{0,1}:X_{0}=0,\ B^{2}_{0,1}:X_{1}=0.\]
For (\ref{78}), we obtain an example if we use a section $C$ and curves $B_{3,3},B^{1}_{0,1},B^{2}_{0,1},B^{3}_{0,1}$ in $\mathbb F_{1}$ given by the equations
\[ B_{3,3}:Y^{3}_{0}+Y^{3}_{1}+Y^{3}_{2}=0,\ B^{1}_{0,1}:X_{0}=0,\ B^{2}_{0,1}:X_{1}=0,\ B^{3}_{0,1}:X_{0}-X_{1}=0.\]
\end{proof}
\begin{pro}\label{pro:3}
For each numerical classes (\ref{3},\ref{34},\ref{64}) of the list in section 6, 
there are a $K3$ surface $X$ and a finite abelian subgroup $G$ of Aut$(X)$ such that  $X/G\cong{\mathbb F}_{n}$ and the numerical class of the branch divisor $B$ of the quotient map $p:X\rightarrow X/G$ is (\ref{3},\ref{34},\ref{64}).

Furthermore, for a pair $(X,G)$ of a $K3$ surface $X$ and a finite abelian subgroup $G$ of Aut$(X)$, if $X/G\cong{\mathbb F}_{n}$ and the numerical class of the branch divisor $B$ of the quotient map $p:X\rightarrow X/G$ is 
(\ref{3},\ref{34},\ref{64}), then
$G$ is
${\mathbb Z}/2{\mathbb Z}^{\oplus 2}$, ${\mathbb Z}/2{\mathbb Z}^{\oplus 3}$, ${\mathbb Z}/2{\mathbb Z}^{\oplus 4}$, 
in order, as a group. 	
\end{pro}
\begin{proof}
Let $B^{1}_{2,2},B^{2}_{2,2}$ be smooth curves on ${\mathbb P}^{1}\times{\mathbb P}^{1}$ such that $B^{1}_{2,2}+B^{2}_{2,2}$ is simple normal crossing. 
Then the numerical class of $2B^{1}_{2,2}+2B^{2}_{2,2}$ is (\ref{3}).
Since $B^{i}_{2,2}=(2C+2F)$ in Pic$(\mathbb P^{1}\times\mathbb P^{1})$, by Theorem \ref{thm:12}, there are the Galois covers $p_{i}:X_{i}\rightarrow {\mathbb P}^{1}\times{\mathbb P}^{1}$ such that the branch divisor of $p_{i}$ is $2B^{i}_{2,2}$ for $i=1,2$ and the Galois group of $p_{i}$ is isomorphic to ${\mathbb Z}/2{\mathbb Z}$ as a group for $i=1,2$. 
Since $B^{1}_{2,2}+B^{2}_{2,2}$ is simple normal crossing, the fibre product $X:=X_{1}\times_{{\mathbb P}^{1}\times{\mathbb P}^{1}}X_{2}$ 
of $p_{1}$ and $p_{2}$ is smooth. 
Therefore, there is the Galois cover $p:X\rightarrow{\mathbb P}^{1}\times{\mathbb P}^{1}$ such that $X$ is a $K3$ surface, the Galois group is isomorphic to ${\mathbb Z}/2{\mathbb Z}^{\oplus 2}$ as a group, and the branch divisor is $2B^{1}_{2,2}+2B^{2}_{2,2}$. 
The rest of this proposition is proved in the same way as Proposition \ref{pro:1}.
More specifically,
let $X$ be a $K3$ surface, $G$ be a finite abelian subgroup of Aut$(X)$ such that $X/G\cong{\mathbb F}_{n}$, and $B$ be the branch divisor of $G$.  Then we get the following.

i) If the numerical class of $B$ is (\ref{3}), then $X\rightarrow X/G$ is given by Theorem \ref{thm:12} and the fibre product.

ii) If the numerical class of $B$ is one of (\ref{34},\ref{64}), then  $X\rightarrow X/G$ is given by 
the composition of the Galois cover $X\rightarrow{\mathbb P}^{1}\times{\mathbb P}^{1}$ whose numerical class of the branch divisor is (\ref{3}) and the Galois cover ${\mathbb P}^{1}\times{\mathbb P}^{1}\rightarrow{\mathbb P}^{1}\times{\mathbb P}^{1}$ which is induced by the Galois cover ${\mathbb P}^{1}\rightarrow{\mathbb P}^{1}$.\\

For (\ref{3}), we obtain an example if we use curves $B^{1}_{2,2},B^{2}_{2,2}$ in $\mathbb P^{1}\times\mathbb P^{1}$ given by the equations
\[ B^{1}_{2,2}:z_{0}^{2}w_{0}^{2}+z_{1}^{2}w_{1}^{2}=0,\ B^{2}_{2,2}:z_{0}^{2}w_{1}^{2}+z_{1}^{2}w_{0}^{2}=0.\]
For (\ref{34}), we obtain an example if we use curves $B^{1}_{1,0},B^{2}_{1,0},B^{1}_{1,2},B^{2}_{1,2}$ in $\mathbb P^{1}\times\mathbb P^{1}$ given by the equations
\[ B^{1}_{1,0}:z_{0}=0,\ B^{2}_{1,0}:z_{1}=0,\ B^{1}_{1,2}:z_{0}w_{0}^{2}+z_{1}w_{1}^{2}=0,\ B^{2}_{1,2}z_{0}w_{1}^{2}+z_{1}w_{0}^{2}=0.\]
For (\ref{64}), we obtain an example if we use curves $B^{1}_{1,0},B^{2}_{1,0},B^{1}_{1,1},B^{2}_{1,1},B^{1}_{0,1},B^{2}_{0,1}$ in $\mathbb P^{1}\times\mathbb P^{1}$ given by the equations
\[ B^{1}_{1,0}:z_{0}=0,\ B^{2}_{1,0}z_{1}=0,\ B^{1}_{1,1}:(z_{0}-2z_{1})w_{0}+(2z_{0}+z_{1})w_{1}=0,\]
\[ B^{2}_{1,1}:z_{0}(w_{0}-2w_{1})+z_{1}(2w_{0}+w_{1})=0,\ B^{1}_{0,1}:w_{0}=0,\ B^{2}_{0,1}:w_{1}=0.\]
\end{proof}
\begin{cro}\label{pro:11}
For each numerical classes 
(\ref{88}),
(\ref{176},\ref{89}),
(\ref{177},\ref{93}),
(\ref{239},\ref{183},
\ref{95},\ref{96})  of the list in section 6, 
there are a $K3$ surface $X$ and a finite abelian subgroup $G$ of Aut$(X)$ such that  $X/G\cong{\mathbb F}_{n}$ and the numerical class of the branch divisor $B$ of the quotient map $p:X\rightarrow X/G$ is
(\ref{88}),
(\ref{176},\ref{89}),
(\ref{177},\ref{93}),
(\ref{239},\ref{183},\ref{95},\ref{96}).

Furthermore, for a pair $(X,G)$ of a $K3$ surface $X$ and a finite abelian subgroup $G$ of Aut$(X)$, if $X/G\cong{\mathbb F}_{n}$ and the numerical class of the branch divisor $B$ of the quotient map $p:X\rightarrow X/G$ is 
(\ref{88}),
(\ref{176},\ref{89}),
(\ref{177},\ref{93}),
(\ref{239},\ref{183},\ref{95},\ref{96}), 
then $G$ is
${\mathbb Z}/2{\mathbb Z}^{\oplus2}$, 
${\mathbb Z}/2{\mathbb Z}^{\oplus 2}$, ${\mathbb Z}/2{\mathbb Z}^{\oplus 3}$, 
${\mathbb Z}/2{\mathbb Z}^{\oplus 2}$, ${\mathbb Z}/2{\mathbb Z}^{\oplus 3}$, 
${\mathbb Z}/2{\mathbb Z}^{\oplus 2}$, ${\mathbb Z}/2{\mathbb Z}^{\oplus 3}$, ${\mathbb Z}/2{\mathbb Z}^{\oplus 2}\oplus{\mathbb Z}/4{\mathbb Z}$, ${\mathbb Z}/2{\mathbb Z}^{\oplus 4}$,
in order, as a group. 	
\end{cro}
\begin{proof}
In the same way as Proposition \ref{pro:1}, we get this corollary.
More specifically, 
let $X$ be a $K3$ surface, $G$ be a finite abelian subgroup of Aut$(X)$ such that $X/G\cong{\mathbb F}_{n}$, and $B$ be the branch divisor of the quotient map $p:X\rightarrow X/G$. 
Then we get the following.

i) If the numerical class of $B$ is one of (\ref{88},\ref{176},\ref{177},\ref{239}), then $X\rightarrow X/G$ is given by Theorem \ref{thm:12} and the fibre product.

ii) If the numerical class of $B$ is (\ref{89}), then  $X\rightarrow X/G$ is given by 
the composition of the Galois cover $X\rightarrow{\mathbb F}_{2}$ whose numerical class of the branch divisor is (\ref{176})
and the Galois cover ${\mathbb F}_{2}\rightarrow{\mathbb F}_{1}$ which is induced by the Galois cover ${\mathbb P}^{1}\rightarrow{\mathbb P}^{1}$ of degree 2.

iii) If the numerical class of $B$ is (\ref{93}), then  $X\rightarrow X/G$ is given by 
the composition of the Galois cover $X\rightarrow{\mathbb F}_{2}$ whose numerical class of the branch divisor is (\ref{177})
and the Galois cover ${\mathbb F}_{2}\rightarrow{\mathbb F}_{1}$ which is induced by the Galois cover ${\mathbb P}^{1}\rightarrow{\mathbb P}^{1}$ of degree 2.

iv) If the numerical class of $B$ is one of (\ref{183},\ref{95},\ref{96}), then  $X\rightarrow X/G$ is given by 
the composition of the Galois cover $X\rightarrow{\mathbb F}_{4}$ whose numerical class of the branch divisor is (\ref{239})
and the Galois cover ${\mathbb F}_{4}\rightarrow{\mathbb F}_{1}$ which is induced by the Galois cover ${\mathbb P}^{1}\rightarrow{\mathbb P}^{1}$ of degree 4.\\

For (\ref{88}), we obtain an example if we use curves $B_{2,4},B_{2,2}$ in $\mathbb F_{1}$ given by the equations
\[ B_{2,4}:X^{2}_{0}Y^{2}_{1}+X^{2}_{1}Y^{2}_{0}+X_{0}X_{1}Y^{2}_{2}=0,\  B_{2,2}:Y^{2}_{0}+Y^{2}_{1}+Y^{2}_{2}=0.\]
For (\ref{176}), we obtain an example if we use curves $B^{1}_{2,4},B^{2}_{2,4}$ in $\mathbb F_{2}$ given by the equations
\[ B^{1}_{2,4}:2Y^{2}_{0}+Y^{2}_{1}+Y^{2}_{2}=0,\ B^{2}_{2,4}:Y^{2}_{0}+Y^{2}_{1}+2Y^{2}_{2}=0.\]
For (\ref{89}), we obtain an example if we use curves $B^{1}_{2,2},B^{2}_{2,2},B^{1}_{0,1},B^{2}_{0,1}$ in $\mathbb F_{1}$ given by the equations
\[ B^{1}_{2,2}:2Y^{2}_{0}+Y^{2}_{1}+Y^{2}_{2}=0,\  B^{2}_{2,2}:Y^{2}_{0}+Y^{2}_{1}+2Y^{2}_{2}=0,\]
\[ B^{1}_{0,1}:X_{0}=0,\ B^{2}_{0,1}:X_{1}=0.\]
For (\ref{177}), we obtain an example if we use a section $C$ and curves $B_{1,2},B_{2,6}$ in $\mathbb F_{2}$ given by the equations
\[ B_{1,2}:Y_{0}+Y_{2}=0,\ B_{2,6}:X^{2}_{0}Y^{2}_{1}+X^{2}_{1}Y^{2}_{0}+(X^{2}_{0}+2X^{2}_{1})Y^{2}_{2}=0.\]
For (\ref{93}), we obtain an example if we use a section $C$ and curves $B_{1,1},B_{2,3},B^{1}_{0,1},B^{2}_{0,1}$ in $\mathbb F_{1}$ given by the equations
\[ B_{1,1}:Y_{0}+Y_{2}=0,\ B_{2,3}:X_{0}Y^{2}_{1}+X_{1}Y^{2}_{0}+(X_{0}+2X_{1})Y^{2}_{2}=0, \]
\[ B^{1}_{0,1}:X_{0}=0,\ B^{2}_{0,1}:X_{1}=0.\]
For (\ref{239}), we obtain an example if we use a section $C$ and curves $B_{1,4},B_{2,8}$ in $\mathbb F_{4}$ given by the equations
\[ B_{1,4}:Y_{0}+Y_{2}=0,\  B_{2,8}:Y^{2}_{0}+Y^{2}_{1}+Y^{2}_{2}=0.\]
For (\ref{183}), we obtain an example if we use a section $C$ and curves $B_{1,2},B_{2,4},B^{1}_{0,1},B^{2}_{0,1}$ in $\mathbb F_{2}$ given by the equations
\[ B_{1,2}:Y_{0}+Y_{2}=0,\  B_{2,4}:Y^{2}_{0}+Y^{2}_{1}+Y^{2}_{2}=0,\]
\[ B^{1}_{0,1}:X_{0}=0,\ B^{2}_{0,1}:X_{1}=0.\]
For (\ref{95}), we obtain an example if we use a section $C$ and curves $B_{1,1},B_{2,2},B^{1}_{0,1},B^{2}_{0,1}$ in $\mathbb F_{1}$ given by the equations
\[ B_{1,1}:Y_{0}+Y_{2}=0,\  B_{2,2}:Y^{2}_{0}+Y^{2}_{1}+Y^{2}_{2}=0,\]
\[ B^{1}_{0,1}:X_{0}=0,\ B^{2}_{0,1}:X_{1}=0.\]
For (\ref{96}), we obtain an example if we use a section $C$ and curves $B_{1,1},B_{2,2},B^{1}_{0,1},B^{2}_{0,1},B^{3}_{0,1}$ in $\mathbb F_{1}$ given by the equations
\[ B_{1,1}:Y_{0}+Y_{2}=0,\  B_{2,2}:Y^{2}_{0}+Y^{2}_{1}+Y^{2}_{2}=0,\]
\[ B^{1}_{0,1}:X_{0}=0,\ B^{2}_{0,1}:X_{1}=0,\  B^{3}:X_{0}-X_{1}=0.\]
\end{proof}
A lattice is a pair $(L,b)$ of a free abelian group $L:= Z^{\oplus n}$ of rank n and a symmetric non degenerate bilinear form $b:L\times L \rightarrow{\mathbb Z}$ taking values in ${\mathbb Z}$. 
The discriminant group of $L$ is $L^{\vee}/L$, where the dual $L^{\vee}:=\{m \in L \otimes{\mathbb Q} |\ b(m, l)\in {\mathbb Z}\ for\ all\ l \in L\}$ (here we denote by $b$ the ${\mathbb Q}$ linear extension of $b$). 
Let $U$ be the hyperbolic lattice, and $A_{n}$ and $E_{n}$ are the negative definite lattices of rank $n$ associated to the corresponding root systems.
\begin{pro}\label{pro:4}
For each classes (\ref{12}),(\ref{13}) of the list in section 6, 
there are a $K3$ surface $X$ and a finite abelian subgroup $G$ of Aut$(X)$ such that $X/G\cong{\mathbb F}_{n}$ and the numerical class of the branch divisor $B$ of the quotient map $p:X\rightarrow X/G$ is (\ref{12}),(\ref{13}).

Furthermore, for a pair $(X,G)$ of a $K3$ surface $X$ and a finite abelian subgroup $G$ of Aut$(X)$, if $X/G\cong{\mathbb F}_{n}$ and the numerical class of the branch divisor $B$ of the quotient map $p:X\rightarrow X/G$ is (\ref{12}),(\ref{13}), 
then $G$ is ${\mathbb Z}/3{\mathbb Z}^{\oplus 2}$, ${\mathbb Z}/3{\mathbb Z}^{\oplus 2}$,
in order, as a group. 	
\end{pro}
\begin{proof}
Let $B^{1}_{1,1}$, $B^{2}_{1,1}$, and $B^{3}_{1,1}$ be smooth curves such that $B^{1}_{1,1}+B^{2}_{1,1}+B^{3}_{1,1}$ is simple normal crossing.  
Since $B^{1}_{1,1}+B^{2}_{1,1}+B^{3}_{1,1}=(3C+3F)$ in Pic$({\mathbb P}^{1}\times{\mathbb P}^{1})$, by Theorem \ref{thm:12},
there is the Galois cover $p':X'\rightarrow {\mathbb P}^{1}\times{\mathbb P}^{1}$ such that the branch divisor is $3B^{1}_{1,1}+3B^{2}_{1,1}+3B^{3}_{1,1}$ and the Galois group is isomorphic to ${\mathbb Z}/3{\mathbb Z}$ as a group.
Since $B^{1}_{1,1}+B^{2}_{1,1}+B^{3}_{1,1}$ is simple normal crossing,  
singular points of $X'$ are rational double points. 
More precisely, the singular locus of $X'$ consists of six $A_2$ points.
Let  $p_{m}:X'_{m}\rightarrow X'$ be the minimal resolution of $X'$.
Then the canonical divisor of $X'_m$ is numerical trivial.
Since $X'_{m}$ has a curve such that the self intersection number is negative,
$X'_m$ is a $K3$ surface or Enriques surface.
Since $X'_{m}$ has an automorphism $s$ of order $3$ such that the curves of Fix$(s)$ are three rational curves $C_{i}$ for $i=1,2,3$, by [\ref{bio:10}],
$X'_{m}$ is a $K3$ surface.
By [\ref{bio:2}, Theorem 2.8 and Proposition 3.2] or [\ref{bio:7}, Table 2], we get that
\[ {\rm Pic}(X'_{m})^{s^{\ast}}:=\{\alpha\in{\rm Pic}(X'_{m}):\ s^{\ast}\alpha=\alpha\}\cong U\oplus E_{6} \oplus A^{3}_{2}. \]
Let $z_1,\ldots,z_6$ be singular points of $X'$, and  $e_{1},\ldots,e_{12}$ be the exceptional divisors of $p_{m}$, where $z_i=p_{m}(e_{2i-1})=p_{m}(e_{2i})$ for $i=1,\ldots,6$. 
Notice that $(e_{2i-1}\cdot e_{2i})=1$, $(e_{2i-1}\cdot e_{2i-1})=-2$, and $(e_{2i}\cdot e_{2i})=-2$.  
Since $C_i\subset$Fix$(s)$ for $i=1,2,3$, we get that $(e_{2i-1}\cup e_{2i})\cap$Fix$(s)$ contains at least 2 points.
Since $s(e_{2i-1}\cup e_{2i})=(e_{2i-1}\cup e_{2i})$ and $e_{2i-1}\cap e_{2i}$ is one point, 
we get that $e_{2i-1}\cap e_{2i}\subset$Fix$(s)$.
Therefore, $s(e_{2i-1})=e_{2i-1}$ and $s(e_{2i})=e_{2i}$, and hence $e_{2i-1},e_{2i}\in$Pic$(X'_{m})^{s^{\ast}}$ for $i=1,\ldots,6$.
Since Pic$(X'_{m})^{s^{\ast}}$ is a primitive sublattice, the minimal primitive sublattice which contains $(p'\circ p_{m})^{\ast}$Pic$({\mathbb P}^{1}\times{\mathbb P}^{1})$ and $e_{1},\ldots,e_{12}$ of Pic$(X'_{m})$ is Pic$(X'_{m})^{s^{\ast}}$. 

Let $f:=p'\circ p_{m}:X'_{m}\rightarrow{\mathbb P}^{1}\times{\mathbb P}^{1}$.
Since $f_{\ast}C_{i}=B^{i}_{1,1}$, we get 
$(C_{i}\cdot f^{\ast}F)=((C+F)\cdot F)=1$ for $i=1,2,3$.
Let 
\[ C'_{1}:=C_{1}+\sum_{i=1}^{6}\frac{(C_{1}\cdot e_{2i-1})}{2}e_{2i-1}+\sum_{i=1}^{6}\frac{(C_{1}\cdot (e_{2i-1}+2e_{2i}))}{6}(e_{2i-1}+2e_{2i}).\]
Then $(C'_{1}\cdot e_{i})=0$ for $i=1,\ldots,12$.
Since $(e_{2i-1}\cdot e_{2i-1})=-2$, $(e_{2i-1}\cdot e_{2i-1}+2e_{2i})=0$, and $(e_{2i-1}+2e_{2i}\cdot e_{2i-1}+2e_{2i})=-6$, we get 
$6C'_{1}\in$Pic$(X'_{m})$.
Therefore, the minimal primitive sublattice $K$ of Pic$(X'_{m})^{s^{\ast}}$, which contains $f^{\ast}C$ and $6C'_{1}$ is a unimodular lattice.
Let $M$ be the minimal primitive sublattice of Pic$(X'_{m})$, which contains the curves $e_{1},\ldots,e_{12}$. 
Then $M\subset U^{\perp}$.
Since $U$ is a unimodular lattice  and $M$ and $U$ are sublattice of Pic$(X'_{m})^{s^{\ast}}$, we get $U\oplus M={\rm Pic}(X'_{m})^{s^{\ast}}$.
Therefore, the rank of $M$ is 12 and $M^{\vee}/M\cong{{\mathbb Z}/3{\mathbb Z}}^{\oplus4}$. Thus, by [\ref{bio:5},Theorem 5.2] there are a $K3$ surface $X$ and a symplectic automorphism $t$ of order $3$ of $X$ such that $X'=X/\langle t\rangle$, and hence there is a finite abelian subgroup $G\subset{\rm Aut}(X)$ such that $X/G\cong{\mathbb P}^{1}\times{\mathbb P}^{1}$, $G\cong{\mathbb Z}/3{\mathbb Z}^{\oplus 2}$, and the  branch divisor is $3B^{1}_{1,1}+3B^{2}_{1,1}+3B^{3}_{1,1}$. 
In the same way, we get the claim for (\ref{13}).

More specifically,  let $X$ be a $K3$ surface $X$, $G$ be a finite abelian subgroup $G$ of Aut$(X)$ such that $X/G\cong{\mathbb P}^{1}\times {\mathbb P}^{1}$, and the numerical class of the branch divisor $B$ of $G$ is (\ref{12}) or (\ref{13}).
By Theorem \ref{thm:12}, there is the Galois cover $p':X'\rightarrow {\mathbb P}^{1}\times{\mathbb P}^{1}$ such that the branch divisor is $B$ and the Galois group is isomorphic to ${\mathbb Z}/3{\mathbb Z}$ as a group.
Then we get that $X$ is the universal cover of $X'$ of degree 3.\\

For (\ref{12}), we obtain an example if we use curves $B^{1}_{1,1}, B^{2}_{1,1},B^{3}_{1,1}$ in $\mathbb P^{1}\times\mathbb P^{1}$ given by the equations
\[ B^{1}_{1,1}:z_{0}w_{0}+z_{1}w_{1}=0,\ B^{2}_{1,1}:z_{0}w_{0}-z_{1}w_{1}=0,\ B^{3}_{1,1}:z_{0}w_{1}+z_{1}w_{0}=0. \]
For (\ref{13}), we obtain an example if we use curves $B_{1,0},B_{1,1},B_{1,2}$ in $\mathbb P^{1}\times\mathbb P^{1}$ given by the equations 
\[ B_{1,0}:z_{0}=0,\ B_{1,1}:z_{0}w_{1}+z_{1}w_{0}=0,\ B_{1,2}:z_{0}w^{2}_{1}+z_{1}w^{2}_{0}+z_{1}w^{2}_{1}=0.\]
\end{proof}
\begin{cro}\label{pro:13}
For each numerical classes 
(\ref{209},\ref{145}),
(\ref{180}),
(\ref{257},\ref{222},
\ref{182},\ref{91}) of the list in section 6, 
there are a $K3$ surface $X$ and a finite abelian subgroup $G$ of Aut$(X)$ such that  $X/G\cong{\mathbb F}_{n}$ and the numerical class $B$ of the branch divisor of the quotient map $p:X\rightarrow X/G$ is 
(\ref{209},\ref{145}),
(\ref{180}),
(\ref{257},\ref{222},\ref{182},\ref{91}).

Furthermore, for a pair $(X,G)$ of a $K3$ surface $X$ and a finite abelian subgroup $G$ of Aut$(X)$, if $X/G\cong{\mathbb F}_{n}$ and the numerical class of the branch divisor $B$ of the quotient map $p:X\rightarrow X/G$ is
(\ref{209},\ref{145}),
(\ref{180}),
(\ref{257},\ref{222},\ref{182},\ref{91}),
then $G$ is
${\mathbb Z}/3{\mathbb Z}^{\oplus 2}$, ${\mathbb Z}/2{\mathbb Z}\oplus{\mathbb Z}/3{\mathbb Z}^{\oplus 2}$, 
${\mathbb Z}/3{\mathbb Z}^{\oplus 2}$,
${\mathbb Z}/3{\mathbb Z}^{\oplus 2}$, ${\mathbb Z}/2{\mathbb Z}\oplus{\mathbb Z}/3{\mathbb Z}^{\oplus 2}$, ${\mathbb Z}/3{\mathbb Z}^{\oplus 3}$, ${\mathbb Z}/2{\mathbb Z}\oplus{\mathbb Z}/3{\mathbb Z}^{\oplus 3}$, 
in order, as a group. 	
\end{cro}
\begin{proof}
In the same way as Proposition \ref{pro:4}, we get this corollary.
More specifically, 
let $X$ be a $K3$ surface $X$, $G$ be a finite abelian subgroup $G$ of Aut$(X)$ such that $X/G\cong{\mathbb F}_{n}$, and $B$ be the branch divisor of the quotient map $p:X\rightarrow X/G$.
Let  $p':X'\rightarrow {\mathbb P}^{1}\times{\mathbb P}^{1}$ be the Galois cover such that the branch divisor is $B$ and which is given by Theorem \ref{thm:12}.
Then we get the following.

i) If the numerical class of $B$ is one of (\ref{209},\ref{180},\ref{257}), then  $X$ is the universal cover of $X'$ of degree 3.

ii) If the numerical class of $B$ is (\ref{145}), then  $X\rightarrow X/G$ is given by 
the composition of the Galois cover $X'\rightarrow{\mathbb F}_{2}$ whose numerical class of the branch divisor is (\ref{145})
and the Galois cover ${\mathbb F}_{2}\rightarrow{\mathbb F}_{1}$ which is induced by the Galois cover ${\mathbb P}^{1}\rightarrow{\mathbb P}^{1}$ of degree 2.

iii) If the numerical class of $B$ is one of (\ref{222},\ref{182},\ref{91}), then  $X\rightarrow X/G$ is given by 
the composition of the Galois cover $X'\rightarrow{\mathbb F}_{6}$ whose numerical class of the branch divisor is (\ref{257})
and the Galois cover ${\mathbb F}_{6}\rightarrow{\mathbb F}_{m}$ which is induced by the Galois cover ${\mathbb P}^{1}\rightarrow{\mathbb P}^{1}$ of degree $\frac{6}{m}$.\\

For (\ref{209}), we obtain an example if we use curves $B^{1}_{1,2},B^{2}_{1,2},B^{3}_{1,2}$ in $\mathbb F_{2}$ given by the equations
\[ B^{1}_{1,2}:Y_{0}+Y_{2}=0,\ B^{2}_{1,2}:Y_{1}+Y_{2}=0,\  B^{3}_{1,2}:Y_{0}+Y_{1}+Y_{2}=0.\]
For (\ref{145}), we obtain an example if we use curves $B^{1}_{1,1},B^{2}_{1,1},B^{3}_{1,1},B^{1}_{0,1},B^{2}_{0,1}$ in $\mathbb F_{1}$ given by the equations
\[\ B^{1}_{1,1}:Y_{0}+Y_{2}=0,\ B^{2}_{1,1}:Y_{1}+Y_{2}=0,\  B^{3}_{1,1}:Y_{0}+Y_{1}+Y_{2}=0,\]
\[  B^{1}_{0,1}:X_{0}=0,\ B^{2}_{0,1}:X_{1}=0.  \] 
For (\ref{180}), we obtain examples if we use a section $C$ and curves $B^{1}_{1,3},B^{2}_{1,3}$ in $\mathbb F_{2}$ given by the equations
\[ B^{1}_{1,3}:X_{0}Y_{0}+X_{0}Y_{1}+X_{1}Y_{2}=0,\ B^{2}_{1,3}:X_{1}Y_{0}+X_{1}Y_{1}+2X_{0}Y_{2}=0.\]
For (\ref{257}), we obtain examples if we use a section $C$ and curves $B^{1}_{1,6},B^{2}_{1,6}$ in $\mathbb F_{6}$ given by the equations
\[ B^{1}_{1,6}:Y_{0}+2Y_{2}=0,\  B^{2}_{1,6}Y_{1}+2Y_{2}=0.\]
For (\ref{222}), we obtain examples if we use a section $C$ and curves $B^{1}_{1,3},B^{2}_{1,3},B^{1}_{0,1},B^{2}_{0,1}$ in $\mathbb F_{3}$ given by the equations
\[ B^{1}_{1,3}:Y_{0}+2Y_{2}=0,\  B^{2}_{1,3}:Y_{1}+2Y_{2}=0,\]
\[ B^{1}_{0,1}:X_{0}=0,\ B^{2}_{0,1}:X_{1}=0. \]
For (\ref{182}), we obtain examples if we use a section $C$ and curves $B^{1}_{1,2},B^{2}_{1,2},B^{1}_{0,1},B^{2}_{0,1}$ in $\mathbb F_{2}$ given by the equations
\[ B^{1}_{1,2}:Y_{0}+2Y_{2}=0,\  B^{2}_{1,2}:Y_{1}+2Y_{2}=0,\]
\[ B^{1}_{0,1}:X_{0}=0,\ B^{2}_{0,1}:X_{1}=0. \]
For (\ref{91}), we obtain examples if we use a section $C$ and curves $B^{1}_{1,1},B^{2}_{1,1},B^{1}_{0,1},B^{2}_{0,1}$ in $\mathbb F_{1}$ given by the equations
\[ B^{1}_{1,1}:Y_{0}+2Y_{2}=0,\  B^{2}_{1,1}:Y_{1}+2Y_{2}=0,\]
\[ B^{1}_{0,1}:X_{0}=0,\ B^{2}_{0,1}:X_{1}=0. \]
\end{proof}
\begin{pro}\label{pro:5}
For each numerical classes (\ref{31},\ref{32}) of the list in section 6, 
there are a $K3$ surface $X$ and a finite abelian subgroup $G$ of Aut$(X)$ such that  $X/G\cong{\mathbb F}_{n}$ and the numerical class of the branch divisor $B$ of the quotient map $p:X\rightarrow X/G$ is (\ref{31},\ref{32}).

Furthermore, for a pair $(X,G)$ of a $K3$ surface $X$ and a finite abelian subgroup $G$ of Aut$(X)$, if $X/G\cong{\mathbb F}_{n}$ and the numerical class of the branch divisor $B$ of the quotient map $p:X\rightarrow X/G$ is (\ref{31},\ref{32}), 
then $G$ is ${\mathbb Z}/2{\mathbb Z}^{\oplus 3}$, ${\mathbb Z}/2{\mathbb Z}^{\oplus 3}$, 
in order, as a group.
\end{pro}
\begin{proof}
Let $B^{i}_{1,1}$ be a smooth curve on ${\mathbb P}^{1}\times{\mathbb P}^{1}$ for $i=1,2,3,4$ such that $\sum_{i=1}^{4}B^{i}_{1,1}$ is simple normal crossing.  
Then the numerical class of $\sum_{i=1}^{4}2B^{i}_{1,1}$ is (\ref{31}).
We set $\{x_{1},x_{2} \}:=B^{1}_{1,1}\cap B^{2}_{1,1}$ and  $\{x_{3},x_{4} \}:=B^{3}_{1,1}\cap B^{3}_{1,1}$. 
Let $Z:={\rm Blow}_{\{x_{1},x_{2},x_{3},x_{4}\}}{\mathbb P}^{1}\times{\mathbb P}^{1}$.  
Let $E_{i}$ be the exceptional divisor for $i=1,2,3,4$. 
Then ${\rm Pic}(Z)={\rm Pic}({\mathbb P}^{1}\times{\mathbb P}^{1})\bigoplus_{i=1}^{4}{\mathbb Z}E_{i}$. 
Let $C_{i}$ be the proper transform of $B^{i}_{1,1}$ for $i=1,2,3,4$. Then  for $i=1,2$ $j=3,4$,
\[ C_{i}=(C+F)-E_{1}-E_{2}\  {\rm and}\   C_{j}=(C+F)-E_{3}-E_{4}\ {\rm in\ Pic}(Z).\] 
By Theorem \ref{thm:12}, there are the Galois covers $p_{1}:Y_{1}\rightarrow Z$ and $p_{2}:Y_{2}\rightarrow Z$ such that 
the branch divisor of $p_{1}$ is $2C_{1}+2C_{2}$, and  that of $p_{2}$ is $2C_{3}+2C_{4}$.
Since $C_{1}\cap C_{2}$ and $C_{3}\cap C_{4}$ are empty sets, $Y_{1}$ and $Y_{2}$ are smooth.
Since $\sum_{i=1}^{4}C^{i}_{1,1}$ is simple normal crossing, $Y:=Y_{1}\times_{Z}Y_{2}$ is smooth and a $K3$ surface. 
Therefore, there is the Galois cover $f:Y\rightarrow Z$ whose branch divisor is $\sum_{i=1}^{4}2C_{i}$ and Galois group is ${\mathbb Z}/2{\mathbb Z}^{\oplus 2}$ as a group. 
Let $C'_{i}$ be a smooth curve on $Y$ such that $f^{\ast}C_{i}=2C'_{i}$ for $i=1,2,3,4$. 
Then
\[ C'_{1}=f^{\ast}((\frac{C}{2},\frac{F}{2})-\frac{1}{2}E_{1}-\frac{1}{2}E_{2})\ {\rm and}\ C'_{3}= f^{\ast}((\frac{C}{2},\frac{F}{2})-\frac{1}{2}E_{3}-\frac{1}{2}E_{4}),\ {\rm in\ Pic}(Y).\]
Thus, we get 
\[ \sum_{i=1}^{4}f^{\ast}E_{i}=2f^{\ast}(C+F)-2C'_{1}-2C'_{2},\ {\rm in\ Pic}(Y).\]
By Theorem \ref{thm:12}, there is the Galois cover $g:W\rightarrow Y$ whose branch divisor is $\sum_{i=1}^{4}2f^{\ast}E_{i}$. 
Let $E'_{i}$ be a smooth curve on $W$ such that $g^{\ast}f^{\ast}E_{i}=2E'_{i}$.
Since $(f^{\ast}E_{i}\cdot f^{\ast}E_{i})=-2$, $(E'_{i}\cdot E'_{i})=-1$ for $i=1,2,3,4$. 
Let $f:W\rightarrow X$ be a contraction of $E'_{1},\ldots,E'_{4}$.
Since $Y$ is a $K3$ surface, $X$ is a $K3$ surface. 
Since $W$ is a double cover of $Y$, there is a symplectic involution $s$ of $X$ such that  $X/\langle s \rangle\rightarrow {\mathbb P}^{1}\times{\mathbb P}^{1}$ is a Galois cover whose branch divisor is $2B^{1}_{1,1}+2B^{2}_{1,1}+2B^{3}_{1,1}+2B^{4}_{1,1}$.
Therefore, there is a finite abelian subgroup $G\subset{\rm Aut}(X)$ such that $X/G\cong{\mathbb P}^{1}\times{\mathbb P}^{1}$, $G\cong{\mathbb Z}/2{\mathbb Z}^{\oplus 3}$, and the branch divisor is $2B^{1}_{1,1}+2B^{2}_{1,1}+2B^{3}_{1,1}+2B^{4}_{1,1}$. 

Next, let $B_{1,0},B_{1,2},B^{1}_{1,1},B^{2}_{1,1}$ be smooth curves on ${\mathbb P}^{1}\times{\mathbb P}^{1}$ such that $B_{1,0}+B_{1,2}+B^{1}_{1,1}+B^{2}_{1,1}$ is simple normal crossing. 
Then the numerical class of $2B_{1,0}+2B_{1,2}+2B^{1}_{1,1}+2B^{2}_{1,1}$ is (\ref{32}). 
We set $\{x_{1},x_{2} \}:=B_{1,0}\cap B_{1,2}$ and  $\{x_{3},x_{4} \}:=B^{1}_{1,1}\cap B^{2}_{1,1}$. 
Let $Z:={\rm Blow}_{\{x_{1},x_{2},x_{3},x_{4}\}}{\mathbb P}^{1}\times{\mathbb P}^{1}$.  
Let $E_{i}$ be the exceptional divisor for $i=1,2,3,4$. 
Then ${\rm Pic}(Z)={\rm Pic}({\mathbb P}^{1}\times{\mathbb P}^{1})\bigoplus_{i=1}^{4}{\mathbb Z}E_{i}$. 
Let $C_{1,0},C_{1,2},C^{1}_{1,1},C^{2}_{1,1}$ be the proper transform of $B_{1,0}$,$B_{1,2}$,$B^{1}_{1,1}$,$B^{2}_{1,1}$ in order. 
Then 
\[ C_{1,0}=C-E_{1}-E_{2}\ {\rm and}\  C_{1,2}=(C+F)-E_{1}-E_{2}\ {\rm in\ Pic}(Z),\]
and 
\[ C^{1}_{1,1}=(C+F)-E_{3}-E_{4}\ {\rm and}\  C^{2}_{1,1}=(C+F)-E_{3}-E_{4}\ {\rm in\ Pic}(Z).\]
Let $p_{1}:Y_{1}\rightarrow Z$ be a cyclic cover whose branch divisor is $2C_{1,0}+2C_{1,2}$, and $p_{2}:Y_{2}\rightarrow Z$ be a cyclic cover whose branch divisor is $2C^{1}_{1,1}+2C^{2}_{1,1}$. 
Then as for the case of (\ref{31}), $Y:=Y_{1}\times_{Z}Y_{2}$ is a $K3$ surface, and there is the Galois cover $f:Y\rightarrow Z$ whose branch divisor is $\sum_{i=1}^{4}2C_{i}$ and Galois group is to ${\mathbb Z}/2{\mathbb Z}^{\oplus 2}$ as a group. 
Since $\frac{f^{\ast}C_{1,0}}{2}\in$Pic$(Y)$ and  $\frac{f^{\ast}C_{1,2}}{2}\in$Pic$(Y)$, 
we get   $\frac{f^{\ast}(C_{1,2}-C_{1,1})}{2}=f^{\ast}(0,\frac{1}{2})\in$Pic$(Y)$.
As for the case of (\ref{31}), we get $\frac{\sum_{i=1}^{4}f^{\ast}E_{i}}{2}\in$Pic$(Y)$, and hence we get the claim for (\ref{32}).

More specifically, 
let $X$ be a $K3$ surface $X$, $G$ be a finite abelian subgroup $G$ of Aut$(X)$ such that $X/G\cong{\mathbb P}^{1}\times {\mathbb P}^{1}$, and the numerical class of the branch divisor $B$ of the quotient map $p:X\rightarrow X/G$ is 
(\ref{31}) or (\ref{32}).
By Theorem \ref{thm:12} and the fibre product, 
there is the Galois cover $p':X'\rightarrow {\mathbb P}^{1}\times{\mathbb P}^{1}$ such that the branch divisor is $B$ and the Galois group is ${\mathbb Z}/2{\mathbb Z}^{\oplus 2}$ as a group.
Then we get that $X$ is the universal cover of $X'$ of degree 2.\\

For (\ref{31}), we obtain an example if we use curves $B^{1}_{1,1},B^{2}_{1,1},B^{3}_{1,1},B^{4}_{1,1}$ in $\mathbb P^{1}\times\mathbb P^{1}$ given by the equations
\[ B^{1}_{1,1}:z_{0}w_{0}+z_{1}w_{1}=0,\ B^{2}_{1,1}:z_{0}w_{0}-z_{1}w_{1}=0,\]
\[ B^{3}_{1,1}:z_{0}w_{1}+z_{1}w_{0}=0,\ B^{4}_{1,1}:z_{0}w_{1}-z_{1}w_{0}=0. \]
For (\ref{32}), we obtain an example if we use curves $B_{1,0},B^{1}_{1,1},B^{2}_{1,1},B_{1,2}$ in $\mathbb P^{1}\times\mathbb P^{1}$ given by the equations
\[ B_{1,0}:z_{0}=0,\ B^{1}_{1,1}:z_{0}w_{0}+z_{1}w_{1}=0,\]
\[ B^{2}_{1,1}:z_{0}w_{1}+z_{1}w_{0}=0,\ B_{1,2}:z_{0}w^{2}_{1}+3z_{1}w^{2}_{0}=0. \]
\end{proof}
\begin{cro}\label{pro:12}
For each numerical classes 
(\ref{144}),
(\ref{163}),
(\ref{()},
\ref{148}),
(\ref{216},
\ref{166}),
(\ref{250},
\ref{212},
\ref{146},
\ref{147}) of the list in section 6, 
there are a $K3$ surface $X$ and a finite abelian subgroup $G$ of Aut$(X)$ such that  $X/G\cong{\mathbb F}_{n}$ and the numerical class of the branch divisor $B$ of the quotient map $p:X\rightarrow X/G$ is
(\ref{144}),
(\ref{163}),
(\ref{()},
\ref{148}),
(\ref{216},
\ref{166}),
(\ref{250},
\ref{212},
\ref{146},
\ref{147}).

Furthermore, for a pair $(X,G)$ of a $K3$ surface $X$ and a finite abelian subgroup $G$ of Aut$(X)$, if $X/G\cong{\mathbb F}_{n}$ and the numerical class of the branch divisor $B$ of the quotient map $p:X\rightarrow X/G$ is
(\ref{144}),
(\ref{163}),
(\ref{()},
\ref{148}),
(\ref{216},
\ref{166}),
(\ref{250},
\ref{212},
\ref{146},
\ref{147}),
then $G$ is  
${\mathbb Z}/2{\mathbb Z}^{\oplus 3}$, 
${\mathbb Z}/2{\mathbb Z}^{\oplus 3}$, 
${\mathbb Z}/2{\mathbb Z}^{\oplus 3}$, ${\mathbb Z}/2{\mathbb Z}^{\oplus 4}$,  
${\mathbb Z}/2{\mathbb Z}^{\oplus 3}$, ${\mathbb Z}/2{\mathbb Z}^{\oplus 4}$,
${\mathbb Z}/2{\mathbb Z}^{\oplus 3}$, ${\mathbb Z}/2{\mathbb Z}^{\oplus 4}$, 
${\mathbb Z}/2{\mathbb Z}^{\oplus 3}\oplus{\mathbb Z}/4{\mathbb Z}$,
${\mathbb Z}/2{\mathbb Z}^{\oplus 5}$, 
in order, as a group. 	
\end{cro}
\begin{proof}
In the same way as Proposition \ref{pro:5}, we get this corollary.
More specifically, 
let $X$ be a $K3$ surface $X$, $G$ be a finite abelian subgroup $G$ of Aut$(X)$ such that $X/G\cong{\mathbb F}_{n}$, and $B$ be the branch divisor of the quotient map $p:X\rightarrow X/G$. 
Then we get following.

i) We assume that the numerical class of $B$ is one of (\ref{144},\ref{163},\ref{()},\ref{216},\ref{250}).
By Theorem \ref{thm:12} and the fibre product, there is the Galois cover $p':X'\rightarrow {\mathbb F}_{n}$ such that the branch divisor is $B$ and the Galois group is $\mathbb Z/2\mathbb Z^{\oplus 2}$ as a group.
Then $X$ is the universal cover of $X'$ of degree 2.

ii) If the numerical class of $B$ is (\ref{148}), then  $X\rightarrow X/G$ is given by 
the composition of the Galois cover $X\rightarrow{\mathbb F}_{2}$ whose numerical class of the branch divisor is (\ref{()})
and the Galois cover ${\mathbb F}_{2}\rightarrow{\mathbb F}_{1}$ which is induced by the Galois cover ${\mathbb P}^{1}\rightarrow{\mathbb P}^{1}$ of degree 2.

iii) If the numerical class of $B$ is (\ref{166}), then  $X\rightarrow X/G$ is given by 
the composition of the Galois cover $X\rightarrow{\mathbb F}_{2}$ whose numerical class of the branch divisor is (\ref{216})
and the Galois cover ${\mathbb F}_{2}\rightarrow{\mathbb F}_{1}$ which is induced by the Galois cover ${\mathbb P}^{1}\rightarrow{\mathbb P}^{1}$ of degree 2.

iv) If the numerical class of $B$ is one of (\ref{212},\ref{147},\ref{146}), then  $X\rightarrow X/G$ is given by 
the composition of the Galois cover $X\rightarrow{\mathbb F}_{4}$ whose numerical class of the branch divisor is (\ref{257})
and the Galois cover ${\mathbb F}_{4}\rightarrow{\mathbb F}_{m}$ which is induced by the Galois cover ${\mathbb P}^{1}\rightarrow{\mathbb P}^{1}$ of degree $\frac{4}{m}$.\\

For (\ref{144}), we obtain an example if we use a section $C$ and curves $B^{1}_{1,2},B^{2}_{1,2},B^{3}_{1,2}$ in $\mathbb F_{1}$ given by the equations 
\[ B^{1}_{1,2}:X_{0}Y_{1}+X_{1}Y_{0}+(X_{0}+X_{1})Y_{2}=0,\ B^{2}_{1,2}:X_{0}Y_{1}+2X_{1}Y_{0}+(2X_{0}+X_{1})Y_{2}=0, \]
\[ B^{3}_{1,2}:2X_{0}Y_{1}+X_{1}Y_{0}+(X_{0}+2X_{1})Y_{2}=0.\]
For (\ref{163}), we obtain an example if we use curves $B_{1,3},B^{1}_{1,1},B^{2}_{1,1},B^{3}_{1,1}$ in $\mathbb F_{1}$ given by the equations 
\[ B_{1,3}:X^{2}_{0}Y_{1}+X^{2}_{1}Y_{0}+X_{0}X_{1}Y_{2}=0,\ B^{1}_{1,1}:Y_{0}+Y_{1}+Y_{2}=0,\]
\[ B^{2}_{1,1}:Y_{0}+2Y_{1}+Y_{2}=0,\ B^{3}_{1,1}:2Y_{0}+Y_{1}+Y_{2}=0.\]
For (\ref{()}), we obtain an example if we use a section $C$ and curves $B_{2,4},B^{1}_{1,2},B^{2}_{1,2}$ in $\mathbb F_{2}$ given by the equations 
\[ B_{2,4}:X^{2}_{0}Y_{1}+(X^{2}_{0}+X^{2}_{1})Y_{2}=0,\ B^{1}_{1,2}:Y_{0}+Y_{2}=0,\ B^{2}_{1,2}:2Y_{0}+2Y_{1}=0. \]
For (\ref{148}), we obtain an example if we use a section $C$ and curves $B_{1,2},B^{1}_{1,1}$,
$B^{2}_{1,1}$,
$B^{1}_{0,1},B^{2}_{0,1}$ in $\mathbb F_{1}$ given by the equations  
\[  B_{1,2}:X_{0}Y_{1}+(X_{0}+X_{1})Y_{2}=0,\ B^{1}_{1,1}:Y_{0}+Y_{2}=0,\ B^{2}_{1,1}:2Y_{0}+2Y_{1}=0, \]
\[ B^{1}_{0,1}:X_{0}=0,\ B^{2}_{0,1}:X_{1}=0. \]
For (\ref{216}), we obtain an example if we use curves $B^{1}_{1,2},B^{2}_{1,2},B^{3}_{1,2},B^{4}_{1,2}$ in $\mathbb F_{2}$ given by the equations  
\[ B^{1}_{1,2}:Y_{0}+2Y_{2}=0,\ B^{2}_{1,2}:Y_{1}+2Y_{2}=0,\] 
\[ B^{3}_{1,2}:3Y_{0}+Y_{1}+Y_{2}=0,\ B^{4}_{1,2}:Y_{0}+Y_{1}+3Y_{2}=0.\]
For (\ref{166}), we obtain an example if we use curves $B^{1}_{1,1}$,$B^{2}_{1,1}$,$B^{3}_{1,1}$,$B^{4}_{1,1}$,$B^{1}_{0,1},B^{2}_{0,1}$ in $\mathbb F_{1}$ given by the equations 
\[ B^{1}_{1,1}:Y_{0}+2Y_{2}=0,\ B^{2}_{1,1}:Y_{1}+2Y_{2}=0,\]
\[ B^{3}_{1,1}:3Y_{0}+Y_{1}+Y_{2}=0,\ B^{4}_{1,1}:Y_{0}+Y_{1}+3Y_{2}=0,\]
\[ B^{1}_{0,1}:X_{0}=0,\ B^{2}_{0,1}:X_{1}=0.\]
For (\ref{250}), we obtain an example if we use a section $C$ and curves $B^{1}_{1,4}$,$B^{2}_{1,4}$,$B^{3}_{1,4}$ in $\mathbb F_{4}$ given by the equations  
\[ B^{1}_{1,4}:Y_{0}+2Y_{2}=0,\ B^{2}_{1,4}:Y_{1}+2Y_{2}=0,\ B^{3}_{1,4}:3Y_{0}+Y_{1}+Y_{2}=0.\]
For (\ref{212}), we obtain examples if we use a section $C$ and curves $B^{1}_{1,2}$,$B^{2}_{1,2}$,$B^{3}_{1,2}$,$B^{1}_{0,1},B^{2}_{0,1}$ in $\mathbb F_{2}$ given by the equations   
\[ B^{1}_{1,2}:Y_{0}+2Y_{2}=0,\ B^{2}_{1,2}:Y_{1}+2Y_{2}=0,\ B^{3}_{1,2}:3Y_{0}+Y_{1}+Y_{2}=0.\]
\[ B^{1}_{0,1}:X_{0}=0,\ B^{2}_{0,1}:X_{1}=0.\]
For (\ref{146}), we obtain examples if we use a section $C$ and curves $B^{1}_{1,1}$,$B^{2}_{1,1}$,$B^{3}_{1,1}$,$B^{1}_{0,1},B^{2}_{0,1}$ in $\mathbb F_{1}$ given by the equations   
\[ B^{1}_{1,1}:Y_{0}+2Y_{2}=0,\ B^{2}_{1,1}:Y_{1}+2Y_{2}=0,\ B^{3}_{1,1}:3Y_{0}+Y_{1}+Y_{2}=0.\]
\[ B^{1}_{0,1}:X_{0}=0,\ B^{2}_{0,1}:X_{1}=0.\]
For (\ref{147}), we obtain an example if we use a section $C$ and curves $B^{1}_{1,1}$,$B^{2}_{1,1}$,$B^{3}_{1,1}$,$B^{1}_{0,1}$,
$B^{2}_{0,1}$,$B^{3}_{0,1}$ in $\mathbb F_{1}$ given by the equations   
\[ B^{1}_{1,1:}Y_{0}+2Y_{2}=0,\ B^{2}_{1,1}:Y_{1}+2Y_{2}=0,\ B^{3}_{1,1}:3Y_{0}+Y_{1}+Y_{2}=0,\ B^{4}_{1,1}:Y_{0}+Y_{1}+3Y_{2}=0,\]
\[ B^{1}_{0,1}:X_{0}=0,\ B^{2}_{0,1}:X_{1}=0,\ B^{3}_{0,1}:X_{0}-X_{1}=0.\]
\end{proof}
\begin{pro}\label{pro:6}
For numerical classes 
(\ref{237},\ref{172},\ref{75},\ref{76}) of the list in section 6, 
there are a $K3$ surface $X$ and a finite abelian subgroup $G$ of Aut$(X)$ such that  $X/G\cong{\mathbb F}_{n}$ and the numerical class of the branch divisor $B$ of the quotient map $p:X\rightarrow X/G$ is
(\ref{237},\ref{172},\ref{75},\ref{76}).

Furthermore, for a pair $(X,G)$ of a $K3$ surface $X$ and a finite abelian subgroup $G$ of Aut$(X)$, if $X/G\cong{\mathbb F}_{n}$ and the numerical class of the branch divisor $B$ of the quotient map $p:X\rightarrow X/G$ is (\ref{237},\ref{172},\ref{75},\ref{76}), 
then $G$ is 
${\mathbb Z}/4{\mathbb Z}$, ${\mathbb Z}/2{\mathbb Z}\oplus{\mathbb Z}/4{\mathbb Z}$, 
${\mathbb Z}/4{\mathbb Z}^{\oplus 2}$,
${\mathbb Z}/2{\mathbb Z}^{2}\oplus{\mathbb Z}/4{\mathbb Z}$,
in order, as a group.	
\end{pro}
\begin{proof}
Let $B_{2,8}$ be a smooth curve on ${\mathbb F}_{4}$. 
Then the numerical class of $2C+4B_{2,8}$ is (\ref{237}). 
Since $B_{2,8}=2C+8F$ in Pic$(\mathbb F_{4})$, by Theorem \ref{thm:12},
there is the Galois cover $p_{1}:X_{1}\rightarrow{\mathbb F}_{4}$ such that the branch divisor is $2B_{2,8}$ and the Galois group is $\mathbb Z/2\mathbb Z$ as a group. 
Let $E_{2,8}$ be a smooth curve on $X_{1}$ such that $p_{1}^{\ast}B_{2,8}=2E_{2,8}$. 
Since $C+B_{2,8}$ is simple normal crossing, $p_{1}^{\ast}C$ is a reduced divisor on $X_{1}$, whose support is a union of pairwise disjoint smooth curves.
Since $p_{1}^{\ast}C+E_{2,8}=p_{1}^{\ast}(2C+4F)=2p_{1}^{\ast}(C+2F)$ in Pic$(X_{1})$, by Theorem \ref{thm:12},
there is a Galois cover $p_{2}:X_{2}\rightarrow X_{1}$ such that the branch divisor is $p_{1}^{\ast}C+E_{2,8}$ and the Galois group is $\mathbb Z/2\mathbb Z$ as a group. 
Then $p:=p_{1}\circ p_{2}:X_{2}\rightarrow{\mathbb F}_{4}$ is the branched cover such that $p$ has $2C+4B_{2,8}$ as the branch divisor.
In the same way of Proposition \ref{pro:1}, $X$ is a $K3$ surface, and $p:X\rightarrow {\mathbb F}_{4}$ is the Galois cover whose Galois group is ${\mathbb Z}/4{\mathbb Z}$ as a group.
In the same way of Proposition \ref{pro:1}, we get the claim for (\ref{172},\ref{75},\ref{76}).

More specifically, 
let $X$ be a $K3$ surface, $G$ be a finite abelian subgroup of Aut$(X)$ such that $X/G\cong{\mathbb F}_{n}$, and $B$ be the branch divisor of the quotient map $p:X\rightarrow X/G$. 
Then we get the following.

i) If the numerical class of $B$ is (\ref{237}), then $X\rightarrow X/G$ is given by the above way.

ii) If the numerical class of $B$ is one of (\ref{172},\ref{75},\ref{76}), then  $X\rightarrow X/G$ is given by 
the composition of the Galois cover $X'\rightarrow{\mathbb F}_{4} $ whose numerical class of the branch divisor is (\ref{237}) and the Galois cover ${\mathbb F}_{4} \rightarrow{\mathbb F}_{m}$ which is induced by the Galois cover ${\mathbb P}^{1}\rightarrow{\mathbb P}^{1}$ of degree $\frac{m}{4}$.\\

For (\ref{237}), we obtain an example if we use a section $C$ and a curve $B_{2,8}$ in $\mathbb F_{4}$ given by the equation  
\[ B_{2,8}:Y^{2}_{0}+Y^{2}_{1}+Y^{2}_{2}=0. \]
For (\ref{172}), we obtain examples if we use a section $C$ and curves $B_{2,4}$,$B^{1}_{0,1}$,$B^{2}_{0,1}$ in $\mathbb F_{2}$ given by the equations  
\[ B_{2,4}:Y^{2}_{0}+Y^{2}_{1}+Y^{2}_{2}=0,\ B^{1}_{0,1}:X_{0}=0,\ B^{2}_{0,1}:X_{1}=0. \]
For (\ref{75}), we obtain examples if we use a section $C$ and curves $B_{2,2}$,$B^{1}_{0,1}$,$B^{2}_{0,1}$ in $\mathbb F_{1}$ given by the equations  
\[ B_{2,2}:Y^{2}_{0}+Y^{2}_{1}+Y^{2}_{2}=0,\ B^{1}_{0,1}:X_{0}=0,\ B^{2}_{0,1}:X_{1}=0. \]
For (\ref{76}), we obtain examples if we use a section $C$ and curves $B_{2,2}$,$B^{1}_{0,1}$,
$B^{2}_{0,1}$,$B^{3}_{0,1}$ in $\mathbb F_{1}$ given by the equations  
\[ B_{2,2}:Y^{2}_{0}+Y^{2}_{1}+Y^{2}_{2}=0,\ B^{1}_{0,1}:X_{0}=0,\ B^{2}_{0,1}:X_{1}=0,\ B^{3}_{0,1}:X_{0}-X_{1}=0. \]
\end{proof}
\begin{pro}\label{pro:7}
For numerical classes (\ref{245},\ref{208}) of the list in section 6, 
there are a $K3$ surface $X$ and a finite abelian subgroup $G$ of Aut$(X)$ such that  $X/G\cong{\mathbb F}_{n}$ and the numerical class of the branch divisor $B$ of the quotient map $p:X\rightarrow X/G$ is
(\ref{245},\ref{208}).

Furthermore, for a pair $(X,G)$ of a $K3$ surface $X$ and a finite abelian subgroup $G$ of Aut$(X)$, if $X/G\cong{\mathbb F}_{n}$ and the numerical class of the branch divisor $B$ of the quotient map $p:X\rightarrow X/G$ is (\ref{245},\ref{208}), 
then $G$ is ${\mathbb Z}/2{\mathbb Z}\oplus{\mathbb Z}/4{\mathbb Z}$, 
${\mathbb Z}/2{\mathbb Z}^{\oplus 2}\oplus{\mathbb Z}/4{\mathbb Z}$,
in order, as a group.
\end{pro}
\begin{proof}
Let $B_{1,6}$ and $B_{1,4}$ be smooth curves on ${\mathbb F}_{4}$ such that $C+B_{1,6}+B_{1,4}$ is simple normal crossing.
Then the numerical class of $4C+2B_{1,6}+4B_{1,4}$ is (\ref{245}). 
Since $C+B_{1,4}=2C+2F$ in Pic$(\mathbb F_{8})$, by Theorem \ref{thm:12},
there is the Galois cover $p_{1}:X_{1}\rightarrow{\mathbb F}_{4}$ such that the branch divisor is $2C+2B_{1,4}$ and the Galois group is $\mathbb Z/2\mathbb Z$ as a group. 
Let $E_{C},E_{1,4}$ be two smooth curves on $X_{1}$ such that $p_{1}^{\ast}C=2E_{C}$ and $p_{1}^{\ast}B_{1,4}=2E_{1,4}$. 
Since $C+B_{1,6}+B_{1,4}$ is simple normal crossing, $p_{1}^{\ast}B_{1,6}$ is a reduced divisor on $X_{1}$, whose support is a union of pairwise disjoint smooth curves.
Since $p_{1}^{\ast}B_{1,6}=p_{1}^{\ast}(C+6F)=p_{1}^{\ast}(C+4F)+p_{1}^{\ast}(2F)=2E_{1,4}+2p_{1}^{\ast}F$ in Pic$(X_{1})$, by Theorem \ref{thm:12}, 
there is the Galois cover $p_{2}:X_{2}\rightarrow X_{1}$ such that the branch divisor is $2p_{1}^{\ast}B_{1,6}$ and the Galois group is $\mathbb Z/2\mathbb Z$. 
Notice taht $\frac{p_{2}^{\ast}p_{1}^{\ast}B_{1,6}}{2}\in$Pic$(X_{2})$.
Since $C+B_{1,6}+B_{1,4}$ is simple normal crossing, 
$p_{2}^{\ast}E_{C}$ and $p_{2}^{\ast}E_{1,4}$ are reduced divisors on $X_{2}$, whose support are unions of pairwise disjoint smooth curves.
Since $p_{2}^{\ast}(E_{C}+E_{1,4})=p_{2}^{\ast}p_{1}^{\ast}(C+2F)=p_{2}^{\ast}p_{1}^{\ast}(C+6F)-p_{2}^{\ast}p_{1}^{\ast}4F= p_{2}^{\ast}p_{1}^{\ast}B_{1,6}-4p_{2}^{\ast}p_{1}^{\ast}F$ in Pic$(X_{2})$ and $\frac{p_{2}^{\ast}p_{1}^{\ast}B_{1,6}}{2}\in$Pic$(X_{2})$, by Theorem \ref{thm:12},
there is the Galois cover $p_{3}:X\rightarrow X_{2}$ such that the branch divisor is $p_{2}^{\ast}(E_{C}+E_{1,4})$ and the Galois group is $\mathbb Z/2\mathbb Z$. 
Then $p:=p_{1}\circ p_{2}\circ p_{3}:X\rightarrow{\mathbb F}_{4}$ is the branched cover such that $p$ has $4C+2B_{1,6}+4B_{1,4}$ as 
the branch divisor.
In the same way of Proposition \ref{pro:1}, $X$ is a $K3$ surface, and $p:X\rightarrow {\mathbb F}_{4}$ is the Galois cover whose Galois group is ${\mathbb Z}/2{\mathbb Z}\oplus{\mathbb Z}/4{\mathbb Z}$ as a group.
In the same way of Proposition \ref{pro:1}, we get the claim for (\ref{208}).

More specifically, 
let $X$ be a $K3$ surface, $G$ be a finite abelian subgroup of Aut$(X)$ such that $X/G\cong{\mathbb F}_{n}$, and $B$ be the branch divisor of the quotient map $p:X\rightarrow X/G$. 
Then we get the following.

i) If the numerical class of $B$ is (\ref{245}), then $X\rightarrow X/G$ is given by the above way.

ii) If the numerical class of $B$ is (\ref{208}), then $X\rightarrow X/G$ is given by 
the composition of the Galois cover $X'\rightarrow{\mathbb F}_{4} $ whose numerical class of the branch divisor is (\ref{245}) and the Galois cover ${\mathbb F}_{4} \rightarrow{\mathbb F}_{2}$ which is induced by the Galois cover ${\mathbb P}^{1}\rightarrow{\mathbb P}^{1}$ of degree $2$.\\

For (\ref{245}), we obtain an example if we use a section $C$ and curves $B_{1,6}$,$B_{1,4}$ in $\mathbb F_{4}$ given by the equations  
\[ B_{1,6}:X^{2}_{0}Y_{1}+X^{2}_{1}Y_{0}+(X^{2}_{0}+2X^{2}_{1})Y_{2}=0,\ B_{1,4}:2Y_{0}+Y_{2}=0. \]
For (\ref{208}), we obtain an example if we use a section $C$ and curves $B_{1,3}$,$B_{1,2}$,
$B^{1}_{0,1},B^{2}_{0,1}$ in $\mathbb F_{2}$ given by the equations  
\[ B_{1,3}:X_{0}Y_{1}+X_{1}Y_{0}+(X_{0}+2X_{1})Y_{2}=0,\ B_{1,2}:2Y_{0}+Y_{2}=0, \]
\[ B^{1}_{0,1}:X_{0}=0,\ B^{2}_{0,1}:X_{1}=0. \]
\end{proof}
\begin{cro}\label{thm:48}
For each numerical classes 
(\ref{263},\ref{249},\ref{207},\ref{206},\ref{141}) of the list in section 6, 
there are a $K3$ surface $X$ and a finite abelian subgroup $G$ of Aut$(X)$ such that  $X/G\cong{\mathbb F}_{n}$ and the numerical class of the branch divisor $B$ of the quotient map $p:X\rightarrow X/G$ is 
(\ref{263},\ref{249},\ref{207},\ref{206},\ref{141}).
	
Furthermore, for a pair $(X,G)$ of a $K3$ surface $X$ and a finite abelian subgroup $G$ of Aut$(X)$, if $X/G\cong{\mathbb F}_{n}$ and the numerical class of the branch divisor $B$ of the quotient map $p:X\rightarrow X/G$ is
(\ref{263}\ref{249},\ref{207},\ref{206},\ref{141}),
then $G$ is 
${\mathbb Z}/2{\mathbb Z}\oplus{\mathbb Z}/4{\mathbb Z}$, 
${\mathbb Z}/2{\mathbb Z}^{\oplus 2}\oplus{\mathbb Z}/4{\mathbb Z}$, 
${\mathbb Z}/2{\mathbb Z}^{\oplus 3}\oplus{\mathbb Z}/4{\mathbb Z}$, 
${\mathbb Z}/2{\mathbb Z}\oplus{\mathbb Z}/4{\mathbb Z}^{\oplus 2}$,
${\mathbb Z}/2{\mathbb Z}\oplus{\mathbb Z}/4{\mathbb Z}\oplus{\mathbb Z}/8{\mathbb Z}$, 
in order, as a group. 	
\end{cro}
\begin{proof}
In the same way Proposition \ref{pro:7}, we get the claim. 
More specifically, 
let $X$ be a $K3$ surface, $G$ be a finite abelian subgroup of Aut$(X)$ such that $X/G\cong{\mathbb F}_{n}$, and $B$ be the branch divisor of the quotient map $p:X\rightarrow X/G$. 
Then we get the following.

i) If the numerical class of $B$ is (\ref{263}), then $X\rightarrow X/G$ is given by the above way.

ii) If the numerical class of $B$ is one of (\ref{141},\ref{206},\ref{207},\ref{249}), then  $X\rightarrow X/G$ is given by 
the composition of the Galois cover $X\rightarrow{\mathbb F}_{8} $ whose numerical class of the branch divisor is (\ref{263}) and the Galois cover ${\mathbb F}_{8} \rightarrow{\mathbb F}_{m}$ which is induced by the Galois cover ${\mathbb P}^{1}\rightarrow{\mathbb P}^{1}$ of degree $\frac{8}{m}$.\\

For (\ref{263}),  we obtain examples if we use a section $C$ and curves $B^{1}_{1,8}$,$B^{2}_{1,8}$ in $\mathbb F_{8}$ given by the equations  
\[  B^{1}_{1,8}:Y_{0}+Y_{1}+Y_{2}=0,\  B^{2}_{1,8}:Y_{0}+Y_{1}+2Y_{2}=0.\]
For (\ref{249}),  we obtain examples if we use a section $C$ and 
curves $B^{1}_{1,4}$,$B^{2}_{1,4}$,
$B^{1}_{0,1},B^{2}_{0,1}$ in $\mathbb F_{4}$ given by the equations  
\[ B^{1}_{1,4}:Y_{0}+Y_{1}+Y_{2}=0,\ B^{2}_{1,4}:Y_{0}+Y_{1}+2Y_{2}=0,\]
\[ B^{1}_{0,1}:X_{0}=0,\ B^{2}:X_{1}=0.\]
For (\ref{206}),  we obtain examples if we use a section $C$ and 
curves $B^{1}_{1,2}$,$B^{2}_{1,2}$,
$B^{1}_{0,1},B^{2}_{0,1}$ in $\mathbb F_{2}$ given by the equations  
\[ B^{1}_{1,2}:Y_{0}+Y_{1}+Y_{2}=0,\ B^{2}_{1,2}:Y_{0}+Y_{1}+2Y_{2}=0,\]
\[ B^{1}_{0,1}:X_{0}=0,\ B^{2}:X_{1}=0.\]
For (\ref{141}),  we obtain examples if we use a section $C$ and 
curves $B^{1}_{1,1}$,$B^{2}_{1,1}$,
$B^{1}_{0,1},B^{2}_{0,1}$ in $\mathbb F_{1}$ given by the equations  
\[ B^{1}_{1,1}:Y_{0}+Y_{1}+Y_{2}=0,\ B^{2}_{1,1}:Y_{0}+Y_{1}+2Y_{2}=0,\]
\[ B^{1}_{0,1}:X_{0}=0,\ B^{2}:X_{1}=0.\]
For (\ref{207}), we obtain an example if we use a section $C$ and curves $B^{1}_{1,2}$,$B^{2}_{1,2}$,
$B^{1}_{0,1}$,$B^{2}_{0,1}$,$B^{3}_{0,1}$ in $\mathbb F_{2}$ given by the equations  
\[ B^{1}_{1,2}:Y_{0}+Y_{1}+Y_{2}=0,\ B^{2}_{1,2}:Y_{0}+Y_{1}+2Y_{2}=0,\]
\[ B^{1}_{0,1}:X_{0}=0,\ B^{2}_{0,1}:X_{1}=0,\ B^{3}_{0,1}:X_{0}-X_{1}=0.\]
\end{proof}
\begin{pro}\label{pro:8}
For each numerical classes 
(\ref{266},\ref{258},\ref{240},\ref{221},\ref{220},\ref{181},\ref{90}) of the list in section 6, 
there are a $K3$ surface $X$ and a finite abelian subgroup $G$ of Aut$(X)$ such that  $X/G\cong{\mathbb F}_{n}$ and the numerical class of the branch divisor $B$ of the quotient map $p:X\rightarrow X/G$ is
(\ref{266},\ref{258},\ref{240},\ref{221},\ref{220},\ref{181},\ref{90}).

Furthermore, for a pair $(X,G)$ of a $K3$ surface $X$ and a finite abelian subgroup $G$ of Aut$(X)$, if $X/G\cong{\mathbb F}_{n}$ and the numerical class of the branch divisor $B$ of the quotient map $p:X\rightarrow X/G$ is 
(\ref{266},\ref{258},\ref{240},\ref{221},\ref{220},\ref{181},\ref{90}), 
then $G$ is 
${\mathbb Z}/2{\mathbb Z}\oplus {\mathbb Z}/3{\mathbb Z}$, 
${\mathbb Z}/2{\mathbb Z}^{\oplus 2}\oplus {\mathbb Z}/3{\mathbb Z}$,
${\mathbb Z}/2{\mathbb Z}\oplus{\mathbb Z}/3{\mathbb Z}^{\oplus 2}$, 
${\mathbb Z}/2{\mathbb Z}^{\oplus 3}\oplus{\mathbb Z}/3{\mathbb Z}$, 
${\mathbb Z}/2{\mathbb Z}\oplus{\mathbb Z}/3{\mathbb Z}\oplus {\mathbb Z}/4{\mathbb Z}$, 
${\mathbb Z}/2{\mathbb Z}^{\oplus 2}\oplus{\mathbb Z}/3{\mathbb Z}^{\oplus 2}$, 
${\mathbb Z}/2{\mathbb Z}\oplus {\mathbb Z}/3{\mathbb Z}^{\oplus 2}\oplus{\mathbb Z}/4{\mathbb Z}$,
in order, as a group.
\end{pro}
\begin{proof}
Let $B^{i}_{1,12}$ be a smooth curve on ${\mathbb F}_{12}$ for $i=1,2$ such that  $C+B^{1}_{1,12}+B^{2}_{1,12}$ is simple normal crossing. 
Then the numerical class of $6C+2B^{1}_{1,12}+3B^{2}_{1,12}$ is (\ref{266}). 
Since $C+B^{1}_{1,12}=2C+12F$ in Pic$(\mathbb F_{12})$, by Theorem \ref{thm:12}, 
there is the Galois cover $p_{1}:X_{1}\rightarrow{\mathbb F}_{12}$ such that the branch divisor is $2C+2B^{1}_{1,12}$ and the Galois group is ${\mathbb Z}/2\mathbb Z$ as a group. 
Since $C+B^{1}_{1,12}+B^{2}_{1,12}$ is simple normal crossing, 
$p_{1}^{\ast}B^{2}_{1,12}$ is a reduced divisor on $X_{1}$, whose support is a union of pairwise disjoint smooth curves.
Since $C$ and $B^{1}_{1,12}$ are smooth curves,
there are smooth curves $E_{C}$, $E^{1}_{1,12}$ on $X_{1}$ such that $p_{1}^{\ast}C=2E_{C}$ and $p_{1}^{\ast}B^{1}_{1,12}=2E^{1}_{1,12}$. 
Since $E_{C}+p_{1}^{\ast}B^{2}_{1,12}=E_{C}+p_{1}^{\ast}(C+12F)=E_{C}+p_{1}^{\ast}C+12p_{1}^{\ast}F=3E_{C}+12p_{1}^{\ast}F$ in Pic$(X_{1})$,
by Theorem \ref{thm:12}, there is the Galois cover $p_{2}:X\rightarrow X_{1}$ such that the branch divisor is $3E_{C}+3p_{1}^{\ast}B^{2}_{1,12}$ and the Galois group is $\mathbb Z/3\mathbb Z$ as a group.
Then $p:=p_{1}\circ p_{2}:X\rightarrow{\mathbb F}_{12}$ is the branched cover such that $p$ has $6C+2B^{1}_{1,12}+3B^{2}_{1,12}$ as 
the branch divisor.
In the same way as Proposition \ref{pro:1}, $X$ is a $K3$ surface, and $p:X\rightarrow{\mathbb F}_{12}$ is the Galois cover whose Galois group is $G\cong{\mathbb Z}/2{\mathbb Z}\oplus {\mathbb Z}/3{\mathbb Z}$ as a group.

More specifically, 
let $X$ be a $K3$ surface, $G$ be a finite abelian subgroup of Aut$(X)$ such that $X/G\cong{\mathbb F}_{n}$, and $B$ be the branch divisor of the quotient map $p:X\rightarrow X/G$. 
Then we get the following.

i) If the numerical class of $B$ is (\ref{266}), then $X\rightarrow X/G$ is given by the above way.

ii) If the numerical class of $B$ is one of (\ref{258},\ref{240},\ref{221},\ref{220},\ref{181},\ref{90}), 
then $X\rightarrow X/G$ is  given by 
the composition of the Galois cover $X'\rightarrow{\mathbb F}_{12} $ whose numerical class of the branch divisor is (\ref{266}) and the Galois cover ${\mathbb F}_{12} \rightarrow{\mathbb F}_{m}$ which is induced by the Galois cover ${\mathbb P}^{1}\rightarrow{\mathbb P}^{1}$ of degree $\frac{12}{m}$.\\

For (\ref{266}), we obtain an example if we use a section $C$ and curves $B^{1}_{1,12}$,$B^{2}_{1,12}$
in $\mathbb F_{12}$ given by the equations  
\[ B^{1}_{1,12}:Y_{0}+2Y_{2}=0,\ B^{2}_{1,12}:Y_{1}+2Y_{2}=0. \]
For (\ref{258}), we obtain examples if we use a section $C$ and curves $B^{1}_{1,6}$,$B^{2}_{1,6}$,
$B^{1}_{0,1},B^{2}_{0,1}$ in $\mathbb F_{6}$ given by the equations  
\[ B^{1}_{1,6}:Y_{0}+2Y_{2}=0,\ B^{2}_{1,6}:Y_{1}+2Y_{2}=0, \]
\[  B^{1}_{0,1}:X_{0}=0,\ B^{2}_{0,1}:X_{1}=0. \]
For (\ref{240}), we obtain examples if we use a section $C$ and curves $B^{1}_{1,4}$,$B^{2}_{1,4}$,
$B^{1}_{0,1},B^{2}_{0,1}$ in $\mathbb F_{4}$ given by the equations  
\[ B^{1}_{1,4}:Y_{0}+2Y_{2}=0,\ B^{2}_{1,4}:Y_{1}+2Y_{2}=0, \]
\[  B^{1}_{0,1}:X_{0}=0,\ B^{2}_{0,1}:X_{1}=0. \]
For (\ref{220}), we obtain examples if we use a section $C$ and curves $B^{1}_{1,3}$,$B^{2}_{1,3}$,
$B^{1}_{0,1},B^{2}_{0,1}$ in $\mathbb F_{3}$ given by the equations  
\[ B^{1}_{1,3}:Y_{0}+2Y_{2}=0,\ B^{2}_{1,3}:Y_{1}+2Y_{2}=0, \]
\[  B^{1}_{0,1}:X_{0}=0,\ B^{2}_{0,1}:X_{1}=0. \]
For (\ref{181}), we obtain examples if we use a section $C$ and curves $B^{1}_{1,2}$,$B^{2}_{1,2}$,
$B^{1}_{0,1},B^{2}_{0,1}$ in $\mathbb F_{2}$ given by the equations  
\[ B^{1}_{1,2}:Y_{0}+2Y_{2}=0,\ B^{2}_{1,2}:Y_{1}+2Y_{2}=0, \]
\[  B^{1}_{0,1}:X_{0}=0,\ B^{2}_{0,1}:X_{1}=0. \]
For (\ref{90}), we obtain examples if we use a section $C$ and curves $B^{1}_{1,1}$,$B^{2}_{1,1}$,
$B^{1}_{0,1},B^{2}_{0,1}$ in $\mathbb F_{1}$ given by the equations  
\[ B^{1}_{1,1}:Y_{0}+2Y_{2}=0,\ B^{2}_{1,1}:Y_{1}+2Y_{2}=0, \]
\[  B^{1}_{0,1}:X_{0}=0,\ B^{2}_{0,1}:X_{1}=0. \]
For (\ref{221}), we obtain an example if we use a section $C$ and curves $B^{1}_{1,3}$,$B^{2}_{1,3}$,
$B^{1}_{0,1},B^{2}_{0,1},B^{3}_{0,1}$ in $\mathbb F_{3}$ given by the equations 
\[ B^{1}_{1,3}:Y_{0}+2Y_{2}=0,\ B^{2}_{1,3}:Y_{1}+2Y_{2}=0, \]
\[  B^{1}_{0,1}:X_{0}=0,\ B^{2}_{0,1}:X_{1}=0,\ B^{3}_{0,1}:X_{0}-X_{1}=0. \]
\end{proof}
\subsection{Complete proof of Theorem \ref{thm:2}} 
In this section,  we will show that there is no numerical class such that it has an abelian $K3$ cover except the numerical classes which are mentioned in Section 3.1. 
Then by Section 3.1, we will get Theorem \ref{thm:2}. 
From here, we use the notations that\\
i) $X$ is a $K3$ surface,\\
ii) $G$ is a finite abelian subgroup of Aut$(X)$ such that $X/G\cong{\mathbb F}_{n}$,\\
iii) $p:X\rightarrow X/G$ is the quotient map, and\\
iv) $B:=\sum_{i=1}^{k}b_{i}B_{i}$ is the branch divisor of $p$.\\ 
Furthermore, 
we use the notation that  $B^{k}_{i,j}$ (or simply $B_{i,j}$) is a smooth curve on ${\mathbb F}_{n}$ such that 
$B^{k}_{i,j}=iC+jF$ in ${\rm Pic}({\mathbb F}_{n})$ if $n\geq0$ where $k\in{\mathbb N}$.\\
For the branch divisor $B=\sum_{i=1}^{m}\sum_{j=1}^{n(i)}b^{i}_{j}B^{j}_{s_{i},t_{i}}$ where $m,n(i)\in{\mathbb N}$, 
we use the notation that 
\[ G^{j}_{s_{i},t_{i}}:=\{g\in G:g_{|p^{-1}(B^{j}_{s_{i},t_{i}})}={\rm id}_{p^{-1}(B^{j}_{s_{i},t_{i}})}\}.\]
Recall that by Theorem \ref{thm:5}, $G^{j}_{s_{i},t_{i}}$ is a cyclic group of order $b^{i}_{j}$ which is generated by a non-symplectic automorphism of order $b^{i}_{j}$.
Since $G$ is abelian, the support of $B$ and  the support of $p^{\ast}B$ are simple normal crossing.
\begin{lem}\label{thm:11}
We assume that $X/G\cong{\mathbb F}_{n}$ for $n\geq 1$.
If $B=aC+\sum_{i=1}^{k}b_{i}(c_{i}C+nc_{i}F)+\sum_{j=1}^{l}d_{j}F_{j}$ in Pic$({\mathbb F}_{n})$ where $a,b_{i},d_{j}\geq2$ and $c_{i},l\geq1$, then $3\geq l\geq 2$ and $d_{1}=\cdots=d_{l}$.
\end{lem}
\begin{proof}
By Theorem \ref{thm:5}, there are pairwise disjoint smooth curves $C_{1},\ldots,C_{m}$ such that $p^{\ast}C=\sum_{i=1}^{m}aC_{i}$. 
Since $C_{1},\ldots,C_{m}$ are pairwise disjoint, we get that $(\sum_{i=1}^{m}C_{i}\cdot \sum_{i=1}^{m}C_{i})=\sum_{i=1}^{m}(C_{i}\cdot C_{i})=m(C_{i}\cdot C_{i})$ for $i=1,\ldots,m$.
Since $(C\cdot C)=-n<0$, $(C_{i}\cdot C_{i})<0$ for $i=1,\ldots,m$.
Since $X$ is a $K3$ surface, $C_{i}$ is a smooth rational curve for $i=1,\ldots,m$.
Let $p_{|C_{i}}:C_{i}\rightarrow C$ be the finite map. 
Let $B_{c_{i},nc_{i}}$ be an irreducible curve on ${\mathbb F}_{n}$.
Since $B_{c_{i},nc_{i}}=c_{i}C+nc_{i}F$ in Pic$({\mathbb F}_{n})$, we get that $C\cap B_{c_{i},nc_{i}}$ is an empty set.
Since the support of $B$ is simple normal crossing , $p_{|C_{i}}$ is the Galois covering whose branch divisor is $\sum_{j=1}^{l}d_{j}(C\cap F_{j})$. 
If $d_{i}\not =d_{j}$, then $p_{|C_{i}}$ must be non-trivial.
Since $G$ is an abelian group, $p_{|C_{i}}$ is the Abelian cover, however by Theorem \ref{thm:10}, this is a non-Abelian cover. This is a contradiction. Therefore, $d_{1}=\cdots=d_{l}$.
\end{proof}
By Lemma \ref{thm:11}, the numerical class of $B$ is not one of 
(\ref{268},\ref{269},\ref{267},\ref{270},\ref{279},\ref{271},\\
\ref{272},\ref{283},\ref{284},\ref{281},\ref{280},\ref{282},\ref{273},\ref{274},\ref{275},\ref{276},\ref{278},\ref{277},\ref{285},\ref{286},\ref{287},\ref{288},\ref{289},\ref{293},\ref{292},\ref{290},\\
\ref{291},\ref{298},\ref{299},\ref{300},\ref{295},\ref{294},\ref{301},\ref{302},\ref{303},\ref{304},\ref{305},\ref{306},\ref{296},\ref{307},\ref{297},\ref{308},\ref{309},\ref{310},\ref{311}) of the list in section 6.
\begin{lem}\label{thm:20}
We assume that $X/G\cong{\mathbb F}_{n}$ for $n\geq 1$.
If $B=aC+\sum_{i=1}^{k}b_{i}B_{i}+\sum_{j=1}^{l}d_{j}B^{j}_{0,1}$ where $a,b_{i},d_{j}\geq2$, then $d_{1}=\cdots=d_{l}$, $2\leq\sum_{i=1}^{k}(C\cdot B_{i})+\sum_{j=1}^{l}(C\cdot B^{j}_{0,1})\leq3$, and $b_{i}=d_{1}$ if $(C\cdot B_{i})\not=0$ for $i=1,\ldots,k$.
\end{lem}
\begin{proof}
In the sane way of Lemma \ref{thm:11}, we get that
for $p^{\ast}C=\sum_{i=1}^{m}C_{i}$, the finite map $p_{|C_{i}}:C_{i}\rightarrow C$ is the Abelian cover between ${\mathbb P}^{1}$ whose branch divisor is $\sum_{j=1}^{l}d_{j}(C\cap F_{j})$ and Galois group is $\{g\in G:|\ g(C_{1})=C_{1}\}$.
By Theorem \ref{thm:10}, we get the claim.
\end{proof}
By Lemma \ref{thm:20}, the numerical class of $B$ is not one of 
(\ref{86},\ref{106},\ref{104},\ref{105},\ref{110},\ref{108},\\
\ref{120},\ref{121},
\ref{139},\ref{140},\ref{113},\ref{116},\ref{114},\ref{136},\ref{111},\ref{115},\ref{112},\ref{187},\ref{185},\ref{186},\ref{191},\ref{189},\ref{190},\ref{188},\ref{204},\ref{230},\\
\ref{228},\ref{229},\ref{225},\ref{219},\ref{231},\ref{232},\ref{235},\ref{238},\ref{247},\ref{242},\ref{255},\ref{254},\ref{252},\ref{253},\ref{259},\ref{264},\ref{265}) of the list in section 6.
\begin{lem}\label{thm:22}
If there are irreducible curves $B_{1}$ and $B_{2}$  and positive even integers $b_{1},b_{2}\geq2$
such that  $B=b_{1}B_{1}+b_{2}B_{2}$, and $(B_{1}\cdot B_{2})\not=0$, then  $(B_{1}\cdot B_{2})=8$.
\end{lem}
\begin{proof}
By Theorem \ref{thm:5}, $G=G_{B_{1}}\oplus G_{B_{2}}$ and $G_{B_{i}}\cong {\mathbb Z}/b_{i}{\mathbb Z}$ for $i=1,2.$
Let $s_{i}\in G_{B_{i}}$ be a generator for $i=1,2$. 
Since $G$ is abelian, $s_{1}^{\frac{b_{i}}{2}}\circ s_{2}^{\frac{b_{2}}{2}}$ is a symplectic automorphism of order 2.
Since $X/G$ is smooth, Fix$(s_{1}^{\frac{b_{i}}{2}}\circ s_{2}^{\frac{b_{2}}{2}})=p^{-1}(B_{1})\cap p^{-1}(B_{2})$. 
Since the support of $B$ is simple normal crossing and $|G|=b_{1}b_{2}$, we get that $|p^{-1}(B_{1})\cap p^{-1}(B_{2})|=(B_{1}\cdot B_{2})$.
By the fact that the fixed locus of a symplectic automorphism of order $2$ are 8 isolated points, we get that $(B_{1}\cdot B_{2})=8$.
\end{proof}
By Lemma \ref{thm:22}, the numerical class of $B$ is not one of (\ref{4},\ref{8},\ref{9},\ref{11},\ref{97},\ref{82},\ref{83},\\
\ref{85},\ref{184},\ref{201},\ref{175}) of the list in section 6.
\begin{lem}\label{thm:27}
If there are irreducible curves $B_{1}$ and $B_{2}$ such that $B=3B_{1}+3B_{2}$ and $(B_{1}\cdot B_{2})\not=0$, then  $(B_{1}\cdot B_{2})=3$.
\end{lem}
\begin{proof}
By Theorem \ref{thm:5}, $G=G_{B_{1}}\oplus G_{B_{2}}$ and $G_{B_{i}}\cong {\mathbb Z}/3{\mathbb Z}$ for $i=1,2.$
Let $s_{i}\in G_{B_{i}}$ be a generator for $i=1,2$. 
Since $G$ is abelian, we may assume that $s_1\circ s_2$ is a non-symplectic automorphism of order 3.
By Theorem  \ref{thm:5}, Fix$(s_1\circ s_2)$ does not contain a curve. 
Then by [\ref{bio:2}, Theorem 2.8] or [\ref{bio:7}, Table 2], Fix$(s_1\circ s_2)$ is only three isolated points.   
Since $X/G$ is smooth, Fix$(s_{1}\circ s_{2})=p^{-1}(B_{1})\cap p^{-1}(B_{2})$. 
Since $B_{1}+B_{2}$ is simple normal crossing and  $G=G_{B_{1}}\oplus G_{B_{2}}$, we get that $|p^{-1}(B_{1})\cap p^{-1}(B_{2})|=(B_{1}\cdot B_{2})$.
Therefore, we get  $(B_{1}\cdot B_{2})=3$.
\end{proof}
By Lemma \ref{thm:27}, the numerical class of $B$ is not one of (\ref{5},\ref{6},\ref{200},\ref{168}) of the list in section 6.
\begin{lem}\label{thm:28}
If there are irreducible curves $B_{i}$ and positive integers $b_{i}\geq2$ for $i=1,\ldots,k$ such that 
$B=\sum_{i=1}^{k}b_{i}B_{i}$ and $G=G_{B_{i}}$ for some $i$, then  $(B_{i}\cdot B_{j})=0$ for $j\not=i$.
\end{lem}
\begin{proof}
Recall that by Theorem \ref{thm:5}, $G_{B_{m}}$ is generated by a non-symplectic automorphism of order $b_{m}$ and Fix$(G_{B_{m}})\supset p^{-1}(B_{m})$ for $m=1,\ldots,k$.
If $(B_{i}\cdot B_{j})\not=0$ for $j\not=0$, then $p^{-1}(B_{i})\cap p^{-1}(B_{j})$ is not an empty set.
By the fact that the fixed locus of an automorphism is a pairwise set of points and curves, this is a contradiction.
\end{proof}
By Lemma \ref{thm:28},  the numerical class of $B$ is not one of (\ref{7},\ref{84},\ref{85,1},\ref{173},\ref{174}) of the list in section 6.
\begin{lem}\label{thm:29}
If there are irreducible curves $B_{1}$ and $B_{2}$ such that $B=2B_{1}+2B_{2}$ and $(B_{1}\cdot B_{2})\not=0$, then $\frac{B_{i}}{2}\in$Pic$({\mathbb F}_{n})$ for $i=1,2$.
\end{lem}
\begin{proof}
By Theorem \ref{thm:5}, $G=G_{B_{1}}\oplus G_{B_{2}}$ and $G_{B_{i}}\cong{\mathbb Z}/2{\mathbb Z}$ for $i=1,2$. 
Since the fixed locus of a non-symplectic automorphism of order 2 is a set of pairwise set of smooth curves or empty set, $X/G_{B_{i}}$ is smooth. 
Then there is a double cover $X/G_{B_{i}}\rightarrow X/G\cong{\mathbb F}_{n}$ whose branch divisor is $2B_{j}$ for $i,j=1,2$ and $i\not=j$. 
By Theorem \ref{thm:12}, $\frac{B_{i}}{2}\in$Pic$({\mathbb F}_{n})$ for $i=1,2$.
\end{proof}
By Lemma \ref{thm:29}, the numerical class of $B$ is not one of (\ref{10},\ref{99},\ref{98}) of the list in section 6.
\begin{lem}\label{thm:30}
If there are irreducible curves $B_{1},B_{2},B_{3}$ such that $B=2B_{1}+3B_{2}+6B_{3}$ and $(B_{2}\cdot B_{2})\geq1$ and $(B_{i}\cdot B_{j})\not=0$ for $1\leq i<j\leq3$, then $(B_{2}\cdot B_{2})=1$.
\end{lem}
\begin{proof}
Theorem \ref{thm:5}, $G_{B_{1}}\cong{\mathbb Z}/2{\mathbb Z}$, $G_{B_{2}}\cong{\mathbb Z}/3{\mathbb Z}$, $G_{B_{3}}\cong{\mathbb Z}/6{\mathbb Z}$.
Since $(B_{i}\cdot B_{j})\not=0$ for $1\leq i<j\leq 3$, we get $G_{B_{1}}\oplus G_{B_{2}}\cap G_{B_{3}}=\{{\rm id}_{X} \}$. 
Therefore, $G=G_{B_{1}}\oplus G_{B_{2}}\oplus G_{B_{3}}$.
Since $(B_{2}\cdot B_{2})>0$, we get that $p^{\ast}B_{2}=3C_{2}$ and the only curve of Fix$(G_{B_{2}})$ is $C_{2}$.

We assume that $(B_{2}\cdot B_{2})\geq2$.
Since $|G|=36$, $(C^{2}_{1,1}\cdot C^{2}_{1,1})\geq8$, and hence the genus of $C^{2}_{1,1}$ is at least $5$. 
By [\ref{bio:2},\ref{bio:7}] and the only curve of Fix$(G_{B_{2}})$ is $C_{2}$, this is a contradiction. 
\end{proof}
By Lemma \ref{thm:30}, the numerical class of $B$ is not one of (\ref{17},\ref{213}) of the list in section 6.
\begin{lem}\label{thm:31}
If there are irreducible curves $B_{1},B_{2},B_{3}$ such that 
 $B=2B_{1}+4B_{2}+4B_{3}$ and $(B_{i}\cdot B_{j})\not=0$ for $1\leq i<j\leq3$, then $(B_{1}\cdot B_{2})=1$.
\end{lem}
\begin{proof}
Theorem \ref{thm:5}, $G_{B_{1}}\cong{\mathbb Z}/2{\mathbb Z}$ and $G_{B_{i}}\cong{\mathbb Z}/4{\mathbb Z}$ for $i=2,3$.
Since $(B_{i}\cdot B_{j})\not=0$ for $1\leq i<j\leq 3$, we get $G_{B_{1}}\cap (G_{B_{2}}\oplus G_{B_{3}})=\{{\rm id}_{X} \}$. 
Therefore, $G=G_{B_{1}}\oplus G_{B_{2}}\oplus G_{B_{3}}$.
Let $s\in G_{B_{1}}$ and $t\in G_{B_{2}}$ be generators. 
Then $s\circ t$ is a non-symplectic automorphism of order 4 and $p^{-1}(B_{1})\cap p^{-1}(B_{2})\subset{\rm Fix}(s\circ t)$.
By Theorem \ref{thm:5} and $G=G_{B_{1}}\oplus G_{B_{2}}\oplus G_{B_{3}}$, Fix$(s\circ t)$ does not contain a curve. 
By [\ref{bio:9},\ {\rm Proposition 1}], 
the number of isolated points of Fix$(s\circ t)$ is 4.  
If $(B_{1}\cdot B_{2})\geq2$, then $|p^{-1}(B_{1})\cap p^{-1}(B_{2})|\geq8$. 
This is a contradiction.
\end{proof}
By Lemma \ref{thm:31}, the numerical class of $B$ is not (\ref{18},\ref{150},\ref{122},\ref{214}) of the list in section 6.
\begin{lem}\label{thm:32} 
We assume that $X/G\cong{\mathbb P}^{1}\times{\mathbb P}^{1}$.
Then $B\not=a(\{q\}\times{\mathbb P}^{1})+bC_{1}+cC_{2}$ where  $C_{1}$ and $C_{2}$ are smooth curves on ${\mathbb P}^{1}\times{\mathbb P}^{1}$, $C_{1}\cap C_{2}\not=\emptyset$, and $a,b,c$ are even integers. 
\end{lem}
\begin{proof}
We assume that $B=a(\{q\}\times{\mathbb P}^{1})+bC_{1}+cC_{2}$ where  $C_{1}$ and $C_{2}$ are smooth curves on ${\mathbb P}^{1}\times{\mathbb P}^{1}$, $C_{1}\cap C_{2}\not=\emptyset$, and $a,b,c$ are even integers.
Since $C_{1}\cap C_{2}\not=\emptyset$, by Theorem \ref{thm:5}, $G=G_{C_{1}}\oplus G_{C{2}}$.
Since $b,c$ are even integers, $G_{C_{1}},G_{C_{2}}$ are cyclic groups, and  $G=G_{C_{1}}\oplus G_{C{2}}$, 
the number of non-symplectic involution of $G$ is 2. Since $(B_{1,0}\cdot C_{i})\not=0$ and $a$ is even, $G$ must have at least 3 non-symplectic involutions. 
This is a contradiction.
\end{proof}
By Lemma \ref{thm:32}, the numerical class of $B$ is not one of (\ref{19},\ref{20},\ref{23},\ref{21},\ref{24},\ref{25},\ref{26}) of the list in section 6.
\begin{lem}\label{thm:33}
If there are irreducible curves $B_{i}$ and positive integers $b_{i}\geq2$ for $i=1,\ldots,k$ such that  $B=\sum_{i=1}^{k}b_{i}B_{i}$, 
 $G=G_{B_{1}}\oplus G_{B_{2}}$ and $b_{1}$ and $b_{2}$ are coprime,
then for each  $i=1,2$ $j=3,\ldots,k$, we get that  $b_{i}$ and $b_{j}$ are coprime if $(B_{i}\cdot B_{j})\not=0$.
\end{lem}
\begin{proof}
Let $s\in G_{B_{1}}$ and $t\in G_{B_{2}}$ be generators.  Theorem \ref{thm:5}, the order of $s$ is $b_{1}$ and that of $t$ is $b_{2}$.
Since $G=G_{B_{1}}\oplus G_{B_{2}}$, there are integers $u$ and $v$ such that $G_{B_{j}}$ is generated by $s^{u}\circ t^{v}$.

We assume that $(B_{1}\cdot B_{j})\not=0$ and $b_{1}$ and $b_{j}$ are not coprime.
Since $b_{1}$ and $b_{2}$ are coprime,  there is an integer $l$ such that $(s^{u}\circ t^{v})^{l}\not={\rm id}_{X}$ and $(s^{u}\circ t^{v})^{l}=s^{m}$ or $t^{m}$.
Since $b_{1}$ and $b_{j}$ are not coprime, we assume that $(s^{u}\circ t^{v})^{l}=s^{m}$.
Then $p^{-1}(B_{1})$ and $p^{-1}(B_{j})$ are contained in Fix$(s^{m})$.
By the fact that the fixed locus of an automorphism is a pairwise set of points and curves, this is a contradiction.
\end{proof}
By Theorem \ref{thm:5}, and Lemma \ref{thm:33}, the numerical class of $B$ is not one of (\ref{22},\ref{28},\ref{227},\ref{226},\ref{244},\ref{243},\ref{261},\ref{262}) of the list in section 6.\\

We assume that the numerical class of $B$ is (\ref{27}) of the list in section 6.
We denote $B$ by $3B_{1,0}+3B_{2,2}+3B_{0,1}$.
By Theorem \ref{thm:5}, $G=G_{2,2}$.
Since $G_{2,2}\cong{\mathbb Z}/3{\mathbb Z}$,  $G$ has 1 subgroups generated by a non-symplectic automorphism of order $3$.
Since $(B_{1,0}\cdot B_{2,2})\not=0$, $G$ contains at least 2 such a subgroup  from Theorem \ref{thm:5}. 
This is a contradiction. 
\begin{lem}\label{thm:34}
If there are irreducible curves $B_{1},B_{2},B_{3}$ such that
$B=2B_{1}+2B_{2}+2B_{3}$, and $(B_{i}\cdot B_{j})\not=0$ for $1\leq i<j\leq3$, then we get that $\frac{B_{3}}{2}\in$Pic$({\mathbb F}_{n})$ if $(B_{1}\cdot B_{2})=4$.
\end{lem}
\begin{proof}
By Theorem \ref{thm:5}, $G_{B_{i}}\cong{\mathbb Z}/2{\mathbb Z}$ for $i=1,2,3$. 	
Since $(B_{i}\cdot B_{j})\not=0$ for $1\leq i<j\leq3$, 
by Theorem \ref{thm:5}, $G=G_{B_{1}}\oplus G_{B_{2}}\oplus G_{B_{3}}$. 

We assume that $(B_{1}\cdot B_{2})=4$. Then $p^{-1}(B_{1})\cap p^{-1}(B_{2})$ is a set of 8 points.
Since the fixed locus of a symplectic automorphism of order 2 is a set of 8 isolated points, $X/G_{B_{1}}\oplus G_{B_{2}}$ is smooth. 
Then there is a double cover $X/G_{B_{1}}\oplus G_{B_{2}}\rightarrow X/G\cong{\mathbb F}_{n}$ whose branch divisor is $2B_{3}$. 
Thus, $\frac{B_{3}}{2}\in$Pic$({\mathbb F}_{n})$ for $i=1,2$.
\end{proof}
By Lemma \ref{thm:34}, the numerical class of $B$ is not one of (\ref{29},\ref{151},\ref{153},\ref{215}) of the list in section 6.
\begin{lem}\label{thm:35}
If there are irreducible curves $B_{1},B_{2},B_{3}$ such that
$B=2B_{1}+2B_{2}+2B_{3}$, and $(B_{i}\cdot B_{j})\not=0$ for $1\leq i<j\leq3$, then  $(B_{i}\cdot B_{j})\leq4$ for $1\leq i<j\leq3$.
\end{lem}
\begin{proof}
By Theorem \ref{thm:5}, $G_{B_{i}}\cong{\mathbb Z}/2{\mathbb Z}$ for $i=1,2,3$ and $G=G_{B_{1}}\oplus G_{B_{2}}\oplus G_{B_{3}}$. 
Let $s,t\in G$ be generators of $G_{B_{i}}$ and $G_{B_{j}}$ respectively where $1\leq i<j\leq3$. Then $s\circ t$ is a symplectic automorphism of order 2 and $p^{-1}(B_{i})\cap p^{-1}(B_{j})\subset$Fix$(s\circ t)$.
Since $|G|=8$, we get $2(B_{i}\cdot B_{j})=|p^{-1}(B_{i})\cap p^{-1}(B_{j})|$.
Thus, we have that $(B_{i}\cdot B_{j})\leq4$. 
\end{proof}
By Lemma \ref{thm:35}, the numerical class of $B$ is not one of (\ref{30},\ref{152}) of the list in section 6.
\begin{lem}\label{thm:36}
We assume that $X/G\cong{\mathbb P}^{1}\times{\mathbb P}^{1}$.
Then $B\not=a_{1}(\{q_{1}\}\times{\mathbb P}^{1})+a_{2}(\{q_{2}\}\times{\mathbb P}^{1})+bC'+c({\mathbb P}^{1}\times\{q_{3}\})$ where $C'$ is an irreducible curve, $C'=(nC+mF)$ in Pic$({\mathbb P}^{1}\times{\mathbb P}^{1})$, $n,m>0$, and $a_{1}a_{2},b,c$ are even integers. 
\end{lem}
\begin{proof}
We assume that $B=a_{1}(\{q_{1}\}\times{\mathbb P}^{1})+a_{2}(\{q_{2}\}\times{\mathbb P}^{1})+bC'+c({\mathbb P}^{1}\times\{q_{3}\})$ where $C'$ is an irreducible curve, $C'=(nC+mF)$, $n,m>0$, and $a_{1}a_{2},b,c$ are even integers. 
By Theorem \ref{thm:5}, $G=G^{2}_{1,0}\oplus G_{C'}$. By  $a_{1}a_{2}$ and $b$ are even integers, the number of non-symplectic involution of $G$ is 2. Since $(B_{0,1}\cdot C')\not=0$ and $(B_{0,1}\cdot B^{i}_{1,0})\not=0$ for $i=1,2$ and $c$ is an even integer, $G$ must have at least 3 non-symplectic involutions. This is a contradiction.
\end{proof}
By Lemma \ref{thm:36}, the numerical class of $B$ is not one of (\ref{39},\ref{40}) of the list in section 6.
\begin{lem}\label{thm:37}
We assume that $X/G\cong{\mathbb P}^{1}\times{\mathbb P}^{1}$.
Then $B\not=a_{1}(\{q_{1}\}\times{\mathbb P}^{1})+b_{1}C_{1}+b_{2}C_{2}+a_{2}({\mathbb P}^{1}\times\{q_{2}\})$ where 
$C_{i}$ is an irreducible curve, $C_{i}=(n_{i}C+m_{i}F)$ in Pic$({\mathbb P}^{1}\times{\mathbb P}^{1})$, $n_{i},m_{i}>0$, for $i=1,2$, and $a_{1},a_{2},b_{1}b_{2}$ are even integers. 
\end{lem}
\begin{proof}
We assume that $B=a_{1}(\{q_{1}\}\times{\mathbb P}^{1})+b_{1}C_{1}+b_{2}C_{2}+a_{2}({\mathbb P}^{1}\times\{q_{2}\})$ where 
$C_{i}$ is an irreducible curve, $C=(n_{i},m_{i})$ in Pic$({\mathbb P}^{1}\times{\mathbb P}^{1})$, $n_{i},m_{i}>0$, for $i=1,2$, and $a_{1},a_{2},b_{1}b_{2}$ are even integers. 
By Theorem \ref{thm:5}, $G=G_{C_{1}}\oplus G_{C_{2}}$. By  $b_{1}b_{2}$ is an even integer, the number of non-symplectic involutions of $G$ is at most 2. 
Since $(B_{1,0}\cdot C_{i})\not=0$ and $(B_{0,1}\cdot C_{i})\not=0$ for $i=1,2$, and $a_{1}$ and $a_{2}$ are even integers, $G$ must have at least 3 non-symplectic involutions. This is a contradiction.
\end{proof}
By Lemma \ref{thm:37},  the numerical class of $B$ is not one of (\ref{35},\ref{36},\ref{37},\ref{38},\ref{43},\ref{44}) of the list in section 6.
\\

We assume that the numerical class of $B$ is (\ref{45}) of the list in section 6.
We denote $B$ by $3B_{1,0}+2B^{1}_{1,1}+6B^{2}_{1,1}+3B_{0,1}$.
By Theorem \ref{thm:5}, $G^{1}_{1,1}\cong{\mathbb Z}/2{\mathbb Z}$ and  $G^{2}_{1,1}\cong{\mathbb Z}/6{\mathbb Z}$ and $G=G^{1}_{1,1}\oplus G^{2}_{1,1}$. 
Then the number of subgroups of G which is generated by a non-symplectic automorphism of order 3 is 1.  
By Theorem \ref{thm:5} and $(B_{1,0}\cdot B^{2}_{1,1})\not=0$, $G$ must have at least 2 such subgroups. 
This is a contradiction.\\

We assume that the numerical class of $B$ is (\ref{46}) of the list in section 6.
We denote $B$ by $3B_{1,0}+3B^{1}_{1,1}+3B^{2}_{1,1}+3B_{0,1}$.
By Theorem \ref{thm:5}, $G^{i}_{1,1}\cong{\mathbb Z}/3{\mathbb Z}$ for $i=1,2$, and $G=G^{1}_{1,1}\oplus G^{2}_{1,1}$. 
Then the number of subgroups of G which is generated by a non-symplectic automorphism of order 3 is 3.  
By Theorem \ref{thm:5}, $(B_{1,0}\cdot B^{i}_{1,1})\not=0$, and $(B_{1,0}\cdot B_{0,1})\not=0$, 
$G$ must have at least 4 such subgroups. 
This is a contradiction.\\

We assume that the numerical class of $B$ is (\ref{48}) of the list in section 6.
We denote $B$ by $2B^{1}_{1,0}+6B^{2}_{1,0}+3B_{1,2}+3B_{0,1}$.
By Theorem \ref{thm:5}, $G^{1}_{1,0}\cong{\mathbb Z}/2{\mathbb Z}$ and  $G_{1,2}\cong{\mathbb Z}/3{\mathbb Z}$, and $G=G^{1}_{1,0}\oplus G_{1,2}$. 
Then the number of subgroups of G which is generated by a non-symplectic automorphism of order 3 is 1.  
By Theorem \ref{thm:5} and $(B_{0,1}\cdot B_{1,2})\not=0$, $G$ must have at least 2 such subgroups. 
This is a contradiction.\\
	
We assume that the numerical class of $B$ is (\ref{54}) of the list in section 6.
We denote $B$ by $2B_{1,0}+2B^{1}_{1,1}+2B^{2}_{1,1}+2B^{3}_{1,1}+2B_{0,1}$.
By Theorem \ref{thm:5} $G^{i}_{1,1}\cong{\mathbb Z}/2{\mathbb Z}$ for $i=1,2,3$.
Since $(B^{i}_{1,1}\cdot B^{j}_{1,1})\not=0$ for $1\leq i<j\leq3$ and $G^{i}_{1,1}\cong{\mathbb Z}/2{\mathbb Z}$ for $i=1,2,3$,
$G=G^{1}_{1,1}\oplus G^{2}_{1,1}\oplus G^{3}_{1,1}$. 
Then the number of non-symplectic involutions of $G$ is 4.  
Since $(B_{1,0}\cdot B_{0,1})\not=0$, $(B_{0,1}\cdot C_{i})\not=0$, and $(B_{1,0}\cdot C_{i})\not=0$ for $i=1,2,3$, $G$ must have at least 5 non-symplectic involutions. 
This is a contradiction.
\begin{lem}\label{thm:38}
We assume that $X/G\cong{\mathbb P}^{1}\times{\mathbb P}^{1}$.
If $B=\sum_{i=1}^{2}a_{i}(\{p_{i}\}\times{\mathbb P}^{1})+bC'+\sum_{j=1}^{2}c_{j}({\mathbb P}^{1}\times\{q_{j}\})$ 
where $C'$ is an irreducible curve, $\{p_{i}\}\times{\mathbb P}^{1}\cap C'\not=\emptyset$, $C\cap{\mathbb P}^{1}\times\{q_{i}\}\not=\emptyset$ $a_{i},c_{1},c_{2},b\in{\mathbb N}_{\geq2}$, then  $a_{1}=a_{2}$ and $c_{1}=c_{2}$.
\end{lem}
\begin{proof}
Let $C_{p_{1}}$ be one of irreducible components of $p^{\ast}(\{p_{1}\}\times{\mathbb P}^{1})$. 
Since $(\{p_{1}\}\times{\mathbb P}^{1}\cdot \{p_{1}\}\times{\mathbb P}^{1})=0$, $C_{p_{1}}$ is an elliptic curve.
Let $\pi:X\rightarrow Y:=X/G_{C'}$ be the quotient map, and $G':=G/G_{C'}$ be a finite abelian subgroup of ${\rm Aut}(Y)$.
Since $\{p_{i}\}\times{\mathbb P}^{1}\cap C\not=\emptyset$, the finite map $\pi_{|C_{p_{1}}}:C_{p_{1}}\rightarrow C'_{p_{1}}:=\pi(C_{p_{1}})$ is a branched cover. 
Since $C_{p_{1}}$ is an elliptic curve, $C'_{p_{1}}$  is ${\mathbb P}^{1}$
Since the branch divisor of the quotient map $\pi':Y\rightarrow Y/G'\cong{\mathbb P}^{1}\times{\mathbb P}^{1}$
is  $\sum_{i=1}^{2}a_{i}\{p_{i}\}\times{\mathbb P}^{1}+\Sigma_{j=1}^{2}c_{j}{\mathbb P}^{1}\times\{q_{j}\}$,
the branch divisor of $\pi_{C'_{p_{1}}}:C'_{p_{1}}\rightarrow {p_{1}}\times{\mathbb P}^{1}$ is  is $c_{1}q_{1}+c_{2}q_{2}$. 
By Theorem \ref{thm:10}, we get that $c_{1}=c_{2}$.  
In the same way, we obtain that $a_{1}=a_{2}$.
\end{proof}
By Lemma \ref{thm:38}, the numerical class of $B$ is not one of (\ref{57},\ref{58},\ref{59},\ref{60}) of the list in section 6.
\\

We assume that the numerical class of $B$ is one of (\ref{67},\ref{67,1},\ref{67,2},\ref{68},\ref{68,1},\ref{69},\ref{70},\ref{71},\ref{72},\ref{73}) of the list in section 6.
By Theorem \ref{thm:10}, there are an abelian surface and a finite group $G$ such that $A/G={\mathbb P}^{1}\times{\mathbb P}^{1}$ and the branch divisor is $B$. 
By Theorem \ref{thm:4}, there is a surjective morphism from a $K3$ surface to an abelian surface. This is a contradiction.
\begin{lem}\label{thm:39}
If $X/G\cong{\mathbb F}_{n}$ where $n\geq1$,
then $B\not=aC+bB_{s,t}+cB_{u,v}+dB_{0,1}$ where $a,d\geq0$ are  even integers, $a=0$ or $a\geq2$, and, $b,c>0$ are even integers.
\end{lem}
\begin{proof}
We assume that $B=aC+bB_{s,t}+cB_{u,v}+dB_{0,1}$ where $a,d\geq0$ are  even integers, $a=0$ or $a\geq2$, and, $b,c>0$ are even integers.
By Theorem \ref{thm:5} and $(B_{s,t}\cdot B_{u,v})\not=0$, we get that $G=G_{s,t}\oplus G_{u,v}$. 
Then the number of non-symplectic involution of $G$ is 2. 
Since $(B_{s,t}\cdot B_{0,1})\not=0$ and $(B_{u,v}\cdot B_{0,1})\not=0$, $G$ must have at least 3 non-symplectic involutions. 
This is a contradiction.
\end{proof}
By Lemma \ref{thm:39}, the numerical class of $B$ is not one of 
(\ref{100},\ref{101},\ref{102},\ref{103},\ref{118},\ref{107},\\
\ref{129},\ref{128},\ref{127},\ref{195},\ref{194},\ref{199},\ref{198},\ref{224},\ref{246}) of the list in section 6.
\begin{lem}\label{thm:40}
For the branch divisor $B=\sum_{i=1}^{k}b_{i}B_{i}$, we get that $\frac{|G|}{b^{2}_{i}}(B_{i}\cdot B_{i})$ is an even integer for $1\leq i\leq k$.
\end{lem}
\begin{proof}
For $i=1,\ldots,k$, we put $p^{\ast}B_{i}=\sum_{j=1}^{l}b_{i}C_{j}$ where $C_{j}$ is a smooth curve for $j=1,\ldots,l$. 
By Theorem \ref{thm:5}, $C_{1},\ldots,C_{l}$ are pairwise disjoint.
Then we get that $\frac{|G|}{b^{2}_{i}}(B_{i}\cdot B_{i})=\sum_{j=1}^{l}(C_{j}\cdot C_{j})$.
Since $X$ is a $K3$ surface, $(C_{j}\cdot C_{j})$ is an even integer, and hence $\frac{|G|}{b^{2}_{i}}(B_{i}\cdot B_{i})$ is an even integer.
\end{proof}
By Lemma \ref{thm:40} the numerical class of $B$ is not one of (\ref{149},\ref{154},\ref{119},\ref{109},\ref{126},\ref{123},\ref{193},\\
\ref{192},\ref{223}) of the list in section 6.
\\

We assume that the numerical class of $B$ is (\ref{155}) of the list in section 6.
We denote $B$ by $2B_{1,2}+4B^{1}_{1,1}+4B^{2}_{1,1}+2B_{0,1}$. 
By Theorem \ref{thm:5}, $G_{1,2}\cong\mathbb Z/2\mathbb Z$, and $G^{i}_{1,1}\cong\mathbb Z/4\mathbb Z$ for $i=1,2$
Since $(B_{1,2}\cdot B^{i}_{1,1})\not=0$ for $i=1,2$, by Theorem \ref{thm:5}, $G=G_{1,2}\oplus G^{1}_{1,1}\oplus G^{2}_{1,1}$, and hence $|G|=36$.
Let $s\in G_{1,2}$ and $t\in G^{1}_{1,1}$ be generators. 
Then $s\circ t$ is a non-symplectic automorphism of order 4 and $p^{-1}(B_{1,2})\cap p^{-1}(B^{1}_{1,1})\subset{\rm Fix}(s\circ t)$.
Since $G=G_{1,2}\oplus G^{1}_{1,1}\oplus G^{2}_{1,1}$ and $B=2B_{1,2}+4B^{1}_{1,1}+4B^{2}_{1,1}+2B_{0,1}$, by Theorem \ref{thm:5}, Fix$(s\circ t)$ does not contain a curve. 
By [\ref{bio:9},\ {\rm Proposition 1}], the number of isolated points of Fix$(s\circ t)$ is 4.  
Since  $(B_{1,2}\cdot B^{1}_{1,1})=2$, we get that $|p^{-1}(B_{1,2})\cdot p^{-1}(B^{1}_{1,1})|\geq8$. 
This is a contradiction.\\

We assume that the numerical class of $B$ is (\ref{156}) of the list in section 6.
We denote $B$ by $2B_{2,3}+2B^{1}_{1,1}+2B^{2}_{1,1}+2B_{0,1}$. 
By Theorem \ref{thm:5}, $G_{2,3}\cong G^{1}_{1,1}\cong G^{2}_{1,1}\cong\mathbb Z/2\mathbb Z$.
Since an intersection of two of $B_{2,3},B^{1}_{1,1},B^{2}_{1,1:}$ is not an empty set,
by Theorem \ref{thm:5}, $G=G_{2,3}\oplus G^{1}_{1,1}\oplus G^{2}_{1,1}$, and hence $|G|=8$.  
Let $s\in G_{2,3}$ and $t\in G_{0,1}$ be generators. 
Since $s$ and $t$ are non-symplectic involutions, Fix$(s)$ and Fix$(t)$ are a set of curves and Fix$(s\circ t)$ is a set of 8 isolated points. 
Since $(B_{2,3}\cdot B_{0,1})=2$, $|p^{-1}(B_{2,3})\cap p^{-1}(B^{1}_{1,1})|=4$.
Since Fix$(s\circ t)\supset p^{-1}(B_{2,3})\cap p^{-1}(B_{0,1})$, $X/(G_{2,3}\oplus G_{0,1})$ has 2 singular points, however,
since the branch divisor of the double cover $X/(G_{2,3}\oplus G_{0,1})\rightarrow X/G$ is $2B^{1}_{1,1}+2B^{2}_{1,1}$ and $(B^{1}_{1,1}\cdot B^{2}_{1,1})=1$,  the number of singular points of $X/(G_{2,3}\oplus G_{0,1})$ must be 1. 
This is a contradiction.  

As for the case of (\ref{156}), the numerical class of $B$ is not one of (\ref{159},\ref{158}) of the list in section 6.
\\

We assume that the numerical class of $B$ is (\ref{157}) of the list in section 6.
We denote $B$ by $2B_{2,2}+2B_{1,2}+2B_{1,1}+2B_{0,1}$. 
By Theorem \ref{thm:5}, $G_{2,2}\cong G_{1,2}\cong G_{1,1}\cong\mathbb Z/2\mathbb Z$.
Since an intersection of two of $B_{2,2},B_{1,2},B_{1,1}$ is not an empty set,
by Theorem \ref{thm:5}, $G=G_{2,2}\oplus G_{1,2}\oplus G_{1,1}$. 
Since $(B_{2,2}\cdot B_{1,2})=4$,  $X/(G_{2,2}\oplus G_{1,2})$ is smooth. 
Then there is a double cover $X/G_{2,2}\oplus G_{1,2}\rightarrow X/G\cong{\mathbb F}_{1}$ whose branch divisor is $2B_{1,1}+2B_{0,1}$. 
Since $\frac{B_{1,1}+B_{0,1}}{2}\not\in$Pic$({\mathbb F}_{1})$, by Theorem \ref{thm:12}, this is a contradiction.\\

We assume that the numerical class of $B$ is (\ref{164}) of the list in section 6.
We denote $B$ by $2B^{1}_{1,2}+2B^{2}_{1,2}+2B^{1}_{1,1}+2B^{2}_{1,1}$. 
By Theorem \ref{thm:5}, $G^{i}_{1,2}\cong G^{i}_{1,1}\cong\mathbb Z/2\mathbb Z$ for $i=1,2$.
Since an intersection of two of $B^{1}_{1,2},B^{2}_{1,2},B^{1}_{1,1},B^{2}_{1,1}$ is not an empty set,
by Theorem \ref{thm:5}, $G=G^{1}_{1,2}\oplus G^{2}_{1,2}\oplus G^{1}_{1,1}\oplus G^{2}_{1,1}$ or $G=G^{1}_{1,2}\oplus G^{2}_{1,2}\oplus G^{1}_{1,1}$. 
We assume that $G=G^{1}_{1,2}\oplus G^{2}_{1,2}\oplus G^{1}_{1,1}\oplus G^{2}_{1,1}$. 
Since $|G|=16$ and $(B^{1}_{1,2}\cdot B^{2}_{1,2})=3$, $|p^{-1}(B^{1}_{1,2}\cap B^{2}_{1,2})|\geq12$.
Since the number of isolated points of symplectic involution is 8, this is a contradiction. 
Therefore,  $G=G^{1}_{1,2}\oplus G^{2}_{1,2}\oplus G^{1}_{1,1}$.

By Theorem \ref{thm:12}, there are the Galois covers $p_{1}:Y_{1}\rightarrow{\mathbb F}_{1}$ and $p_{2}:Y_{2}\rightarrow{\mathbb F}_{1}$
such that the branch divisor of $p_{1}$ is $2B^{1}_{1,2}+2B^{2}_{1,2}$ and that of $p_{2}$ is $2B^{1}_{1,1}+2B^{2}_{1,1}$.
Let $X':=Y_{1}\times_{{\mathbb F}_{1}}Y_{2}$. 
Then there is the Galois cover $q:X'\rightarrow{\mathbb F}_{1}$ whose branch divisor is $2B^{1}_{1,2}+2B^{2}_{1,2}+2B^{1}_{1,1}+2B^{2}_{1,1}$ and Galois group is isomorphic to ${\mathbb Z}/2{\mathbb Z}^{\oplus 2}$ as a group.
By Theorem \ref{thm:3}, there is a symplectic automorphism of order 2 $s\in G$ such that  $X'=X/\langle s\rangle$. 
Since $s$ is symplectic, the minimal resolution $f:X'_{m}\rightarrow X'$ is a $K3$ surface. 
Let $e_{1},\ldots,e_{8}$ be the exceptional divisors of $f$. 
We set $\{p_{1},p_{2},p_{3}\}:=B^{1}_{1,2}\cap B^{2}_{1,2}$, and $\{p_{4}\}:=B^{1}_{1,1}\cap B^{2}_{1,1}$.
Let $\pi:{\rm Blow}_{\{p_{1},p_{2},p_{3},p_{4}\}}{\mathbb F}_{1}\rightarrow{\mathbb F}_{1}$ be the blow-up of ${\mathbb P}^{1}\times{\mathbb P}^{1}$ at points $p_{1},p_{2},p_{3},p_{4}$, and $E_{i}:=\pi^{-1}(p_{i})$ be an exceptional divisor of $\pi$ for $i=1,2,3,4$. 
Since the support of $B$ is simple normal crossing, in the same way of Proposition \ref{pro:5}, 
there is a Galois cover $q:X'_{m}\rightarrow {\rm Blow}_{\{p_{1},p_{2},p_{3},p_{4}\}}{\mathbb F}_{1}$ whose branch divisor is  $2C^{1}_{1,2}+2C^{2}_{1,2}+2C^{1}_{1,1}+2C^{2}_{1,1}$ and Galois group is isomorphic to ${\mathbb Z}/2{\mathbb Z}^{\oplus 2}$ as a group, 
where $C^{1}_{1,2},C^{2}_{1,2},C^{1}_{1,1},C^{2}_{1,1}$ are proper transform of $B^{1}_{1,2},B^{2}_{1,2},B^{1}_{1,1},B^{2}_{1,1}$ by the birational map $\pi^{-1}$ in order. 
Notice that $q^{\ast}(\sum_{i=1}^{4}E_{i})=\sum_{j=1}^{8}e_{j}$ and there is the commutative diagram:
$$
\xymatrix{
	X' \ar[r] &{\mathbb F}_{1} \\
	X'_{m} \ar[u]^{f}\ar[r]_{q} &{\rm Blow}_{\{p_{1},p_{2},p_{3},p_{4}\}}{\mathbb F}_{1} \ar[u]^{\pi}.
}
$$
Furthermore, we put $\{x_{1},\ldots,x_{8}\}:=$Fix$(s)$. Then ${\rm Blow}_{\{x_{1},\ldots,x_{8\}}}X/\langle s\rangle=X'_{m}$, the branch divisor of the double cover ${\rm Blow}_{\{x_{1},\ldots,x_{8}\}}X\rightarrow X'_{m}$ is $\sum_{j=1}^{8}e_{j}$, and there is the commutative diagram:
$$
\xymatrix{
	X \ar[r] &X' \\
	{\rm Blow}_{\{x_{1},\ldots,x_{8}\}}X \ar[u]\ar[r] &X'_{m} \ar[u].
}
$$
In the same way of Proposition \ref{pro:5},  we get that
\[ \sum_{i=1}^{4}q^{\ast}E_{i}=2(\pi\circ q)^{\ast}(C+\frac{3}{2}F)-2C^{1}_{1,2}-2C^{1}_{1,1}\ {\rm in\ Pic}(X'_{m}).\]
Since ${\rm Blow}_{\{x_{1},\ldots,x_{8}\}}X$ and $X'_{m}$ are smooth, and  $q^{\ast}(\sum_{i=1}^{4}E_{i})=\sum_{j=1}^{8}e_{j}$, 
we get that $\frac{\sum_{i=1}^{4}q^{\ast}E_{i}}{2}\in$Pix$(X'_{m})$, and hence  $\frac{F}{2}\in$Pic$(X'_{m})$.

Since $C^{1}_{1,2}\cap C^{2}_{1,2}$ is an empty set and $\frac{C^{1}_{1,2}+C^{2}_{1,2}}{2}\in$Pic$({\rm Blow}_{\{p_{1},p_{2},p_{3},p_{4}\}}{\mathbb F}_{1})$, by Theorem \ref{thm:12},
there is the Galois cover $g:Z\rightarrow{\rm Blow}_{\{p_{1},p_{2},p_{3},p_{4}\}}{\mathbb F}_{1}$ such that $Z$ is smooth,  the branch divisor is $2C^{1}_{1,2}+2C^{2}_{1,2}$, and the Galois group is isomorphic to ${\mathbb Z}/2{\mathbb Z}$ as a group.
By Theorem \ref{thm:3}, there is a non-symplectic automorphism of order 2 $\rho$ of $X'_{m}$ such that $X'_{m}/\langle \rho\rangle=Z$. 
Let $h:X'_{m}\rightarrow Z$ be the quotient map. 
Then $q=g\circ h$, and hence $\frac{F}{2}\in$Pic$(X'_{m})^{\rho}$.
Since the degree of $g$ is 2 and $(C^{1}_{1,2}\cdot \frac{F}{2})=\frac{1}{2}$ and $\frac{g^{\ast}C^{1}_{1,2}}{2}\in$Pic$(Z)$, we get that $g^{\ast}\frac{F}{2}\not\in$Pic$(Z)$.
Recall that $C^{i}_{1,1}=C+F-e_{4}$ in Pic$({\rm Blow}_{\{p_{1},p_{2},p_{3},p_{4}\}}{\mathbb F}_{1})$ for $i=1,2$.
Since the branch divisor of $h$ is $2g^{\ast}C^{1}_{1,1}+2g^{\ast}C^{2}_{1,1}$, we get that $q^{\ast}(\frac{1}{2}C+\frac{1}{2}F-e_{4})\in$Pic$(X'_{m})$.
By [\ref{bio:9}], Pic$(X'_{m})^{\rho}$ is generated by $h^{\ast}{\rm Pic}(Z)$ and $q^{\ast}(\frac{1}{2}C+\frac{1}{2}F-e_{4})$ over ${\mathbb Z}$.
Since $g^{\ast}\frac{F}{2}\not\in$Pic$(Z)$ and $2q^{\ast}(\frac{1}{2}C+\frac{1}{2}F-e_{4})\in h^{\ast}{\rm Pic}(Z)$, we may assume that there is $\alpha\in {\rm Pic}(Z)$ such that
\[ q^{\ast}\frac{F}{2}=h^{\ast}\alpha+q^{\ast}(\frac{1}{2}C+\frac{1}{2}F-e_{4}). \]
Then $g^{\ast}(\frac{-1}{2}C+e_{4})\in$Pic$(Z)$. 
Since the degree of $g$ is 2 and $(C^{1}_{1,2}\cdot \frac{-1}{2}C+e_{4})=\frac{3}{2}$ and $\frac{g^{\ast}C^{1}_{1,2}}{2}\in$Pic$(Z)$,
we get that  $(\frac{g^{\ast}C^{1}_{1,2}}{2}\cdot g^{\ast}(\frac{-1}{2}C+e_{4}))=\frac{3}{2}$. 
By the assumption that $\frac{g^{\ast}C^{1}_{1,2}}{2}\in$Pic$(Z)$ and $g^{\ast}(\frac{-1}{2}C+e_{4})\in$Pic$(Z)$, this is a contradiction.
Therefore, the numerical class of $B$ is not (\ref{164}) of the list in section 6.
\\

We assume that the numerical class of $B$ is (\ref{165}) of the list in section 6.
We denote $B$ by $2B_{1,2}+2B^{1}_{1,1}+2B^{2}_{1,1}+2B^{3}_{1,1}+2B_{0,1}$. 
By Theorem \ref{thm:5}, $G_{1,2}\cong G^{i}_{1,1}\cong G_{0,1}\cong\mathbb Z/2\mathbb Z$ where $i=1,2,3$.
Since an intersection of two of $B_{1,2},B^{1}_{1,1},B^{2}_{1,1},B^{3}_{1,1},B_{0,1}$ is not an empty set,
by Theorem \ref{thm:5}, $G=G_{1,2}\oplus G^{1}_{1,1}\oplus G^{2}_{1,1}\oplus G^{3}_{1,1}$.
Let $G_{s}$ be the subgroup of $G$ which consists of symplectic automorphisms of $G$. 
Then $G_{s}\cong\mathbb Z/2\mathbb Z^{\oplus 3}$.
By [\ref{bio:17}], the number of singular points of $X/G_{s}$ is 14, however,
since the branch divisor of the double cover $X/G_{s}\rightarrow X/G$ is $B=2B_{1,2}+2B^{1}_{1,1}+2B^{2}_{1,1}+2B^{3}_{1,1}+2B_{0,1}$ and the support of $B$ is simple normal crossing,  
the number of singular points of $X/G_{s}$ is 13. 
This is a contradiction.
Therefore, the numerical class of $B$ is not (\ref{165}) of the list in section 6.
\begin{lem}\label{thm:41}
If $X/G\cong{\mathbb F}_{n}$ where $n\geq1$, then $B\not=aC+bB_{s,t}+cB_{u,v}$ where $a,b,c>0$ are even integers, and
$(C\cdot B_{s,t})\not=0$ and $(C\cdot B_{u,v})\not=0$, i.e. $s\not=t$ or $u\not=v$.
\end{lem}
\begin{proof}
We assume that $B=aC+bB_{s,t}+cB_{u,v}$ where $a,b,c>0$ are even integers, and
$(C\cdot B_{s,t})\not=0$ and $(C\cdot B_{u,v})\not=0$
By Theorem \ref{thm:5} and $(B_{s,t}\cdot B_{u,v})\not=0$, $G=G_{s,t}\oplus G_{u,v}$. 
Then the number of non-symplectic involutions of $G$ is 2. 
Since $(C\cdot B_{s,t})\not=0$ and $(C\cdot B_{u,v})\not=0$, $G$ must have at least 3 non-symplectic involutions. 
This is a contradiction.
\end{proof}
By Lemma \ref{thm:41}, the numerical class of $B$ is not one of (\ref{117},\ref{125},\ref{124},\ref{197}) of the list in section 6.
\\

We assume that the numerical class of $B$ is (\ref{142}) of the list in section 6.
We denote $B$ by $2B_{1,0}+2B_{1,4}+2B^{1}_{1,1}+2B^{2}_{1,1}$.
By Theorem \ref{thm:5}, $G_{1,0}\cong G_{1,4}\cong G^{i}_{1,1}\cong\mathbb Z/2\mathbb Z$ where $i=1,2$. 
Since $(B_{1,4}\cdot B^{i}_{1,1})\not=0$ for $i=1,2$, by Theorem \ref{thm:5}, $G=G_{1,4}\oplus G^{1}_{1,1}\oplus G^{2}_{1,1}$. 
Let $s\in G^{1}_{1,1}$ and $t\in G^{2}_{1,1}$ be generators. 
Since the number of non-symplectic automorphisms of order 2  of $G$ is 4 and Theorem \ref{thm:5}, 
we may assume that Fix$(s)$ is the support of $p^{\ast}B^{1}_{1,1}$.
Since the support of $B$ is simple normal crossing and $(B_{1,4}\cdot B^{1}_{1,1})=4$, $X/(G_{1,4}\oplus G^{1}_{1,1})$ is smooth.
Then there is the Galois cover $X/G_{1,4}\oplus G^{1}_{1,1}\rightarrow{\mathbb F}_{1}$ such that the branch divisor is $2B_{1,0}+2B^{2}_{1,1}$ and the Galois group is isomorphic to ${\mathbb Z}/2{\mathbb Z}$ as a group. 
Since $\frac{B_{1,0}+B^{2}_{1,1}}{2}\not\in$Pic$({\mathbb F}_{1})$, this is a contradiction.

As for the case of (\ref{142}), the numerical class of $B$ does not (\ref{143}) of the list in section 6.
\\

We assume that the numerical class of $B$ is (\ref{178}) of the list in section 6.
We denote $B$ by $3B_{1,0}+3B_{1,2}+3B_{1,4}$. 
By Theorem \ref{thm:5} and $(B_{1,2}\cdot B_{1,4})\not=0$, $G=G_{1,2}\oplus G_{1,4}$.
Let $s\in G_{1,4}$ be a generator of $G_{1,4}$. 
Then the only curve of Fix$(s)$ is $C_{1,4}$.
Since $(B_{1,4}\cdot B_{1,4})=6$, the genus of $C_{1,4}$ is 4. 
By [\ref{bio:2},\ref{bio:7}], Fix$(s)$ does not have isolated points, and hence $X/G_{1,4}$ is smooth.
Let $q:X/G_{1,4}\rightarrow X/G$ be the quotient map. 
Then the degree of $q$ is 3, and the branch divisor of $q$ is $3B_{1,0}+3B_{1,2}$. 
Since the degree of $q$ is 3 and $X/G_{1,4}$ is smooth, $\frac{3}{3^{2}}(B_{1,0}\cdot B_{1,0})$ is an integer.
Since $(B_{1,0}\cdot B_{1,0})=-2$, $\frac{3}{3^{2}}(B_{1,0}\cdot B_{1,0})=-\frac{2}{3}$.
This is a contradiction. \\ 

We assume that the numerical class of $B$ is (\ref{179}) of the list in section 6.
We denote $B$ by $3B_{1,0}+3B_{1,2}+3B_{1,3}+3B_{0,1}$. 
By Theorem \ref{thm:5}, $G_{1,0}\cong G_{1,2}\cong G_{1,3}\cong G_{0,1}\cong\mathbb Z/3\mathbb Z$. 
Since $(B_{1,2}\cdot B_{1,3})\not=0$, by Theorem \ref{thm:5}, $G=G_{1,2}\oplus G_{1,3}$.
Let $s,t\in G$ be generators of $G_{1,2}$ and $G_{1,3}$ respectively such that  $s\circ t$ is a non-symplectic automorphism of order 3.
Since $G=G_{1,2}\oplus G_{1,3}$, the number of subgroups of $G$ which are generated by a non-symplectic automorphism of order 3 is 3. 
Since $(B_{1,2}\cdot B_{0,1})\not=0$ and $(B_{1,3}\cdot B_{0,1})\not=0$, we get that $p^{-1}(B_{0,1})$ is contained in Fix$(s\circ t)$, 
and hence $p^{-1}B_{1,0}$ is contained in Fix$(s)$.
Since $|G|=9$,  there is an elliptic curve $C_{0,1}$ on $X$ such that $p^{\ast}B_{0,1}=3C_{0,1}$. 
By [\ref{bio:2},\ref{bio:7}], the number of isolated points of Fix$(s\circ t)$ is 3. 
Since $(B_{1,0}\cdot B_{1,3})=1$ and $(B_{1,2}\cdot B_{1,3})=3$, we have $|p^{-1}(B_{1,0}\cup B_{1,2})\cap p^{-1}(B_{1,3})|=4$. 
Since $p^{-1}(B_{1,0}\cup B_{1,2})\subset$Fix$(s)$ and $p^{-1}(B_{1,3})\subset$Fix$(t)$, 
we get that $p^{-1}(B_{1,0}\cup B_{1,2})\cap p^{-1}(B_{1,3})\subset$Fix$(s \circ t)$. 
By the fact that  the number of isolated points of Fix$(s\circ t)$ is 3, this is a contradiction.  
\\

We assume that the numerical class of $B$ is (\ref{196}) of the list in section 6.
We denote $B$ by $2B_{1,0}+2B_{1,4}+2B_{2,4}$. 
By Theorem \ref{thm:5}, $G=G_{1,4}\oplus G_{2,4}$.
Let $s\in G$ be a generator of $G_{1,4}$. 
Since $(B_{1,0}\cdot B_{1,4})\not=0$ and $(B_{1,4}\cdot B_{2,4})\not=0$,  the only curve of Fix$(s)$ is $C_{1,4}$.
Since the fixed locus of a non-symplectic involution does not have isolated points, $X/G_{1,4}$ is smooth. 
Let $q:X/G_{1,4}\rightarrow X/G\cong{\mathbb F}_{2}$ be the quotient map. 
The degree of $q$ is 2 and the branch divisor of $q$ is $2B_{1,0}+2B_{2,2}$. 
Since  $\frac{B_{1,0}+B_{2,2}}{2}\not\in$Pic$({\mathbb F}_{2})$, by Theorem \ref{thm:12}, this is a contradiction. \\ 

We assume that the numerical class of $B$ is (\ref{210}) of the list in section 6.
We denote $B$ by $2B_{1,0}+2B^{1}_{1,3}+2B^{2}_{1,3}+2B_{1,2}$. 
By Theorem \ref{thm:5}, $G^{i}_{1,3}\cong G_{1,2}\cong\mathbb Z/2\mathbb Z$ where $i=1,2$. 
Since an intersection of two of $B^{1}_{1,3},B^{2}_{1,3},B_{1,2}$ is not an empty set, $
G=G^{1}_{1,3}\oplus G^{2}_{1,3}\oplus G_{1,2}$. 
Since $|G|=8$ and $(B^{1}_{1,3}\cdot B^{2}_{1,3})=4$, $Y:=X/(G^{1}_{1,3}\oplus G^{2}_{1,3})$ is smooth.
Then there is the Galois cover $q:Y\rightarrow X/G$ such that the branch divisor is $2B_{1,0}+2B_{1,2}$, 
and the Galois group is ${\mathbb Z}/2{\mathbb Z}$ as a group.
Since the fixed locus of a non-symplectic automorphism of order 2 does not have isolated points,
$X/G^{1}_{1,3}$ is smooth, and  there is the Galois cover $q'':X/G^{1}_{1,3}\rightarrow Y$ such that the branch divisor of $q''$ is $2q^{\ast}B^{1}_{1,3}$ and the Galois group of $q''$ is ${\mathbb Z}/2{\mathbb Z}$ as a group.
Since $Y$ and $X/G^{1}_{1,3}$ are smooth, and the degree of $q''$ is two, we get that $\frac{q^{\ast}B^{1}_{1,3}}{2}\in$Pic$(Y)$.	
Recall that the branch divisor of $q$ is $2B_{1,0}+2B_{1,2}$, and the degree of $q$ is two.
Since $\frac{q^{\ast}B_{1,2}}{2}\in$Pic$(Y)$, we get that $\frac{q^{\ast}F}{2}=\frac{q^{\ast}B^{1}_{1,3}}{2}-\frac{q^{\ast}B_{1,2}}{2}\in$Pic$(Y)$.
Since $(B_{1,0}\cdot F)=1$, we get that $(\frac{q^{\ast}B_{1,0}}{2}\cdot \frac{q^{\ast}F}{2})=\frac{1}{2}$.
Since  $\frac{q^{\ast}B_{1,0}}{2}\in$Pic$(Y)$ and $\frac{q^{\ast}F}{2}\in$Pic$(Y)$, this is a contradiction.
Therefore, the numerical class of $B$ is not (\ref{210}).\\

We assume that the numerical class of $B$ is (\ref{211}) of the list in section 6.
We denote $B$ by $2B_{1,0}+2B_{1,3}+2B^{1}_{1,2}+2B^{2}_{1,2}+2B_{0,1}$. 
By Theorem \ref{thm:5}, $G_{1,3}\cong G^{i}_{1,2}\cong\mathbb Z/2\mathbb Z$ where $i=1,2$.
Since an intersection of two of $B_{1,3},B^{1}_{1,2},B^{2}_{1,2}$ is not an empty set, by Theorem \ref{thm:5}, 
$G=G_{1,3}\oplus G^{1}_{1,2}\oplus G^{2}_{1,2}$. 
Let $s\in G^{1}_{1,2}$ be a generator. 
Since the number of non-symplectic automorphisms of order 2 of $G$ is 4 and Theorem \ref{thm:5}, we may assume that
$p^{-1}(B^{1}_{1,3})$ and $p^{-1}(B_{1,0})$ are contained in Fix$(s)$.
Since the support of $B$ is simple normal crossing and $(B_{1,3}\cdot B_{1,0}+B^{1}_{1,2})=4$, 
$X/(G_{1,3}\oplus G^{1}_{1,2})$ is smooth and there is the Galois cover $X/(G_{1,3}\oplus G^{1}_{1,2})\rightarrow{\mathbb F}_{2}$ such that the branch divisor is $2B^{2}_{1,2}+2B_{0,1}$ and the Galois group is ${\mathbb Z}/2{\mathbb Z}$ as a group. 
Since $\frac{B^{2}_{1,2}+B_{0,1}}{2}\not\in$Pic$({\mathbb F}_{2})$, this is a contradiction.\\

We assume that the numerical class of $B$ is (\ref{241}) of the list in section 6. 
We denote $B$ by $2B_{1,0}+3B^{1}_{1,4}+6B^{2}_{1,4}$. 
By Theorem \ref{thm:5}, $G_{1,0}\cong\mathbb Z/2\mathbb Z$, $G^{1}_{1,4}\cong \mathbb Z/3\mathbb Z$, $G^{2}_{1,4}\cong\mathbb Z/6\mathbb Z$, and $G=G^{1}_{1,4}\oplus G^{2}_{1,4}$.
Let $s$ be a generator of $G^{1}_{1,4}$.
Since $(B^{1}_{1,4}\cdot B^{1}_{1,4})=4$, the genus of $C^{1}_{1,4}$ is $5$ where $p^{\ast}B^{1}_{1,4}=3C^{1}_{1,4}$.
Since $G_{1,0}\cong\mathbb Z/2\mathbb Z$ and $(B^{1}_{1,4}\cdot B^{2}_{1,4})\not=0$, the only curve of Fix$(s)$ is $C^{1}_{1,4}$. 
By [\ref{bio:2},\ref{bio:7}], this is a contradiction.  
\\

We assume that the numerical class of $B$ is (\ref{248}) of the list in section 6.
We denote $B$ by $2B_{1,0}+4B^{1}_{1,4}+4B^{2}_{1,4}$. 
By Theorem \ref{thm:5}, $G^{i}_{1,4}\cong\mathbb Z/4\mathbb Z$ for $i=1,2$. 
Since $(B^{1}_{1,4}\cdot B^{2}_{1,4})\not=0$, by Theorem \ref{thm:5}, $G=G^{1}_{1,4}\oplus G^{2}_{1,4}$.
Let $s\in G^{1}_{1,4}$ and $t\in G^{2}_{1,4}$ be generators. 
Then non-symplectic involutions of $G$ are $s^{2}$ and $t^{2}$. 
By Theorem \ref{thm:5}, we may assume that Fix$(s^{2})=p^{-1}(B_{1,0})\cup p^{-1}(B^{1}_{1,4})$ and Fix$(t^{2})=p^{-1}(B^{2}_{1,4})$. 
For a symplectic involution $s^{2}\circ t^{2}$, since $X/G$ is smooth, Fix$(s^{2}\circ t^{2})\subset{\rm Fix}(s^{2})\cap{\rm Fix}(t^{2})$.
Since $(C\cdot B^{i}_{1,4})=0$ and $(B^{1}_{1,4}\cdot B^{1}_{1,4})=4$, we get that $p^{-1}(B_{1,0}\cup B^{1}_{1,4})\cap p^{-1}(B^{2}_{1,4})$ are 4 points.
By the fact that the fixed locus of a symplectic involution of a $K3$ surface are 8 isolated points, this is a contradiction.
\\

We assume that the numerical class of $B$ is (\ref{260}) of the list in section 6.
We denote $B$ by $3B_{1,0}+2B^{1}_{1,6}+6B^{2}_{1,6}$. 
By Theorem \ref{thm:5} and $(B^{1}_{1,6}\cdot B^{2}_{1,6})\not=0$, $G=G^{1}_{1,6}\oplus G^{2}_{1,6}$. 
Let $\rho_{1},\rho_{2}\in G$ be generators of $G_{B^{1}_{1,6}}$ and $G_{B^{2}_{1,6}}$ respectively. 
Then $\rho^{2}_{2}$ is a non-symplectic automorphism of order 3 and a generator of $G_{1,0}$. 
Since $(C\cdot C)=-6$ and $|G|=12$, we get that $p^{\ast}C=\sum_{j=1}^{4}3C_{j}$ where $C_{j}$ is a smooth rational curve. 
Then $C_{1},\ldots,C_{4},C^{2}_{1,6}\subset$Fix$(\rho^{2}_{2})$ where $p^{\ast}B^{2}_{1,6}=6C^{2}_{1,6}$.
By [\ref{bio:2},\ref{bio:7}], this is a contradiction. 
\\

We assume that the type of $B$ is (\ref{41}) of the list in section 6.
We denote $B$ by $4B^{1}_{1,0}+4B^{2}_{1,0}+2B_{1,3}+2B_{0,1}$. 
We take the Galois cover $q:{\mathbb P}^{1}\times{\mathbb P}^{1}\rightarrow{\mathbb P}^{1}\times{\mathbb P}^{1}$ whose branch divisor is $4B^{1}_{1,0}+4B^{2}_{1,0}$. 
Since the support of $B$ is simple normal crossing, $q^{\ast}(2B_{1,3}+2B_{0,1})=2B_{4,3}+2B_{0,1}$. 
By Theorem \ref{thm:4}, there is the Galois morphism $g:X\rightarrow{\mathbb P}^{1}\times{\mathbb P}^{1}$ such that the branch divisor is $2B_{4,3}+2B_{0,1}$ and the Galois group is abelian.
Since the numerical class of $2B_{4,3}+2B_{0,1}$ is (\ref{8}), this is a contradiction.

As for the case of (\ref{41}), the numerical class of $B$ is not one of 
(\ref{42},\ref{47},\ref{49},\ref{55},\ref{56},\ref{312},\ref{313},\\
\ref{61},\ref{65},\ref{130},\ref{131},
\ref{160},\ref{161},\ref{132},\ref{162},\ref{133},\ref{134},\ref{137},\ref{138},\ref{135},\ref{92},\ref{87},\ref{94},\ref{202},
\ref{205},\ref{203},\ref{233},\\
\ref{234}) of the list in section 6
by (\ref{8},\ref{7},\ref{10},\ref{8},\ref{25},\ref{22},\ref{28},\ref{22},\ref{22},\ref{200},\ref{184},\ref{213},\ref{214},\ref{201},\ref{215},\\
\ref{241},\ref{241},\ref{248},\ref{248},\ref{260},
\ref{178},
\ref{175},\ref{196},\ref{241},\ref{248},\ref{257},\ref{260},\ref{261}) in order.\\

Therefore, we get Theorem \ref{thm:2}.
\section{Abelian groups of K3 surfaces with smooth quotient}
In this section, first of all, we will show Theorem \ref{thm:42} and \ref{thm:45}. 
Next, we will show Theorem \ref{thm:1}. 
By Section 3, we had that if $X/G$ is ${\mathbb P}^{2}$ or ${\mathbb F}_{n}$, then $G$ is one of ${\mathcal AG}$ as a group. 
\begin{pro}\label{thm:17}
Let $X$ be a $K3$ surface and $G$ be a finite subgroup of Aut$(X)$ such that $X/G$ is a smooth rational surface. 
For a birational morphism $f:X/G\rightarrow {\mathbb F}_{n}$, we get that  $0\leq n\leq 12$. 
\end{pro}
\begin{proof}
Let $f:X/G\rightarrow {\mathbb F}_{n}$ be a birational morphism, $e_{i}$ be the exceptional divisors for $i=1,\ldots,m$, and $B=\sum_{i=1}^{k}b_{i}B_{i}$ be the branch divisor.
Since $X/G$ and ${\mathbb F}_{n}$ are smooth and $f$ is a birational morphism, we get ${\rm Pic}(X/G)=f^{\ast}{\rm Pic}({\mathbb F}_{n})\bigoplus_{i=1}^{m}{\mathbb Z}e_{i}$ and  there are positive integers $a_{i}$ for $i=1,\ldots,m$ such that $K_{X/G}=f^{\ast}K_{{\mathbb F}_{n}}+\sum_{i=1}^{m}a_{i}e_{i}$. 
By Theorem \ref{thm:6}, 
\[ 0=f^{\ast}K_{{\mathbb F}_{n}}+\sum_{i=1}^{m}a_{i}e_{i}+\sum_{i=1}^{k}\frac{b_{i}-1}{b_{i}}B_{i}. \]
Since ${\rm Pic}(X/G)=f^{\ast}{\rm Pic}({\mathbb F}_{n})\bigoplus_{i=1}^{m}{\mathbb Z}e_{i}$, 
at least one of $B_{1},\ldots ,B_{k}$ is not an exceptional divisor of $f$. 
By rearranging if necessary, we assume that $B_{i}$ is not an exceptional divisor of $f$ for $1\leq i\leq u$, and  $B_j$ is an exceptional divisor of $f$ for $u+1\leq j\leq k$.
Then $f_{\ast}B_{i}$ is an irreducible curve on ${\mathbb F}_{n}$ for $1\leq i\leq u$.
Therefore, for $1\leq i\leq u$, there are non-negative integers $c_{i},d_{i},g^{i}_{j}$ such that
\[ B_{i}=f^{\ast}(c_{i}C+d_{i}F)-\sum_{j=1}^mg^i_je_j\ {\rm in\ Pic}(X/G)\]
where $(c_{i},d_{i})=(1,0)$, $(0,1)$, or $d_{i}\geq c_{i}n>0$. 
Since $K_{{\mathbb F}_{n}}=-2C-(n+2)F$ in Pic$({\mathbb F}_{n})$, by Theorem \ref{thm:6}, we get that $2=\sum_{i}\frac{b_{i}-1}{b_{i}}c_{i}$ and $n+2=\sum_{i}\frac{b_{i}-1}{b_{i}}d_{i}$.
In the same way as Theorem \ref{thm:8}, we get this proposition.
\end{proof}
Let $X$ be a $K3$ surface, $G$ be a finite subgroup of Aut$(X)$ such that $X/G$ is smooth, and $f:X/G\rightarrow{\mathbb F}_{n}$ be a birational morphism.
By Proposition \ref{thm:17}, we get $0\leq n\leq 12$. 
By the proof of Proposition \ref{thm:17}, the numerical class of $f_{\ast}B$ is one of the list on Section 3.
Let $B=\sum_{i=1}^{k}b_{i}B_{i}+\sum_{j=k+1}^{l}b_{j}B_{j}$ where $B_{i}$ is not an exceptional divisor of $f$ for $i=1,\ldots,k$ and $B_{j}$ is an exceptional divisor of $f$ for $j=k+1,\ldots,l$. 
Since $(X/G)\backslash \bigcup_{j=k+1}^{l}B_{j}$ is isomorphic to ${\mathbb F}_{n}\backslash \bigcup_{j=k+1}^{l}f(B_{j})$ and $f(B_{j})$ is a point for $j=k+1,\ldots l$,  $(X/G)\backslash \bigcup_{j=k+1}^{l}B_{j}$ is simply connected.
By Theorem \ref{thm:5}, $G$ is generated by $G_{1},\ldots, G_{k}$.
Therefore, as for the case of Hirzebruch surface, we  will guess $G$ from the numerical class of $f_{\ast}B$.
Recall that if $G$ is abelian, then $G_{i}$ is a cyclic group, which is generated by a purely non-symplectic automorphism of order $b_{i}$.
If $f_{\ast}B_{1}=C$, or $F$, then $G$ is generated by $G_{2},\ldots, G_{k}$, and 
If $(f_{\ast}B_{1},f_{\ast}B_{2})=(C,F)$, then $G$ is generated by $G_{3},\ldots, G_{k}$.

Recall that since $X/G$ is a smooth rational, $X/G$ is given by blowups of ${\mathbb F}_{n}$.
Next, we will investigate the relationship between a branch divisor and exceptional divisors of blow-ups.
\begin{lem}\label{thm:9,0}
Let $X$ be a $K3$ surface, and $G\subset{\rm Aut}(X)$ a finite subgroup such that $X/G$ is a smooth rational surface, and $B$ be the branch divisor of the quotient map $p:X\rightarrow X/G$.
For a birational morphism $h:X/G\rightarrow T$ where $T$ is a smooth projective surface,
let $e_{i}$ be the exceptional divisor of $h$ for $i=1,\ldots,m$.
Then for $i=1,\ldots,m$ we have that $h(e_{i})\in{\rm Supp}(h_{\ast}B)$.
\end{lem}
\begin{proof}
Let $e_{1},\ldots,e_{m}$ be the exceptional divisors of $h$. 
Since $X/G$ and $T$ are smooth and $h$ is birational, 
${\rm Pic}(X/G)=h^{\ast}{\rm Pic}(T)\bigoplus_{j=1}^{m}{\mathbb Z}e_{j}$ and there are positive integers $a_{i}$ such that 
\[ K_{X/G}=h^{\ast}K_{T}+\sum_{i=1}^{m}a_{i}e_{i}. \] 
We assume that  $h(e_{i})\not\in{\rm Supp}(h_{\ast}B)$ for some $1\leq i\leq m$. 
For simply, we assume that $i=1$, i.e. $h(e_{1})\not\in{\rm Supp}(h_{\ast}B)$.
Let $B_{1},\ldots,B_{k}$ be irreducible components of $B$ such that $B_{j}$ is not an exceptional divisor of $h$ for $j=1,\ldots,k$. 
Since $h(e_{1})\not\in{\rm Supp}(h_{\ast}B)$, there are integers $c_{j,s}$ such that $B_{j}=h^{\ast}C_{j}+\sum_{s=2}^{m}c_{j,s}e_{s}$, where $C_{j}$ is an irreducible curve in $T$.
By Theorem \ref{thm:6}, we get that 
\[ 0=(h^{\ast}K_{T}+\sum_{i=1}^{m}a_{i}e_{i})+\sum_{j=1}^{k}\frac{b_{j}-1}{b_{j}}(h^{\ast}C_{j}+\sum_{s=2}^{m}c_{j,s}e_{s})+\sum_{j=1}^{m}l_{j}e_{j}\ {\rm in\ Pic}(X/G),\]
 where $l_{j}=0$ or $l_{j}=\frac{d_{j}-1}{d_{j}}$ for some an integer $d_{j}\geq 2$.  Since $a_{i}\geq1$, $c_{j,1}=0$, $l_{j}\geq0$, and ${\rm Pic}(X/G)=h^{\ast}{\rm Pic}(T)\bigoplus_{j=1}^{m}{\mathbb Z}e_{j}$, this is a contradiction.   
\end{proof}
\begin{pro}\label{thm:9}
Let $X$ be a $K3$ surface, $G\subset{\rm Aut}(X)$ a finite subgroup such that the quotient space $X/G$ is smooth, and $B$ be the branch divisor of the quotient morphism $p:X\rightarrow X/G$.
	Let $f:X/G\rightarrow T$ be a birational morphism where $T$ is a smooth surface, 
	$e_1,\ldots,e_m$ be the exceptional divisors of $f$, and $f_{\ast}B:=\sum_{i=1}^ub_i\widetilde{B_i}$ where 
	$\widetilde{B_i}$ is an irreducible curves on $U$ for $i=1,\ldots, u$.
If $\widetilde{B_i}$ is smooth for each $1\leq i\leq u$, then for $1\leq j\leq m$ there are $1\leq s< t\leq u$ such that $f(e_j)\in \widetilde{B_s}\cap \widetilde{B_t}$.
\end{pro}
\begin{proof}
	We set $B=\sum_{i=1}^ub_iB_i+\sum_{j=u+1}^kb_jB_j$ where $B_i$ is not an exceptional divisor of $f$ for $i=1,\ldots,u$, and $B_j$ is  an exceptional divisor of $f$ for $j=u+1,\ldots,k$.
	Then $f_{\ast}B=\sum_{i=1}^ub_if_{\ast}B_i$.
	We assume that $f_{\ast}B_i$ is a smooth curve for $i=1,\ldots ,u$.
By Lemma \ref{thm:9,0}, $f(e_i)\in{\rm supp}(f_{\ast}B)$ for $i=1,\ldots, m$.
%
Let $S:=X/G$, $Z:=\{f(e_{1}),\ldots,f(e_{m})\}:=\{z_1,\ldots,z_v\}\subset T$ where $v:=|\{f(e_{1}),\ldots,f(e_{m})\}|$, $q:{\rm Blow}_ZT\rightarrow T$ be the blow-up, and $E_{i}:=q^{-1}(z_i)$ be the exceptional divisor of $q$ for $1\leq i \leq v$.
	Then there is a birational morphism $g:S\rightarrow {\rm Blow}_ZT$ such that $f=q\circ g$, i.e. the following diagram is commutative:
	$$
	\xymatrix{
		{\rm Blow}_ZT\ar[r]^{q}&T \\
		S\ar[u]^{g}.\ar[ur]_{f}& 
	}
	$$
	By changing the number if necessary,
	we assume that $g(e_i)=E_i$ for $1\leq i\leq v$.
	Then the exceptional divisors of $g$ are $e_{v+1},\ldots,e_m$.
	Since ${\rm Pic}({\rm Blow}_ZT)=q^{\ast}{\rm Pic}(T)\bigoplus_{j=1}^v{\mathbb Z}E_j$ and $f=q\circ g$,
	\[ {\rm Pic}(S)=g^{\ast}{\rm Pic}({\rm Blow}_ZT)\bigoplus_{j=v+1}^{m}{\mathbb Z}e_{j}=\Bigl(f^{\ast}{\rm Pic}(T)\bigoplus_{i=1}^v{\mathbb Z}g^{\ast}E_{i}\Bigr)\bigoplus_{j=v+1}^{m}{\mathbb Z}e_{j}.\]
	Since $K_{{\rm Blow}_ZT}=q^{\ast}K_T+\sum_{j=1}^vE_j$, 
\[ K_S=g^{\ast}K_{{\rm Blow}_ZT}+\sum_{i=v+1}^{m}a'_ie_i=\Bigl(f^{\ast}K_T+\sum_{j=1}^vg^{\ast}E_{i}\Bigr)+\sum_{i=v+1}^{m}a'_ie_i\]
	where $a'_i$ is a positive integer for $i=v+1,\ldots,m$. 
	
	We assume that for some $1\leq i\leq m$, $f(e_i)\not \in f_{\ast}B_s\cap f_{\ast}B_t$ for each $1\leq s<t\leq u$.  
	Since $Z=\{f(e_{1}),\ldots,f(e_v)\}$, we assume that $1\leq i\leq v$.
	for simplicity, we assume that $i=1$.
	In addition, 
	since $f(e_j)\in{\rm supp}(f_{\ast}B)$ for $j=1,\ldots, m$, 
	by changing the number if necessary, we assume that $f(e_1)\in {\rm supp}(f_{\ast}B_1)$, and $f(e_1)\not\in {\rm supp}(f_{\ast}B_j)$ for $2\leq j\leq u$.
	Recall that the exceptional divisors of $q$ are $E_1,\ldots,E_v$, the exceptional divisors of $g$ are $e_{v+1},\ldots,e_m$, and $g(e_i)=E_i$ for $1\leq i\leq v$.
	Since $f=q\circ g$, for $j=1,\ldots,u$
	there are non-negative integers $c_{j,s},c'_{j,t}$ such that
	\[B_j=f^{\ast}f_{\ast}B_{j}-\sum_{s=1}^vc_{j,s}g^{\ast}E_{s}-\sum_{t=v+1}^{m}c'_{j,t}e_{t}\ \ {\rm in\ Pic}(S).\]
	Since $f(e_1)\not \in f_{\ast}B_j$ for $2\leq j\leq u$, we get that $c_{j,1}=0$ for $2\leq j\leq u$.
	Since $f_{\ast}B_1$ is smooth, $c_{1,1}=1$.
	Since $K_S=f^{\ast}K_T+\sum_{j=1}^vg^{\ast}E_i+\sum_{i=v+1}^{m}a'_ie_i$ and $0=K_S+\sum_{i=1}^k\frac{b_i-1}{b_i}B_i$ in Pic$(S)$,
	\begin{equation*}
		\begin{split}
			0=&(f^{\ast}K_T+\sum_{j=1}^vg^{\ast}E_i+\sum_{i=v+1}^{m}a'_ie_i)\\
			&+\sum_{i=1}^u\frac{b_i-1}{b_i}(f^{\ast}f_{\ast}B_{j}-\sum_{s=1}^vc_{j,s}g^{\ast}E_{s}-\sum_{t=v+1}^{m}c'_{j,t}e_{t})\\
			&+\sum_{j=u+1}^k\frac{b_j-1}{b_j}B_j\ \ {\rm in\ Pic}(S).
		\end{split}
	\end{equation*}
	From the coefficient of $g^{\ast}E_1$, we get that $1=\frac{b_{1}-1}{b_{1}}$
	Since $b_1\geq 2$, this is a contradiction.
\end{proof}
Let $X$ be a $K3$ surface, $G$ be a finite subgroup of Aut$(X)$ such that $X/G$ is a smooth rational surface, and $B$ be the branch divisor of the quotient map $p:X\rightarrow X/G$. 
Let $h:X/G\rightarrow T$ be a birational morphism where $T$ is a smooth projective surface, and $e_{1},\ldots,e_{m}$ be the exceptional divisors of $h$. 
We set $h_{\ast}B:=\sum_{j=1}^{l} b_{j}B'_{j}$. 
We write $B=\sum_{i=1}^{l}b_{i}B_{i}+\sum_{j=l+1}^{k}b_{j}B_{j}$ such that $h_{\ast}B_{i}=B'_{i}$ for $i=1,\ldots,l$. 
Then $B_j$ is one of the exceptional divisor of $h$ for $j=l+1,\ldots,k$, and for $i=1,\ldots,l$ there are non-negative integers $c_{i,1},\ldots,c_{i,m}$ such that $B_{i}=h^{-1}_{\ast}B'_{i}-\sum_{t=1}^{m}c_{i,t}e_{t}$.
\begin{mar}
In the above situation,  for $e_{u}$ and $e_{v}$ where $1\leq u<v\leq m$ and $h(e_{u})=h(e_{v})$, we get that $c_{i,u}=0$ if and only if $c_{i,v}=0$.
\end{mar}
\begin{mar}\label{thm:16}
In the situation of Proposition \ref{thm:9}, we assume that $T={\mathbb F}_{n}$.
Then there are positive integers $a_{1},\ldots,a_{m}$ such that $K_{X/G}=h^{\ast}K_{{\mathbb F}_{n}}+\sum_{i=1}^{m}a_{i}e_{i}$.
By the proof of Proposition \ref{thm:9}, we get that $a_{1}=\cdots =a_{u}=1$ and 
\[ 1+\frac{\beta_{i}-1}{\beta_{i}}=\sum_{j=1}^{k}\frac{b_{j}-1}{b_{j}}c_{i,j}\ for\ i=1,\ldots,u,\]
where $\beta_{i}=1$ if $e_{i}$ is not an irreducible component of $B$, and $\beta_{i}$ is the ramification index at $e_{i}$ if $e_{i}$ is an irreducible component of $B$.

Furthermore, we assume that $X/G\not={\rm Blow}_{\{h(e_{1}),\ldots,h(e_{u})\}}{\mathbb F}_{n}$. 
For the birational morphism $g:X/G\rightarrow {\rm Blow}_{\{h(e_{1}),\ldots,h(e_{u})\}}{\mathbb F}_{n}$ in the proof of Proposition \ref{thm:9},
we rearrange the order so that $\{g(e_{u+1}),\ldots,g(e_{u+v})\}=\{g(e_{u+1}),\ldots,g(e_{m})\}$, 
where $v:=|\{g(e_{u+1}),\ldots,g(e_{m})\}|$.
Like the proof of Proposition \ref{thm:9}, by considering the blow-up of ${\rm Blow}_{\{h(e_{1}),\ldots,h(e_{u})\}}{\mathbb F}_{n}$ at $\{g(e_{u+1}),\ldots,g(e_{u+v})\}$, we get that $a_{u+1}=\cdots =a_{u+v}=2$ and 
\[ 2+\frac{\beta_{i}-1}{\beta_{i}}=\sum_{j=1}^{k}\frac{b_{j}-1}{b_{j}}c_{i,j}\ for\ i=u+1,\ldots,u+v,\]
where $\beta_{i}=1$ if $e_{i}$ is not an irreducible component of $B$, and $\beta_{i}$ is the ramification index at $e_{i}$ if $e_{i}$ is an irreducible component of $B$.
\end{mar}
Recall that by Theorem \ref{thm:5}, $G_{B_{i}}$ is generated by a non-symplectic automorphism of order $b_{i}$.
As a corollary of Theorem \ref{thm:5} and Proposition \ref{thm:9}, we get the following Theorem \ref{thm:43}.
\begin{thm}\label{thm:43}
Let $X$ be a $K3$ surface, $G$ be a finite subgroup of Aut$(X)$ such that $X/G$ is smooth, and $B$ be the branch divisor of the quotient map $p:X\rightarrow X/G$.
Let $f:X\rightarrow S$ be the birational morphism where $S$ minimal rational surface. 
We put $f_{\ast}B:=\sum_{i=1}^{k}b_{i}B_{i}$ where $B_{i}$ is an irreducible curve for $i=1,\ldots,k$.
We denote by $G_{s}$ the subgroup of $G$, which consists of symplectic automorphisms of $G$, and $b$  the least common multiple of $b_{1},\ldots,b_{k}$.
Then there is a purely non-symplectic automorphism $g\in G$ of order $b$ such that
$G$ is the semidirect product $G_{s}\rtimes\langle g \rangle$
of $G_{s}$ and $\langle g \rangle$. 
\end{thm}
\begin{proof}
Since $G_{s}$ is a normal subgroup of $G$ and $G/G_{s}$ is a cyclic group, in order to show Theorem \ref{thm:43}, we only show that there is a purely non-symplectic automorphism $g\in G$ of order $b$. 

First of all, we assume that $X/G\cong{\mathbb P}^{2}$. 
We put $B:=\sum_{i=1}^{k}b_{i}B_{i}$ where $B_{i}$ is an irreducible curve for $i=1,\ldots,k$.
By Theorem \ref{thm:6}, $0=\sum_{i=1}^k\frac{b_{i}-1}{b_i}{\rm deg}\,B_{i}+{\rm deg}\,K_{\mathbb P^2}$, in which $K_{\mathbb P^{2}}$ is the canonical line bundle of $\mathbb P^{2}$.
Since the degree of $K_{{\mathbb P}^{2}}$ is $-3$ and $\frac{1}{2}\leq\frac{l-1}{l}<1$ for any positive integer $l$, we get that $4\leq\sum_{i=1}^{k}{\rm deg}B_{i}\leq6$.
If $\sum_{i=1}^{k}{\rm deg}B_{i}=6$, then $b_{1}=\cdots=b_{k}=2$.
By Theorem \ref{thm:5}, in this case the statement of theorem is established. 
We assume that $\sum_{i=1}^{k}{\rm deg}B_{i}\leq5$.
By [\ref{bio:1},\ {\rm Theorem 2}], $b=b_{i}$ for some $1\leq i\leq k$ or $b=l.c.m(b_{i},b_{j})$ for $i<j$.
By Theorem \ref{thm:5}, in the former case, we get this theorem. 

For the latter, i.e. if $b\not=b_{i}$ for $1\leq i\leq k$, then 
$B$ is one of (i)
$3L_{1}+3L_{2}+3L_{3}+2L_{4}+2L_{5}$ 
where $L_{3}$ pass through the points $L_{1}\cap L_{2}$ and $L_{4}\cap L_{5}$ (see [\ref{bio:1},\ pp.\ 408]),
(ii) $3L_{1}+3L_{2}+3L_{3}+2Q$ 
where $L_{1},L_{2}$ are the tangent to $Q$ and $L_{3}$ is in general position
with respect to $L_{1}\cup L_{2}\cup Q$ (see [\ref{bio:1},\ pp.\ 408]), and 
(iii) $2L_{1}+2L_{2}+3L_{3}+3Q$ 
where $L_{1},L_{2},L_{3}$ are three distinct tangent lines to $Q$ (see [\ref{bio:1},\ pp.\ 410]).
Here, $L_{i}$ and $Q$ are smooth curves on ${\mathbb P}^{2}$ with deg\,$L_{i}=1$ and deg\,$Q=2$ for $i=1,\ldots,5$.
Then there are $1\leq i<j\leq k$ such that $b=l.c.m(b_{i},b_{j})$, $B_{i}+B_{j}$ is simple normal crossing, and $(B_{i}\cap B_{j})\backslash\cup_{s\not=i,j}B_{s}$ is not an empty set.
For clarity, we may assume that $i=1$, $j=2$. 
We take one point $y\in (B_{1}\cap B_{2})\backslash\cup_{i=3}^{k}B_{i}$.
Let $x\in p^{-1}(y)$. By the assumption for $y$ and Theorem \ref{thm:3}, there are open subset $V\subset{\mathbb P}^{2}$ and $U\subset X$ 
such that $y\in V$, $x\in U$,  
$p_{|U}:U\rightarrow V$ is isomorphic to $\{z\in{\mathbb C}^{2}:\ |z|<1\}\ni (z_{1},z_{2})\mapsto(z^{b_{1}}_{1},z^{b_{2}}_{2})\in\{z\in{\mathbb C}^{2}:\ |z|<1\}$, and hence $G_{x}:=\{g\in G\,|\,g(x)=x\}\cong\mathbb Z/b_{1}\mathbb Z\oplus\mathbb Z/b_{2}\mathbb Z$.
Since $b=l.c.m(b_{1},b_{2})$, there is a purely non-symplectic automorphism $g\in G$ with order $b$.\\

Next, we assume that $X/G\cong{\mathbb F}_n$. 
By the list of the numerical class of $B$ in Section 6, 
if the numerical class of $B$ is not one of 
(\ref{312},\ref{67,1},\ref{68,1},\ref{72},\ref{80},\ref{145},\ref{130},\ref{86},\ref{268},\ref{267},\ref{133},
\ref{279},\ref{92},\ref{284},\ref{273},\ref{292},\ref{217},\ref{222},\ref{220}),
then $b=b_{i}$ for some $1\leq i\leq k$. 
Therefore, by Theorem \ref{thm:5}, we get this theorem.
If the numerical class of $B$ is one of 
(\ref{312},\ref{67,1},\ref{68,1},\ref{72},\ref{145},\ref{86},\ref{268},\ref{267},\ref{133},\ref{279}, \ref{92},\ref{284},\ref{273},\ref{292},\ref{217},\ref{222},
\ref{220}), then there are $1\leq i<j\leq k$ such that $b=l.c.m(b_{i},b_{j})$, $B_{i}+B_{j}$ is simple normal crossing, and $(B_{i}\cap B_{j})\backslash\cup_{s\not=i,j}B_{s}$ is not an empty set.
As for the case of $\mathbb P^{2}$, we get this theorem.

We assume that the numerical class of $B$ is (\ref{80}).
We write $B=3B_{3,3}+2B_{0,1}^1+2B_{0,1}^2$.
Since $B^{1}_{0,1}\cap B^2_{0,1}$ is an empty set, if $B_{3,3}\cap B^1_{0,1}$ is not one point, then by $(B_{3,3}\cdot B^1_{0,1})=3$, there is a point $y\in B_{3,3}\cap B^{1}_{0,1}$ such that the support of $B$ is simple normal crossing at $y$. 
Since $b=6$, by Theorem \ref{thm:5}, we get this theorem.
Therefore, we assume that $B_{3,3}\cap B^{1}_{0,1}$ and $B_{3,3}\cap B^2_{0,1}$ are one point.
Let $q:X/G_{s}\rightarrow X/G$ be the quotient map.
Then the singular locus of $X/G_s$ is $q^-1(B_{3,3}\cap B^{1}_{0,1})\cup q^-1(B_{3,3}\cap B^2_{0,1})$.
Since the Galois group of $q$ is $G/G_{s}\cong{\mathbb Z}/6{\mathbb Z}$, the branch divisor of $q$ is $B$, and $B_{3,3}\cap B^{1}_{0,1}$ and $B_{3,3}\cap B^2_{0,1}$ are one point, $X/G_{s}$ has just two singular point.
By [\ref{bio:17}, {\rm Theorem}\ 3], this is a contradiction.
Therefore, if the numerical class of $B$ is (\ref{80}), then we get this theorem.
As for the case of (\ref{80}), we get this theorem for (\ref{130}).\\

Finally,  we assume that $X/G$ is not $\mathbb P^2$ or $\mathbb F_n$.
We take a birational morphism $f:X/G\rightarrow{\mathbb F}^{n}$ where $0\leq n$. 
Let $e_{1},\ldots,e_{m}$ be the exceptional divisors of $f$.
In the same way of the case where $X/G\cong\mathbb P^2$ or $\mathbb F_n$, we only consider the case that 
the numerical class of $f_{\ast}B$ is one of  
(\ref{312},\ref{67,1},\ref{68,1},\ref{72},\ref{80},\ref{145},
\ref{130},\ref{86},\ref{268},\ref{267},\ref{133},
\ref{279},
\ref{92},\ref{284},\ref{273},\ref{292},\ref{217},\ref{222},
\ref{220}).

We assume that the numerical class of $f_{\ast}B$ is (\ref{312}).
By Remark \ref{thm:16}, there are positive integers $a_{1},\ldots,a_{5},b$ such that 
\[ 1+\frac{b-1}{b}=\frac{2}{3}a_{1}+\frac{5}{6}a_{2}+\frac{1}{2}a_{3}+\frac{3}{4}a_{4}+\frac{3}{4}a_{5}. \]
Since the numerical class of $f_{\ast}B$ is (\ref{312}), we may assume that $a_{1}$ or $a_{2}$ is 0, and either $a_{4}$ or $a_{5}$ is 0.
However, there are not such positive integers. 
Therefore, the numerical class of $f_{\ast}B$ is not (\ref{312}).
As for the case of (\ref{312}), the numerical class of $B$ is not one of 
(\ref{68,1},\ref{72},\ref{268},\ref{267},\ref{273},\ref{292},\ref{217},\ref{220}).\\

We assume that the numerical class of $f_{\ast}B$ is (\ref{67,1}).
By Remark \ref{thm:16}, there are positive integers $a_{1},\ldots,a_{6},b$ such that 
\[ 1+\frac{b-1}{b}=\frac{1}{2}a_{1}+\frac{2}{3}a_{2}+\frac{5}{6}a_{3}+\frac{1}{2}a_{4}+\frac{3}{4}a_{5}+\frac{3}{4}a_{6}. \]
Since the numerical class of $f_{\ast}B$ is (\ref{67,1}), we may assume that
two of $a_{1}$, $a_{2}$ and $a_{3}$ are 0, and two of $a_{4}$, $a_{5}$ and $a_{6}$ are 0.
The integers satisfying the above conditions is only $(a_{1},\ldots,a_{6},b)=(1,0,0,1,0,0,12)$.
Therefore, for $B:=\sum_{j=1}^{l}B_{j}B_{j}$, $b_{i}=12$ for some $1\leq i\leq l$.
By Theorem \ref{thm:5}, if the numerical class of $f_{\ast}B$ is (\ref{312}), then we get this theorem.
As for the case of (\ref{67,1}), if the numerical class of $B$ is one of (\ref{133},\ref{279}), then we get this theorem.\\

We assume that the numerical class of $B$ is (\ref{80}).
By Remark \ref{thm:16}, there are positive integers $a_{1},\ldots,a_{6},b$ such that 
\[ 1+\frac{b-1}{b}=\frac{2}{3}a_{1}+\frac{1}{2}a_{2}+\frac{1}{2}a_{3}. \]
Since the numerical class of $f_{\ast}B$ is (\ref{80}), we may assume that either $a_{2}$ or $a_{3}$ is $0$.
The integers satisfying the above conditions is  $(a_{1},a_{2},a_{3},b)=(2,1,0,6)$ or $(2,0,1,6)$.
Therefore, we get this of theorem.
As for the case of (\ref{80}), if the numerical class of $B$ is one of (\ref{145},\ref{130},\ref{86},\ref{92},\ref{284},\ref{222}), then we get this theorem.
\end{proof}
\begin{thm}\label{thm:44}
Let $X$ be a $K3$ surface and $G$ be a finite subgroup of Aut$(X)$ such that $X/G$ is smooth.
For a birational morphism $f:X/G\rightarrow{\mathbb F}_{n}$ where $0\leq n$, 
we get that $n$ is not one of $5,7,9,10,11$. 
\end{thm}
\begin{proof}
Let $p:X\rightarrow X/G$ be the quotient map, and $B:=\sum_{i=1}^{k}b_{i}B_{i}$ be the branch divisor of $p$.
Let $f:X/G\rightarrow{\mathbb F}_{n}$ be a birational morphism where $0\leq n$, and $e_{1},\ldots,e_{m}$ be the exceptional divisors of $f$.

First we will show this theorem for the cases where $f$ is an isomorphism, i.e. $X/G\cong{\mathbb F}_{n}$.
By Theorem \ref{thm:6},  $n=0,1,2,3,4,5,6,7,8,9$, or $12$.
We assume that $n=5,7,$ or $9$.
Then the numerical class of $B$ is one of 
(\ref{255},\ref{307},\ref{254},\ref{252},\ref{253},
\ref{297},\ref{309},\ref{265},\ref{311}) of the list in section 6.

We assume that the numerical class of $B$ is (\ref{255}).
We denote $B$ by $4B_{1,0}+2B_{1,5}+4B_{1,6}$.
Let $p^{\ast}B_{1,0}=\sum_{i=1}^{m}4C_{i}$ where $C_{i}$ is a smooth curve for $i=1,\ldots,m$.
Since $(B_{1,0}\cdot B_{1,0})<0$, $(C_{i}\cdot C_{i})<0$.
Since $X$ is a $K3$ surface, and $C_{i}$ is irreducible, we get that $(C_{i}\cdot C_{i})=-2$.
Since the degree of $p$ is $|G|$, and $(B_{1,0}\cdot B_{1,0})=-5$, 
we get that $\frac{-5|G|}{16}=-2m+2\sum_{1\leq i<j\leq m}(C_{i}\cdot C_{j})$,  
and hence $\frac{5|G|}{32}\leq m$.
Let $p^{\ast}B_{1,6}=\sum_{j=1}^{l}4C'_{j}$ where $C'_{j}$ is a smooth curve for $j=1,\ldots,l$.
Since $(B_{1,0}\cdot B_{1,6})=1$, $\frac{|G|}{16}=m(C_{1}\cdot\sum_{j=1}^{l}C'_{j})$.
Since $(C_{1}\cdot\sum_{j=1}^{l}C'_{j})\geq1$, we get that $m\leq \frac{|G|}{16}$.
By $\frac{5|G|}{32}\leq m$ and $m\leq \frac{|G|}{16}$, we get that  the numerical class of $B$ is not (\ref{255}).
As for the case of (\ref{255}), the numerical class of $f_{\ast}B$ is not one of (\ref{307},\ref{254},\ref{252},\ref{253},
\ref{297},\ref{309},\ref{265},\ref{311}).
Therefore, if $X/G\cong{\mathbb F}_{n}$, then $n\not=5,7,9,10,11$.

Next, we assume that $f$ is not an isomorphism, i.e. $X/G$ is not a Hirzebruch surface $\mathbb F_{n}$.
By Theorem \ref{thm:17}, $n=0,1,2,3,4,5,6,7,8,9$, or $12$. 
We assume that $n=5,7,$ or $9$.
The numerical class of $f_{\ast}B$ is one of 
(\ref{255},\ref{307},\ref{254},\ref{252},\ref{253},
\ref{297},\ref{309},\ref{265},\ref{311}). 

We assume that the numerical class of $f_{\ast}B$ is (\ref{255}).
Let $p^{\ast}B_{1,0}=4\sum_{i=1}^{m}C_{i}$ where $C_{i}$ is a smooth curve for $i=1,\ldots,m$.
Since the degree of $p$ is $|G|$, by $(C\cdot F)=1$, we get that $|G|=4m(C_{1}\cdot p^{\ast}f^{\ast}F)$, 
and hence $|G|$ is a multiple of $4m$.
Since $f_{\ast}B_{1,0}=C$, $(B_{1,0},B_{1,0})\leq(C\cdot C)=-5$.
By $\frac{|G|}{16}(B_{1,0}\cdot B_{1,0})=-2m+2\sum_{1\leq i<j\leq m}(C_{i}\cdot C_{j})$,
we get that $m=\frac{|G|}{4}$.
Since the numerical class of $f_{\ast}B$ is (\ref{255}), there must be positive integers $a_{1},a_{2},a_{3},b$ such that 
\[ 1+\frac{b-1}{b}=\frac{3}{4}a_{1}+\frac{1}{2}a_{2}+\frac{3}{4}a_{3}, \]
and either $a_{1}$ or $a_{2}$ is 0.
The integers satisfying the above conditions is only $(a_{1},a_{2},a_{3},b)=(1,0,1,2)$, 
and hence $f(e_{i})\in f_{\ast}B_{1,5}\cap f_{\ast}B_{1,6}$ for each $i=1,\ldots,l$. 
Since $(f_{\ast}B_{1,5}\cdot f_{\ast}B_{1,6})=1$,  $f_{\ast}B_{1,5}\cap f_{\ast}B_{1,6}$ is one point.
We put $x:=f_{\ast}B_{1,5}\cap f_{\ast}B_{1,6}$.
Let $q:{\rm Blow}_{x}{\mathbb F}_{5}\rightarrow{\mathbb F}_{5}$ be the blow-up of ${\mathbb F}_{5}$ at $x$.
Then there is a birational morphism $g:X/G\rightarrow {\rm Blow}_{x}{\mathbb F}_{5}$ such that $f=q\circ g$.
Let $C':=g_{\ast}B_{1,0}$. 
Let $E$ be the exceptional divisor of $q$.
Since $f(e_{i})=x$ for each $i=1,\ldots,l$, $g(e_{i})\in E$ for each $i=1,\ldots,l$.
Since $g_{\ast}B=4C'+2g_{\ast}B_{1,5}+4g_{\ast}B_{1,6}+2E$, if $g$ is not an isomorphism, then
there must be  integers $a_{1},a_{2},a_{3},a_{4},b$ such that 
\[ 2+\frac{b-1}{b}=\frac{3}{4}a_{1}+\frac{1}{2}a_{2}+\frac{3}{4}a_{3}+\frac{1}{2}a_{4}, \]
and if $a_{1}$ is not 0, then either $a_{2}=a_{3}=0$.
However, there are not such positive integers. 
Therefore, $g$ is an isomorphism, i.e. $X/G={\rm Blow}_{x}{\mathbb F}_{5}$, and hence
$B=4B_{1,0}+2B_{1,5}+4B_{1,6}+2E$ and $(B_{1,0}\cdot E)=1$.
We put $p^{\ast}E=2\sum_{j=1}^{u}C'_{j}$ where $C'_{j}$ is a smooth curve for $j=1,\ldots,u$.
Since $m=\frac{|G|}{den}$, $\frac{|G|}{2}=|G|(C_{1}\cdot\sum_{j=1}^{u}C'_{j})$. 
This is a contradiction.
Therefore, the numerical class of $B$ is not (\ref{255}).
As for the case of (\ref{255}), the numerical class of $B$ is not one of (\ref{309},\ref{265}).\\

We assume that the numerical class of $f_{\ast}B$ is (\ref{307}).
Then there must be integers $a_{1},a_{2},a_{3},a_{4},b$ such that 
\[ 1+\frac{b-1}{b}=\frac{3}{4}a_{1}+\frac{1}{2}a_{2}+\frac{3}{4}a_{3}+\frac{3}{4}a_{4}, \]
and if $a_{1}$ is not zero, then $a_{2}=a_{3}=0$.
The integers satisfying the above condition is $(a_{1},a_{2},a_{3},a_{4},b)=(1,0,0,1,2)$ or $(0,0,1,1,2)$.
Therefore, for each $i=1,\ldots,l$, we get that  $f(e_{i})\in f_{\ast}B_{1,0}\cap f_{\ast}B_{0,1}$ or $f(e_{i})\in f_{\ast}B^{2}_{1,5}\cap f_{\ast}B_{0,1}$.
If $f(e_{i})\in f_{\ast}B^{2}_{1,5}\cap f_{\ast}B_{0,1}$ for all $i=1,\ldots,l$, then $(B_{1,0}\cdot B_{1,0})=-5$ and $(B_{1,0}\cdot B_{0,1})=1$.
However, as for the case of $X/G\cong {\mathbb F}_{n}$, we can see that such things can not happen.
Therefore, $f(e_{i})\in f_{\ast}B_{1,0}\cap f_{\ast}B_{0,1}$ for some $i=1,\ldots,l$. 
By using the blow-up of ${\mathbb F}_{5}$ at $x:=f_{\ast}B_{1,0}\cap f_{\ast}B_{0,1}$, 
as for the case of (\ref{255}), this is a contradiction.
Therefore, the numerical class of $B$ is not (\ref{307}).
As for the case of (\ref{307}), the numerical class of $B$ is not (\ref{311}).\\

We assume that the numerical class of $f_{\ast}B$ is (\ref{254}).
Then there must be integers $a_{1},a_{2},a_{3},b$ such that 
\[ 1+\frac{b-1}{b}=\frac{5}{6}a_{1}+\frac{1}{2}a_{2}+\frac{2}{3}a_{3}. \]
The integers satisfying the above condition is only $(a_{1},a_{2},a_{3},b)=(1,0,1,2)$, and hence $f(e_{i})\in f_{\ast}B_{1,5}\cap f_{\ast}B_{1,6}$ for each $i=1,\ldots,l$. 
Since $(f_{\ast}B_{1,5}\cdot f_{\ast}B_{1,6})=1$,  $f_{\ast}B_{1,5}\cap f_{\ast}B_{1,6}$ is one point.
We put $x:=f_{\ast}B_{1,5}\cap f_{\ast}B_{1,6}$.
Let $q:{\rm Blow}_{x}{\mathbb F}_{5}\rightarrow{\mathbb F}_{5}$ be the blow-up of ${\mathbb F}_{5}$ at $x$.
As for the case of (\ref{255}), since  there are no integers $a_{1},a_{2},a_{3},a_{4},b$ such that 
\[ 2+\frac{b-1}{b}=\frac{3}{4}a_{1}+\frac{1}{2}a_{2}+\frac{3}{4}a_{3}+\frac{1}{2}a_{4}, \]
we get that $X/G={\rm Blow}_{x}{\mathbb F}_{5}$, and hence
$B=6B_{1,0}+2B_{1,6}+3B_{1,6}+2E$, and $(B_{1,0}\cdot E)=1$.
We put $p^{\ast}E=2\sum_{j=1}^{u}C'_{j}$ where $C'_{j}$ is a smooth curve for $j=1,\ldots,u$.
Since $(E\cdot E)=-1$, we get that $u=\frac{|G|}{4}+\sum_{1\leq i<j\leq u}(C'_{i}\cdot C'_{j})$, and hence $u\geq\frac{|G|}{4}$.
Since $(B_{1,0}\cdot E)=1$, $\frac{|G|}{12}$ is a multiple of $u$. 
This is a contradiction.
Therefore, the numerical class of $B$ is not (\ref{254}).\\

We assume that the numerical class of $f_{\ast}B$ is (\ref{252}).
Then there must be positive integers $a_{1},a_{2},a_{3},a_{4},b$ such taht 
\[ 1+\frac{b-1}{b}=\frac{5}{6}a_{1}+\frac{1}{2}a_{2}+\frac{2}{3}a_{3}+\frac{2}{3}a_{4}, \]
and $a_{1}a_{3}=0$.
The integers satisfying the above conditions is $(a_{1},a_{2},a_{3},a_{4},b)=(1,0,0,1,2)$ or $(0,1,1,1,6)$.
Therefore, for each $i=1,\ldots,l$, we get that  $f(e_{i})\in f_{\ast}B_{1,0}\cap f_{\ast}B_{0,1}$ or $f(e_{i})\in f_{\ast}B_{1,6}\cap f_{\ast}B_{1,5}\cap f_{\ast}B_{0,1}$.
If $f(e_{i})\in f_{\ast}B_{1,6}\cap f_{\ast}B_{1,5}\cap f_{\ast}B_{0,1}$ for all $i=1,\ldots,l$, then $(B_{1,0}\cdot B_{1,0})=-5$ and $(B_{1,0}\cdot B_{0,1})=1$.
We get that this is not established in the same way as in the case of $X/G\cong {\mathbb F}_{n}$.
By using the blow-up of ${\mathbb F}_{5}$ at $x:=f_{\ast}B_{1,0}\cap f_{\ast}B_{0,1}$, 
as for the case of (\ref{254}), we get that there is no case where $f(e_{i})\in f_{\ast}B_{1,0}\cap f_{\ast}B_{0,1}$ for some $i=1,\ldots,l$.
Therefore, the numerical class of $B$ is not (\ref{252}).
As for the case of (\ref{252}), the numerical class of $B$ is not one of (\ref{253},\ref{297}).
\end{proof}
\begin{cro}\label{thm:46}
Let $X$ be a $K3$ surface and $G$ be a finite subgroup of Aut$(X)$ such that $X/G$ is smooth.
If there is a birational morphism $f:X/G\rightarrow{\mathbb F}_{n}$ from the quotient space $X/G$ to a Hirzebruch surface ${\mathbb F}_{n}$ where $n=6,8$, or $12$, then $f$ is an isomorphism, i.e. $X/G$ is a Hirzebruch surface.
\end{cro}
\begin{proof}
Let $n\geq1$ and $C_{-n}\subset{\mathbb F}_{n}$ be the unique irreducible curve such that $(C_{-n}\cdot C_{-n})=-n$.
Since for $x\in {\mathbb F}_{n}$,
if $x\in C_{-n}$, then ${\rm Blow}_{x}{\mathbb F}_{n}={\rm Blow}_{y}{\mathbb F}_{n+1}$ where $y\in{\mathbb F}_{n+1}\backslash C_{-(n+1)}$, and 
if $x\not\in C_{-n}$, then ${\rm Blow}_{x}{\mathbb F}_{n}={\rm Blow}_{y}{\mathbb F}_{n-1}$ where $y\in C_{-(n-1)}$,
by Theorem \ref{thm:44}, we get this corollary.
\end{proof}
\begin{thm}\label{thm:66}
Let $X$ be a $K3$ surface and $G$ be a finite abelian subgroup of ${\rm Aut}(X)$. 
If $X/G$ is smooth, then $G$ is isomorphic to one of ${\mathcal AG}$ as groups.
\end{thm}
\begin{proof}
Since $X/G$ is smooth, the quotient space $X/G$ is an Enriques surface or a rational surface. 
If $X/G$ is Enriques, then $G\cong{\mathbb Z}/2{\mathbb Z}$ as a group and ${\mathbb Z}/2{\mathbb Z}\in{\mathcal AG}$.
By Section 3, if $X/G\cong{\mathbb F}_{n}$, then $G$ is isomorphic to one of ${\mathcal AG}$ as a group.
By [\ref{bio:1}], if $X/G\cong{\mathbb P}^{2}$, then $G$ is isomorphic to one of ${\mathcal AG}$ as a group.
Therefore, we assume that $X/G$ is rational, and $X/G\not={\mathbb P}^{2}$ or ${\mathbb F}_{n}$.

Let $f:X/G\rightarrow{\mathbb F}_{n}$ be a birational morphism where $0\leq n\leq12$, and $B$ be the branch divisor of $G$.
By Theorem \ref{thm:44} and Corollary \ref{thm:46}, we may assume that $0\leq n \leq 4$.
By the proof of Theorem \ref{thm:17}, the numerical class of $f_{\ast}B$ is one of the list in section 6.\\

We assume that the numerical class of $f_{\ast}B$ is one of (\ref{2},\ref{16},\ref{53},\ref{33},\ref{66},\ref{63},\ref{3},\ref{34},\ref{64},\ref{31},\ref{32},\\
\ref{8},\ref{9},\ref{10},\ref{11},\ref{20},\ref{23},\ref{24},\ref{25},\ref{26},\ref{29},\ref{30},\ref{42},\ref{43},\ref{44},\ref{49},\ref{54},\ref{55},\ref{56},\ref{74},\ref{88},\ref{144},\ref{163},\ref{81},\ref{78},\ref{93},\ref{89},\ref{96},\ref{148},\\
\ref{166},\ref{147},\ref{82},\ref{99},\ref{101},\ref{102},\ref{132},\ref{98},\ref{103},\ref{151},\ref{152},\ref{156},\ref{153},\ref{157},\ref{162},\ref{164},\ref{165},\ref{85},\ref{85,1},\ref{87},\ref{126},\\
\ref{125},\ref{124},\ref{123},\ref{129},\ref{128},\ref{94},\ref{127},\ref{142},\ref{143},\ref{159},\ref{158},\ref{169},\ref{176},\ref{177},\ref{()},\ref{216},\ref{171},\ref{183},\ref{212},\ref{201},\\
\ref{215},\ref{175},\ref{174},\ref{196},\ref{197},\ref{199},\ref{198},\ref{210},\ref{211},\ref{219},\ref{301},\ref{231},\ref{232},\ref{302},\ref{235},\ref{303},\ref{236},\ref{239},
\ref{250}) of\\
the list in section 6.
By Theorem \ref{thm:5}, $G$ is generated by automorphisms $g_{1},\ldots,g_{m}$ where $1\leq m\leq5$ and the order of $g_{i}$ is two for $i=1,\ldots,m$.
Therefore, $G$ is ${\mathbb Z}/2{\mathbb Z}^{\oplus a}$ where $1\leq a \leq5$ as a group.\\
	
We assume that the numerical class of $f_{\ast}B$ is one of
(\ref{1},\ref{14},\ref{50},\ref{12},\ref{13},\ref{5},\ref{6},\ref{7},\ref{27},\ref{46},\\
\ref{47},\ref{167},\ref{209},\ref{170},\ref{180},\ref{182},\ref{200},\ref{168},\ref{173},\ref{178},\ref{179},\ref{238},\ref{304},\ref{247},\ref{306}) of the list in section 6.
By Theorem \ref{thm:5}, $G$ is generated by automorphisms $g_{1},\ldots,g_{m}$ where $1\leq m\leq3$ and the order of $g_{i}$ is 3 for $i=1,\ldots,m$.
Therefore, $G$ is ${\mathbb Z}/3{\mathbb Z}^{\oplus b}$ where $1\leq b \leq3$ as a group.\\

We assume that the numerical class of $f_{\ast}B$ is one of	
(\ref{17},\ref{22},\ref{28},\ref{40},\ref{37},\ref{38},\ref{45},\ref{48},\ref{58},\ref{59},\\
\ref{60},\ref{313},\ref{61},\ref{65},\ref{67},\ref{67,2},\ref{72},\ref{80},\ref{77},\ref{145},\ref{91},\ref{130},\ref{149},\ref{154},\ref{160},\ref{86},\ref{268},\ref{106},\ref{104},\ref{105},\ref{270},\ref{134},\ref{110},\\
\ref{108},\ref{109},\ref{107},\ref{135},\ref{272},\ref{92},\ref{284},\ref{113},\ref{116},\ref{114},\ref{136},\ref{111},\ref{115},\ref{112},\ref{278},\ref{277},\ref{285},\ref{286},\ref{287},\ref{181},\\
\ref{213},\ref{288},\ref{202},\ref{187},\ref{185},\ref{186},\ref{203},\ref{289},\ref{293},\ref{191},\ref{189},\ref{190},\ref{188},\ref{204},\ref{291},\ref{217},\ref{222},\ref{221},\ref{299},\ref{223},\\
\ref{224},\ref{233},\ref{227},\ref{226},\ref{234},\ref{225},\ref{294},\ref{240},\ref{241},\ref{305},\ref{242},\ref{244},\ref{243},\ref{296}) of the list in section 6.
By Theorem \ref{thm:5}, $G$ is generated by automorphisms $g_{i},\ldots,g_{m}$,
$h_{1},\ldots h_{n}$, where $1\leq m\leq3$, $1\leq n\leq 2$, the order of $g_{i}$ is 2 for $i=1,\ldots,m$, and the order of $h_{j}$ is 3 for $j=1,\ldots,n$. 
Therefore, $G$ is ${\mathbb Z}/2{\mathbb Z}^{\oplus d}\oplus {\mathbb Z}/3{\mathbb Z}^{\oplus e}$ where (d,e)=(1,1),(1,2),(1,3),(2,1),(2,2),
(3,1),(3,2) as a group.\\

We assume that the numerical class of $f_{\ast}B$ is one of
(\ref{15},\ref{51},\ref{52},\ref{62},\ref{4},\ref{18},\ref{19},\ref{21},\ref{39},\ref{41},\ref{35},\\
\ref{36},\ref{57},\ref{79},\ref{95},
\ref{146},\ref{75},\ref{76},\ref{97},\ref{100},\ref{131},\ref{150},\ref{155},\ref{83},\ref{84},\ref{117},\ref{119},\ref{118},\ref{137},\ref{122},\ref{120},\ref{121},\ref{139},\\
\ref{140},\ref{282},\ref{172},\ref{208},\ref{206},\ref{207},
\ref{184},\ref{214},\ref{205},\ref{193},\ref{195},\ref{192},\ref{194},
\ref{298},\ref{300},\ref{230},\ref{228},\ref{229},\ref{295},\ref{237},\\
\ref{245},\ref{249},\ref{248},\ref{246}) of the list in section 6.
By Theorem \ref{thm:5}, $G$ is generated by automorphisms $g_{i},\ldots,g_{m}$,
$h_{1},\ldots h_{n}$ where the order of $g_{i}$ is 2 for $i=1,\ldots,m$,  the order of $h_{j}$ is 4 for $j=1,\ldots,n$, and $(n,m)$ is one of (0,1),(0,2),
(0,3),
(1,1),(1,2),
(2,1),(3,1).
Therefore, $G$ is ${\mathbb Z}/2{\mathbb Z}^{\oplus f}\oplus {\mathbb Z}/4{\mathbb Z}^{\oplus g}$ where 
(f,g)=(0,1),(0,2),
(0,3),(1,1),
(1,2),(2,1),(3,1) as a group.\\

We assume that the numerical class of $f_{\ast}B$ is (\ref{312}) of the list in section 6.
We denote $B$ by $3B^{1}_{1,0}+6B^{2}_{1,0}+2B_{1,1}+4B^{1}_{0,1}+4B^{2}_{0,1}+\sum_{j=1}^{l}b'_{i}B'_{i}$ where $f_{\ast}B^{i}_{1,0}=(1,0)$, $f_{\ast}B^{i}_{0,1}=(0,1)$ in Pic$({\mathbb P}^{1}\times{\mathbb P}^{1})$, and $B'_{j}$ is an exceptional divisor of $f$ for $j=1,\ldots,l$.
By Theorem \ref{thm:5}, $G\cong{\mathbb Z}/2{\mathbb Z}^{\oplus i}\oplus{\mathbb Z}/3{\mathbb Z}\oplus{\mathbb Z}/4{\mathbb Z}$ where $i=$0 or 1. 
If $G\cong{\mathbb Z}/2{\mathbb Z}\oplus{\mathbb Z}/3{\mathbb Z}\oplus{\mathbb Z}/4{\mathbb Z}$, then $G$ is one of ${\mathcal AG}$ as a group.
We assume that $G\cong{\mathbb Z}/3{\mathbb Z}\oplus{\mathbb Z}/4{\mathbb Z}$.
By Remark \ref{thm:16},  there are integers $\beta,a_{j}\geq0$ such that 
\[ 1+\frac{\beta-1}{\beta}=\frac{5}{6}a_{1}+\frac{1}{2}a_{2}+\frac{2}{3}a_{3}+\frac{11}{12}a_{4}+\frac{11}{12}a_{5}.\]
Since $G\cong{\mathbb Z}/3{\mathbb Z}\oplus{\mathbb Z}/4{\mathbb Z}$,  $\beta$=1,2,3,4,6, or 12. 
Since $f_{\ast}B=3(1,0)+6(1,0)+2(1,1)+4(0,1)+4(0,1)$, the support of $f_{\ast}B$ is simple normal crossing.
Since each irreducible component of $f_{\ast}B$ is smooth,  $a_{j}=0$ or $1$ for each $1\leq j\leq 5$. 
Since $f_{\ast}B=3(1,0)+6(1,0)+2(1,1)+4(0,1)+4(0,1)$,
the non-zero element of $\{a_{1},a_{2}\}$ is just one, and the non-zero element of $\{a_{4},a_{5}\}$ is just one.
The integers which satisfy the above condition are $(\beta,a_{1},a_{2},a_{3})=(12,1,0,1)$ and $(a_{4},a_{5})=(1,0)$ or (0,1). 
Therefore, $f(e_{i})\not\in f_{\ast}B^{2}_{1,0}$ for $i=1,\ldots,l$. 
By the fact that $f_{\ast}B^{2}_{1,0}=(1,0)$ and $f_{\ast}B_{1,1}=(1,1)$ in Pic$({\mathbb P}^{1}\times{\mathbb P}^{1})$ and the fact that 
$f(e_{i})\not\in f_{\ast}B^{2}_{1,0}$ for $i=1,\ldots,l$, 
we get that $B^{2}_{1,0}\cap B_{1,1}$ is not an empty set, and hence
$p^{-1}(B^{2}_{1,0})\cap p^{-1}(B_{1,1})$ is an empty set.
Since $G\cong{\mathbb Z}/3{\mathbb Z}\oplus{\mathbb Z}/4{\mathbb Z}$,  
the number of subgroup of $G$ which is generated by a non-symplectic automorphism of order 2 is one. Since each ramification index of $B^{2}_{1,0}$ and $B_{1,1}$ is divided by 2, by Theorem \ref{thm:5}, 
there is a non-symplectic automorphism $g$ of order 2 such that Fix$(g)\supset f^{-1}B^{2}_{1,0}$ and Fix$(g)\supset f^{-1}B_{1,1}$. Since $p^{-1}(B^{2}_{1,0})\cap p^{-1}(B_{1,1})\not=\emptyset$, this is a contradiction.
Therefore, if the numerical class of $f_{\ast}B$ is (\ref{312}), then $G$ is one of ${\mathcal AG}$ as a group.

As for the case of (\ref{312}), if the numerical class of $f_{\ast}B$ is one of (\ref{90},\ref{133},\ref{271},\ref{281},
\ref{292},\ref{290},\ref{220}) of the list in section 6, then $G$ is  one of ${\mathcal AG}$ as a group.\\

We assume that the numerical class of $f_{\ast}B$ is (\ref{67,1}) of the list in section 6.
We denote $B$ by $2B^{1}_{1,0}+3B^{2}_{1,0}+6B^{3}_{1,0}+2B^{1}_{0,1}+4B^{2}_{0,1}+4B^{3}_{0,1}+\sum_{j=1}^{l}b'_{i}B'_{i}$ where $f_{\ast}B^{i}_{1,0}=(1,0)$, $f_{\ast}B^{i}_{0,1}=(0,1)$, and $B'_{j}$ is an exceptional divisor of $f$ for $j=1,\ldots,l$.
By Theorem \ref{thm:5}, $G\cong{\mathbb Z}/2{\mathbb Z}^{\oplus i}\oplus{\mathbb Z}/3{\mathbb Z}\oplus{\mathbb Z}/4{\mathbb Z}$ where $i=0,1$,or 2. 
There are some integers $\beta,a_{j}$ such that 
\[1+\frac{\beta-1}{\beta}=\frac{1}{2}a_{1}+\frac{2}{3}a_{2}+\frac{5}{6}a_{3}+\frac{1}{2}a_{4}+\frac{3}{4}a_{5}+\frac{3}{4}a_{6}.\]
Since $G\cong{\mathbb Z}/2{\mathbb Z}^{\oplus i}\oplus{\mathbb Z}/3{\mathbb Z}\oplus{\mathbb Z}/4{\mathbb Z}$ where $i=0,1$, or 2,
we get $\beta$=1,2,3,4,6, or 12.  
Since $f_{\ast}B=2(1,0)+3(1,0)+6(1,0)+2(0,1)+4(0,1)+4(0,1)$, the support of $f_{\ast}B$ is simple normal crossing. 
Since each irreducible component of $f_{\ast}B$ is smooth,
$a_{j}=0$ or $1$ for each $1\leq j\leq 6$, and by Theorem \ref{thm:9} the non-zero element of $\{a_{1},a_{2},a_{3}\}$ is just one, and the non-zero element of $\{a_{4},a_{5},a_{6}\}$ is just one.
From the above, $(\beta,a_{1},a_{2},a_{3},b_{1},b_{2},b_{3})=(1,1,0,0,1,0,0)$.
Therefore, $f(e_{j})\in f_{\ast}(B^{1}_{1,0})\cap f_{\ast}(B^{1}_{0,1})$ for $j=1,\ldots,l$.
Since $((1,0)\cdot(1,0))=0$, 
we get that $(p^{\ast}B^{i}_{1,0}\cdot p^{\ast}B^{i}_{1,0})=0$ for $i=2,3$. 
Since $X$ is a $K3$ surface, 
the support of $p^{\ast}B^{i}_{1,0}$ is a union of elliptic curves for $i=2,3$.
Since $G\cong{\mathbb Z}/2{\mathbb Z}^{\oplus i}\oplus{\mathbb Z}/3{\mathbb Z}\oplus{\mathbb Z}/4{\mathbb Z}$ where $i=0,1$, or 2,
the number of subgroups of $G$ which are generated by a non-symplectic automorphism of order 3 is one, 
and hence there is a non-symplectic automorphism $g$ of order 3 such that Fix$(g)$ has at least two elliptic curves. 
By [\ref{bio:2},\ref{bio:7}], this is a contradiction.
Therefore, the numerical class of $f_{\ast}B$ is not (\ref{67,1}).

As for the case of (\ref{67,1}),  the numerical class of $f_{\ast}B$ is not one of (\ref{70},\ref{279}) of the list in section 6.
\\

If the numerical class of $f_{\ast}B$ is (\ref{68}) of the list in section 6,
then  by Theorem \ref{thm:5}, $G\cong{\mathbb Z}/2{\mathbb Z}^{\oplus i}\oplus{\mathbb Z}/4{\mathbb Z}^{\oplus j}$ where $(i,j)$ is one of 
(0,1), (0,2), (1,1), (1,2), (2,1), (2,2), (3,1).
We assume that $G\cong{\mathbb Z}/2{\mathbb Z}^{\oplus 2}\oplus{\mathbb Z}/4{\mathbb Z}^{\oplus 2}$. Since $G$ is generated by non-symplectic automorphism of order $2$ and $4$, $G_{s}:=\{g\in G:\ g\ {\rm is\ symplectic}\}\cong{\mathbb Z}/2{\mathbb Z}^{\oplus 2}\oplus{\mathbb Z}/4{\mathbb Z}$. 
By the classification of finite symplectic groups ([\ref{bio:12},\ref{bio:13},\ref{bio:17}]),
we see that there is no $G_{s}$ where $G_{s}\cong{\mathbb Z}/2{\mathbb Z}^{\oplus 2}\oplus{\mathbb Z}/4{\mathbb Z}$. 
Therefore, $G\cong{\mathbb Z}/2{\mathbb Z}^{\oplus i}\oplus{\mathbb Z}/4{\mathbb Z}^{\oplus j}$ where $(i,j)$ is one of (0,1), (0,2), (1,1), (1,2), (2,1), (3,1), and if the numerical class of $f_{\ast}B$ is (\ref{68}), then $G$ is one of ${\mathcal AG}$ as a group.

As for the case of (\ref{68}), if the numerical class of  $f_{\ast}B$ one of (\ref{69},\ref{73},\ref{161},\ref{138}) of the list in section 6, then $G$ is one of ${\mathcal AG}$ as a group.\\

We assume that the numerical class of $f_{\ast}B$  is (\ref{68,1}) of the list in section 6.
We denote $B$ by $2B^{1}_{1,0}+4B^{2}_{1,0}+4B^{3}_{1,0}+3B^{1}_{0,1}+3B^{2}_{0,1}+3B^{3}_{0,1}+\sum_{j=1}^{l}b'_{i}B'_{i}$ 
By Theorem \ref{thm:5}, $G\cong{\mathbb Z}/2{\mathbb Z}^{\oplus i}\oplus{\mathbb Z}/3{\mathbb Z}\oplus{\mathbb Z}/4{\mathbb Z}$ where $i=0,1$, or 2. 
As for the case of (\ref{65}),  there are integers $\beta,a_{j}$ such that 
\[ 1+\frac{\beta-1}{\beta}=\frac{1}{2}a_{1}+\frac{3}{4}a_{2}+\frac{3}{4}a_{3}+\frac{2}{3}a_{4}+\frac{2}{3}a_{5}+\frac{2}{3}a_{6},\]
and $a_{j}=0$ or $1$ for each $1\leq j\leq 6$, $\beta=1,2,3,4,6$, or 12, the non-zero element of $\{a_{1},a_{2},a_{3}\}$ is only one, and the non-zero element of $\{a_{4},a_{5},a_{6}\}$ is only one, however, integers which satisfy the above condition do not exist.
Therefore, the numerical class of $f_{\ast}B$ is not (\ref{68,1}).

As for the case of (\ref{68,1}),  the numerical class of $f_{\ast}B$ is not one of 
(\ref{141},\ref{269},\ref{267},\\
\ref{283},\ref{280},\ref{273},\ref{274},\ref{275},\ref{276}) of the list in section 6.
\\

We assume that the numerical class of $f_{\ast}B$ is (\ref{71}) of the list in section 6.
We denote $B$ by $2B^{1}_{1,0}+4B^{2}_{1,0}+4B^{3}_{1,0}+2B^{1}_{0,1}+2B^{2}_{0,1}+2B^{3}_{0,1}+2B^{4}_{0,1}+\sum_{i=1}^{n}b'_{i}B'_{i}$, where $f_{\ast}B^{i}_{1,0}=(1,0)$, $f_{\ast}B^{i}_{0,1}=(0,1)$ in Pic$({\mathbb P}^{1}\times{\mathbb P}^{1})$, and $f_{\ast}B'_{i}=0$. 
By Theorem \ref{thm:5}, $G\cong{\mathbb Z}/2{\mathbb Z}^{\oplus i}\oplus{\mathbb Z}/4{\mathbb Z}$, where $i=0,1,2,3,$ or 4.
We assume that  $G\cong{\mathbb Z}/2{\mathbb Z}^{\oplus 4}\oplus{\mathbb Z}/4{\mathbb Z}$. 
By Theorem \ref{thm:5}, $G=G^{1}_{1,0}\oplus G^{2}_{1,0}\oplus G^{1}_{0,1}\oplus G^{2}_{0,1}\oplus G^{3}_{0,1}$.
As for the case of $(\ref{67,1})$, 
we get that  $f(e_{i})\in B^{1}_{1,0}\cap B^{j}_{0,1}$ for each $i=1,\ldots,m$ where $j=1,2,3,4$.
Therefore, we get $(B^{2}_{1,0}\cdot B^{j}_{0,1})=1$.
Let $s\in G^{2}_{1,0}$ be a generator.
Since $G=G^{1}_{1,0}\oplus G^{2}_{1,0}\oplus G^{1}_{0,1}\oplus G^{2}_{0,1}\oplus G^{3}_{0,1}$, by Theorem \ref{thm:5}, there is a non-symplectic automorphism  $t\in G^{j}_{0,1}$ for some $j=1,2,3$ such that Fix$(t\circ s)$ does not have a curve. 
Since $(B^{2}_{1,0}\cdot B^{j}_{0,1})=1$ and $|G|=2^{3}4$, we get that $|p^{-1}(B^{2}_{1,0})\cap p^{-1}(B^{j}_{0,1})|=8$.
By [\ref{bio:9},\ {\rm Proposition 1}], the number of isolated points of Fix$(t\circ s)$ is 4. 
This is a contradiction.  
Therefore, if the numerical class of $f_{\ast}B$ is (\ref{71}), then $G\cong{\mathbb Z}/2{\mathbb Z}^{\oplus i}\oplus{\mathbb Z}/4{\mathbb Z}$ where $i=0,1,2$, or 3, and hence $G$ is one of ${\mathcal AG}$ as a group.\\

We assume that the numerical class of $f_{\ast}B$ is (\ref{271}) of the list in section 6.
We denote $B$ by $3B_{1,0}+2B^{1}_{1,1}+6B^{2}_{1,1}+4B^{1}_{0,1}+12B^{2}_{0,1}+\sum_{j=1}^{l}b'_{i}B'_{i}$ where $f_{\ast}B^{i}_{s,t}=sC+tF$, in Pic$({\mathbb F}_{1})$, and $B'_{j}$ is an exceptional divisor of $f$ for $j=1,\ldots,l$.
By Theorem \ref{thm:5}, $G\cong{\mathbb Z}/2{\mathbb Z}^{\oplus i}\oplus{\mathbb Z}/3{\mathbb Z}\oplus{\mathbb Z}/4{\mathbb Z}$ where $i=0$ or $1$. 
Then the number of subgroup of $G$ which is generated by a non-symplectic automorphism of order 3 is one.
By the above, for $e_{i}$, there are integers $\beta,a_{j}\geq0$ such that 
\[ 1+\frac{\beta-1}{\beta}=\frac{2}{3}a_{1}+\frac{1}{2}a_{2}+\frac{5}{6}a_{3}+\frac{3}{4}a_{4}+\frac{11}{12}a_{5}.\]
Since $G\cong{\mathbb Z}/3{\mathbb Z}\oplus{\mathbb Z}/4{\mathbb Z}$,  $\beta$=1,2,3,4,6, or 12. 
Since $f_{\ast}B=3(1,0)+6(1,0)+2(1,1)+4(0,1)+4(0,1)$, the support of $f_{\ast}B$ is simple normal crossing.
Since each irreducible component of $f_{\ast}B$ is smooth,  $a_{j}=0$ or $1$ for each $1\leq j\leq 5$. 
The integers which satisfy the above condition are $(\beta,a_{1},a_{2},a_{3},a_{4},a_{5})=(4,0,0,1,0,1)$. 
Therefore, $f(e_{i})\not\in f_{\ast}B_{1,0}\cap f_{\ast}B^{2}_{0,1}$ for $i=1,\ldots,l$, and hence
$p^{-1}(B_{1,0})\cap p^{-1}(B^{2}_{0,1})$ is not an empty set.
Since $G_{1,0}\cong{\mathbb Z}/3{\mathbb Z}$, $G^{2}_{0,1}\cong{\mathbb Z}/12{\mathbb Z}$, and $p^{-1}(B_{1,0})\cap p^{-1}(B^{2}_{0,1})$ is not an empty set, we get that 
the number of subgroup of $G$ which is generated by a non-symplectic automorphism of order 3 is at least two. 
This is a contradiction.
Therefore, the numerical class of $f_{\ast}B$ is not (\ref{271}).

As for the case of (\ref{271}), the numerical class of $f_{\ast}B$ is not (\ref{281}) of the list in section 6.
\end{proof}
\section{Abelian groups of Enriques surfaces with smooth quotient}
Let $E$ be an Enriques surface and $H$ be a finite abelian subgroup of Aut$(E)$ such that $E/H$ is smooth.
Let $X$ be the $K3$-cover of $E$, and $G:=\{ s\in{\rm Aut}(X):s\ {\rm is\ a\ lift\ of\ some}\ h\in H\}$. 
Then $G$ is a finite abelian group, $G$ has a non-symplectic involution whose fixed locus is empty, $X/G=E/H$, and the branch divisor of $G$ is that of $H$.
\begin{thm}
Let $E$ be an Enriques surface and $H$ be a finite subgroup of Aut$(E)$. 
We assume that the quotient space $E/H$ is smooth and there is a birational morphism from $E/H$ to a Hirzebruch surface ${\mathbb F}_{n}$, where $0\leq n$. Then $0\leq n\leq 4$. 	
\end{thm}
\begin{proof}
Let $f:E/H\rightarrow{\mathbb F}_{n}$ be a birational morphism, and $B:=\sum_{i=1}^{k}b_{i}B_{i}$ be the branch divisor of the quotient map $E\rightarrow E/H$. 
Since the canonical line bundle of an Enriques surface is numerically trivial, by Theorem \ref{thm:6}, the numerical class of $f_{\ast}B$ is one of  Section 3.
By [\ref{bio:10}, Proposition 4.5], $G$ does not have a non-symplectic automorphism whose order is odd.
Therefore,  $b_{i}$ is even number for each $i=1,\ldots,k$ by Theorem \ref{thm:5}.
By the list of the numerical class of  Section 3, we get the claim.
\end{proof}
\begin{thm}\label{thm:25}
For each numerical classes (\ref{53},\ref{51},\ref{52},\ref{66},\ref{63},\ref{62},\ref{64},\ref{89},\ref{95},\ref{96},\ref{148},\ref{166},\ref{146},\ref{147},
\ref{141},\ref{183},\ref{212},\ref{206},\ref{207},\ref{249}) of the list in section 6, 
there are an Enriques surface $E$ and a finite abelian subgroup $H$ of Aut$(E)$ such that $E/H$ is a Hirzebruch surface ${\mathbb F}_{n}$, and 
the numerical class of the branch divisor $B$ of the quotient map $E\rightarrow E/H$ is (\ref{53},\ref{51},\ref{52},\ref{66},\ref{63},\ref{62},\ref{64},\ref{89},\ref{95},\ref{96},\ref{148},\ref{166},\ref{146},\ref{147},\ref{141},\ref{183},\ref{212},\ref{206},\ref{207},\ref{249}).

Furthermore, for a pair $(E,H)$ of an Enriques surface $E$ and a finite abelian subgroup $H$ of Aut$(E)$, if $E/H\cong{\mathbb F}_{n}$ and the numerical class of the branch divisor $B$ of the quotient map $E\rightarrow E/H$ is 
(\ref{53},\ref{51},\ref{52},\ref{66},\ref{63},\ref{62},\ref{64},\ref{89},\ref{95},\ref{96},\ref{148},\ref{166},\ref{146},\ref{147},\ref{141},\ref{183},
\ref{212},
\ref{206},
\ref{207},
\ref{249}),
then $H$ is 
${\mathbb Z}/2{\mathbb Z}^{\oplus 2}$, ${\mathbb Z}/4{\mathbb Z}^{\oplus 2}$, 
${\mathbb Z}/2{\mathbb Z}\oplus{\mathbb Z}/4{\mathbb Z}$, 
${\mathbb Z}/2{\mathbb Z}^{\oplus 4}$, ${\mathbb Z}/2{\mathbb Z}^{\oplus 3}$, 
${\mathbb Z}/2{\mathbb Z}^{\oplus 2}\oplus{\mathbb Z}/4{\mathbb Z}$, ${\mathbb Z}/2{\mathbb Z}^{\oplus 3}$, 
${\mathbb Z}/2{\mathbb Z}^{\oplus 2}$, ${\mathbb Z}/2{\mathbb Z}\oplus{\mathbb Z}/4{\mathbb Z}$, 
${\mathbb Z}/2{\mathbb Z}^{\oplus 3}$, ${\mathbb Z}/2{\mathbb Z}^{\oplus 3}$, ${\mathbb Z}/2{\mathbb Z}^{\oplus 3}$, 
${\mathbb Z}/2{\mathbb Z}^{\oplus 2}\oplus{\mathbb Z}/4{\mathbb Z}$, ${\mathbb Z}/2{\mathbb Z}^{\oplus 4}$, ${\mathbb Z}/4{\mathbb Z}\oplus{\mathbb Z}/8{\mathbb Z}$, 
${\mathbb Z}/2{\mathbb Z}^{\oplus 2}$, ${\mathbb Z}/2{\mathbb Z}^{\oplus 3}$, ${\mathbb Z}/4{\mathbb Z}^{\oplus 2}$, 
${\mathbb Z}/2{\mathbb Z}^{\oplus 2}\oplus{\mathbb Z}/4{\mathbb Z}$, 
${\mathbb Z}/4{\mathbb Z}\oplus{\mathbb Z}/8{\mathbb Z}$, in order, as a group. 
\end{thm}
\begin{proof}
Let $X$ be the $K3$-cover of $E$, $G:=\{ s\in{\rm Aut}(X):s\ {\rm is\ a\ lift\ of\ some}\ h\in H\}$, and $p:X\rightarrow X/G$ be the quotient map.
Then $G$ is a finite abelian group, $X/G\cong {\mathbb F}_{n}$, and the branch divisor of $p$ is $B$. 
Since $b_{i}$ is power of two for each $i=1,\ldots,k$, $G\cong{\mathbb Z}/2{\mathbb Z}^{\oplus s}\oplus{\mathbb Z}/4{\mathbb Z}^{\oplus t}\oplus {\mathbb Z}/8{\mathbb Z}^{\oplus u}$ where $s,t,u\geq0$.
By Theorem \ref{thm:5}, and the assumption that $G$ has a non-symplectic automorphism of order 2 such that whose fixed locus is an empty set, we get $s+t+u\geq3$, and hence
 the numerical class of $B$ is one of 
(\ref{53},\ref{51},\ref{52},\ref{33},\ref{66},\ref{63},\ref{62},\ref{34},\ref{64},\ref{31},\ref{32},\ref{144},\ref{163},\ref{78},\ref{93},\ref{89},\ref{95},\ref{96},\ref{148},\ref{166},\ref{146},\ref{147},\ref{76},\ref{141},\ref{()},\\
\ref{216},\ref{183},\ref{212},\ref{208},\ref{206},\ref{207},\ref{249},\ref{250}) of the list in section 6.
\\

We assume that the numerical class of $B$ is (\ref{53}).
We denote $B$ by $2B^{1}_{1,0}+2B^{2}_{1,0}+2B_{2,2}+2B^{1}_{0,1}+2B^{2}_{0,1}$.
By Proposition \ref{pro:2}, $G=G^{1}_{1,0}\oplus G_{2,2}\oplus G^{1}_{0,1}\cong\mathbb Z/2\mathbb Z^{\oplus 3}$.
Let $s,t,u,\in G$ be generators of $G^{1}_{1,0}$, $G^{1}_{0,1}$, and $G_{2,2}$ respectively.
Then the non-symplectic automorphisms of $G$ are $s$, $t$, $u$, and $s\circ t\circ u$. 

From here, we will show that Fix$(s\circ t\circ u)$ is an empty set.
We assume that the curves of Fix$(s)$ are only $p^{-1}(B^{1}_{1,0})$. 
Since $s$ is a non-symplectic automorphism of order 2, the quotient space $X/\langle s\rangle$ is a smooth rational surface.
The quotient map $q:X/\langle s\rangle\rightarrow X/G\cong{\mathbb P}^{1}\times{\mathbb P}^{1}$ is the Galois cover such that the branch divisor is $2B^{2}_{0,1}+2B_{2,2}+2B^{1}_{0,1}+2B^{2}_{0,1}$, 
and the Galois group is isomorphic to ${\mathbb Z}/2{\mathbb Z}^{\oplus 2}$ as a group.
By Theorem \ref{thm:12}, there is the Galois cover $g:Y\rightarrow X/G$  whose branch divisor is $2B_{2,2}+2B^{1}_{0,1}+2B^{2}_{0,1}$ and Galois group is isomorphic to ${\mathbb Z}/2{\mathbb Z}^{\oplus 2}$ as a group. 
Since Fix$(s)$ is not an empty set and the order of $s$ is 2,  $X/\langle s\rangle$ is a smooth rational surface.
By Theorem \ref{thm:4}, there is the Galois cover $h:X/\langle s\rangle\rightarrow Y$ such that $q=g\circ h$. 
Since the degree of $q$ is 4 and that of $g$ is 4, $h$ is an isomorphism. 
Since the branch divisor of $q$ is not that of $g$, this is a contradiction. 
Therefore, Fix$(s)$ is $p^{-1}(B^{1}_{1,0})\cup p^{-1}(B^{2}_{1,0})$.
In the same way, Fix$(t)$ is $p^{-1}(B^{1}_{0,1})\cup p^{-1}(B^{2}_{0,1})$.
Therefore, by Theorem \ref{thm:5}, Fix$(s\circ t\circ u)$ is an empty set, 
and hence $E:=X/\langle s\circ t\circ u\rangle$ is an Enriques surface.
Let $H:=G/\langle s\circ t\circ u\rangle$. 
Then $E/H\cong{\mathbb P}^{1}\times{\mathbb P}^{1}$, $H\cong{\mathbb Z}/2{\mathbb Z}^{\oplus 2}$, and the branch divisor of $H$ is $B$.
It is easy to show that for an Enriques surface $E$ and a finite abelian subgroup $H$ of Aut$(E)$ such that $E/H\cong{\mathbb P}^{1}\times{\mathbb P}^{1}$ if the numerical class of $H$ is (\ref{53}), then $H\cong{\mathbb Z}/2{\mathbb Z}^{\oplus 2}$. 

As for the case of (\ref{53}), the claim is established for (\ref{89}).\\

We assume that the numerical class of $B$ is (\ref{51}).
We denote $B$ by $4B^{1}_{1,0}+4B^{2}_{1,0}+2B_{1,1}+4B^{1}_{0,1}+4B^{2}_{0,1}$.
By Proposition \ref{pro:2}, $G=G^{1}_{1,0}\oplus G_{1,1}\oplus G^{1}_{0,1}\cong\mathbb Z/2\mathbb Z\oplus\mathbb Z/4\mathbb Z^{\oplus 2}$.
Let $s,t,u,\in G$ be generators of $G^{1}_{1,0}$, $G^{1}_{0,1}$, and $G_{1,1}$ respectively.
By Theorem \ref{thm:5}, $s$ and $t$ are non-symplectic automorphism of order 4 and $u$ is a non-symplectic automorphism of order 2.
By Theorem \ref{thm:5}, $G^{2}_{1,0}$ is generated by $s\circ t^{2x}\circ u^{y}$ where $x,y=0$ or 2.
Since $(s\circ t^{2x}\circ u^{y})^{2}=s^{2}$ for $x,y=0$ or 2, we get that  Fix$(s^{2})$ is $p^{-1}(B^{1}_{1,0})\cup p^{-1}(B^{2}_{1,0})$.
As for the case of (\ref{53}), we get the claim for (\ref{51}).

As for the case of (\ref{51}), the claim is established for (\ref{141}).\\

We assume that the numerical class of $B$ is (\ref{52}).
We denote $B$ by $4B^{1}_{1,0}+4B^{2}_{1,0}+2B_{1,2}+2B^{1}_{0,1}+2B^{2}_{0,1}$.
By Proposition \ref{pro:2}, $G=G^{1}_{1,0}\oplus G_{1,2}\oplus G^{1}_{0,1}\cong\mathbb Z/2\mathbb Z^{\oplus 2}\oplus\mathbb Z/4\mathbb Z$.
Let $s,t,u\in G$ be generators of $G^{1}_{1,0}$, $G^{1}_{0,1}$, and $G_{1,2}$ respectively.
As for the case of (\ref{53}), Fix$(t)$ is $p^{-1}(B^{1}_{0,1})\cup p^{-1}(B^{2}_{0,1})$.
As for the case of (\ref{51}), Fix$(s)$ is the support of $p^{-1}(B^{1}_{1,0})\cup p^{-1}(B^{2}_{1,0})$.
As for the case of (\ref{53}), we get the claim for (\ref{141}).\\

We assume that the numerical class of $B$ is (\ref{33}).
We denote $B$ by $2B^{1}_{1,0}+2B^{2}_{1,0}+2B^{3}_{1,0}+2B_{1,4}$.
Let $s_{1},s_{2},t\in G$ be generators of $G^{1}_{1,0}$, $G^{2}_{1,0}$, and $G_{1,4}$ respectively.
By Proposition \ref{pro:2}, $G=G^{1}_{1,0}\oplus G^{2}_{1,0}\oplus G_{1,4}\cong{\mathbb Z}/2{\mathbb Z}^{\oplus 3}$.
Then the non-symplectic involutions of $G$ are $s_{1},s_{2},t$ and $s_{1}\circ s_{2}\circ t$.

We assume that Fix$(s_{1})$ is $p^{-1}(B^{1}_{1,0})\cup p^{-1}(B^{3}_{1,0})$.
Then $X/\langle s_{1}\rangle$ is a smooth rational surface, and 
the quotient map $q:X/\langle s_{1}\rangle\rightarrow X/G\cong{\mathbb P}^{1}\times{\mathbb P}^{1}$ is the Galois cover such that the branch divisor is $2B^{2}_{0,1}+2B_{1,4}$, 
and the Galois group is isomorphic to ${\mathbb Z}/2{\mathbb Z}^{\oplus 2}$ as a group.
Since ${\mathbb P}^{1}\times{\mathbb P}^{1}\backslash B^{2}_{1,0}$ is simply connected, in the same way of the proof of Theorem \ref{thm:5}, this is a contradiction.
Therefore, Fix$(s_{i})$ is $p^{-1}(B^{i}_{1,0})$ for $i=1,2$, and hence Fix$(s_{1}\circ s_{2}\circ t)$ is $p^{-1}(B^{3}_{1,0})$.
There are not an Enriques surface $E$ and a finite abelian subgroup $H$ of Aut$(E)$ such that $E/H\cong{\mathbb P}^{1}\times{\mathbb P}^{1}$ and the numerical class of the branch divisor of $H$ is (\ref{33}).

As for the case of (\ref{33}), we get the claim for (\ref{78},\ref{76}).\\

We assume that the numerical class of $B$ is (\ref{66}).
We denote $B$ by $2B^{1}_{1,0}+2B^{2}_{1,0}+2B^{3}_{1,0}+2B_{1,1}+2B^{1}_{0,1}+2B^{2}_{0,1}+2B^{3}_{0,1}$.
By Proposition \ref{pro:2}, $G=\oplus_{i=1}^{2}G^{i}_{1,0}\oplus G_{1,1}\oplus _{i=1}^{2}G^{i}_{0,1}$, and hence the number of non-symplectic automorphisms of order 2 of $G$ is 16. 
By Theorem \ref{thm:5}, $G$ has a non-symplectic automorphism of order 2 whose fixed locus is an empty set.
Furthermore, it is easy to show that for an Enriques surface $E$ and a finite abelian subgroup $H$ of Aut$(E)$ such that $E/H\cong{\mathbb P}^{1}\times{\mathbb P}^{1}$ if the numerical class of $H$ is (\ref{66}), then $H\cong{\mathbb Z}/2{\mathbb Z}^{\oplus 4}$. 
	
As for the case of (\ref{66}), the claim is established for (\ref{63},\ref{62},\ref{64},\ref{96},\ref{148},\ref{166},\ref{146},\ref{147},\ref{212},\\
\ref{207}).\\
	
We assume that the numerical class of $B$ is (\ref{34}).
We denote $B$ by $2B^{1}_{1,0}+2B^{2}_{1,0}+2B^{1}_{1,2}+2B^{2}_{1,2}$.
By Proposition \ref{pro:3}, $G=G^{1}_{1,0}\oplus G^{1}_{1,2}\oplus G^{2}_{1,2}$.
Let $s,t,u\in G$ be generators of $G^{1}_{1,0}$, $G^{1}_{1,2}$, and $G^{2}_{1,2}$ respectively.
Then the non-symplectic automorphisms of order 2 of $G$ are $s$, $t$, $u$, and $s\circ t\circ u$.
We assume that Fix$(s\circ t\circ u)$ is an empty set.
Since $(B^{i}_{1,0}\cdot B^{j}_{1,2})\not=0$ for $i,j=1,2$, 
Fix$(s)$ is $p^{-1}(B^{1}_{1,0})\cup p^{-1}(B^{2}_{1,0})$. 
Since $(B^{1}_{1,0}+B^{2}_{1,0}\cdot B^{1}_{1,2})=4$, $X/(G^{1}_{1,0}\oplus G^{1}_{1,2})$ is smooth.
Since $G=G^{1}_{1,0}\oplus G^{1}_{1,2}\oplus G^{2}_{1,2}$,
the branch divisor of the quotient map $X/(G^{1}_{1,0}\oplus G^{1}_{1,2})\rightarrow X/G\cong{\mathbb F}_{2}$ is $2B^{2}_{1,0}$ and its degree is 2. 
Since $\frac{B^{2}_{1,0}}{2}\not\in$Pic$({\mathbb P}^{1}\times{\mathbb P}^{1})$ and $X/(G^{1}_{1,0}\oplus G^{1}_{1,2})$ is smooth,
this is a contradiction .
Therefore, there are not an Enriques surface $E$ and a finite abelian subgroup $H$ of Aut$(E)$ such that $E/H\cong{\mathbb F}_{1}$ and the numerical class of branch divisor of $H$ is (\ref{34}).

As for the case of (\ref{34}), we get that there are not an Enriques surface $E$ and a finite abelian subgroup $H$ of Aut$(E)$ such that $E/H\cong{\mathbb F}_{n}$ and the numerical class of the branch divisor of $H$ is (\ref{93}).\\
	
We assume that the numerical class of $B$ is (\ref{31}).
We denote $B$ by $2B^{1}_{1,1}+2B^{2}_{1,1}+2B^{3}_{1,1}+2B^{4}_{1,1}$.
By Proposition \ref{pro:5}, $G=G^{1}_{1,1}\oplus G^{2}_{1,1}\oplus G^{3}_{1,1}$. 
Let $s_{i}\in G^{i}_{1,1}$ be a generator of $G^{i}_{1,1}$ for $i=1,2,3,4$.
By Theorem \ref{thm:5}, Fix$(s_{i})$ is not an empty set for $i=1,2,3,4$.
Since $G\cong {\mathbb Z}/2{\mathbb Z}^{\oplus 3}$, 
 $s_{4}=s_{1}\circ s_{2}\circ s_{3}$, and hence $G$ does not have a non-symplectic automorphism of order 2 whose fixed locus is an empty set.
Therefore, there are not an Enriques surface $E$ and a finite abelian subgroup $H$ of Aut$(E)$ such that $E/H\cong{\mathbb F}_{1}$ and the numerical class of the branch divisor of $H$ is (\ref{31}).

As for the case of (\ref{31}), we get that there are not an Enriques surface $E$ and a finite abelian subgroup $H$ of Aut$(E)$ such that $E/H\cong{\mathbb F}_{n}$ and the numerical class of the branch divisor of $H$ is (\ref{31},\ref{32},\ref{144},\ref{163},\ref{216}).\\	

We assume that the numerical class of $B$ is (\ref{95}).
We denote $B$ by $2B_{1,0}+2B_{1,1}+2B_{2,2}+4B^{1}_{0,1}+4B^{2}_{0,1}$.
By Corollary \ref{pro:11}, $G=G_{1,1}\oplus G_{2,2}\oplus G^{1}_{0,1}$. 
Let $q:X/\langle G_{1,0},G_{1,1},G_{2,2}\rangle\rightarrow X/G\cong{\mathbb F}_{1}$ be the quotient map. 
Then the branch divisor of $q$ is  $4B^{1}_{0,1}+4B^{2}_{0,1}$.
By Theorem \ref{thm:4}, $X/\langle G_{1,0},G_{1,1},G_{2,2}\rangle\cong{\mathbb F}_{4}$, and the branch divisor of $\langle G_{1,0},G_{1,1},G_{2,2}\rangle$ is $2B_{1,0}+2q^{\ast}B_{1,1}+2q^{\ast}B_{2,2}$. 
Let $s,t,u\in G$ be generators of $G_{1,1}$, $G_{2,2}$, and $G^{1}_{0,1}$ respectively.
Then Fix$(s)$ is the support of $p^{\ast}B_{1,0}$ and that of $p^{\ast}B_{1,1}$. 
Then as for the case of (\ref{53}), we get the claim.

As for the case of (\ref{95}), the claim is established for (\ref{183},\ref{206},\ref{249}).\\

We assume that the numerical class of $B$ is (\ref{()}).
We denote $B$ by $2B_{1,0}+2B_{1,4}+2B^{1}_{1,2}+2B^{2}_{1,2}$.
By Corollary \ref{pro:12}, $G=G_{1,4}\oplus G^{1}_{1,2}\oplus G^{2}_{1,2}$.
Let $s,t,u\in G$ be generators of $G_{1,4}$, $G^{1}_{1,2}$, and $G^{2}_{1,2}$ respectively.
Then the non-symplectic automorphisms of $G$ are $s$, $t$, $u$, and $s\circ t\circ u$.
Since each fixed locus of $s$, $t$, and  $u$ is not an empty set, by Theorem \ref{thm:5}, if $G$ has a non-symplectic automorphism of order 2 whose fixed locus is an empty set, then that is $s\circ t\circ u$.
We assume that Fix$(s\circ t\circ u)$ is an empty set.
Then we may assume that  Fix$(t)$ is $p^{-1}(B_{1,0})\cup p^{-1}(B^{1}_{1,2})$. 
Since $(B_{1,0}+B^{1}_{1,2}\cdot B_{1,4})=6$, we get $|p^{-1}(B_{1,0}\cup B^{1}_{1,2})\cap p^{-1}(B_{1,4})|=12$.
Since $s\circ t$ is a symplectic automorphism of order 2 and $p^{-1}(B_{1,0}\cup B^{1}_{1,2})\cap p^{-1}(B_{1,4})$ is contained in Fix$(s\circ t)$,
this is a contradiction.  
Therefore, there are not an Enriques surface $E$ and a finite abelian subgroup $H$ of Aut$(E)$ such that $E/H\cong{\mathbb F}_{1}$ and the numerical class of the branch divisor of $H$ is (\ref{()}).\\

We assume that the numerical class of $B$ is (\ref{208}).
We denote $B$ by $4B_{1,0}+2B_{1,3}+4B_{1,2}+2B^{1}_{0,1}+2B^{2}_{0,1}$.
By Proposition \ref{pro:7}, $G=G_{1,3}\oplus G_{1,2}\oplus G^{1}_{0,1}$.
Let $s,t,u\in G$ be generators of $G_{1,3}$, $t\in G_{1,2}$, and $u\in G^{1}_{0,1}$ respectively.
Then the non-symplectic automorphisms of $G$ are $s$, $t^{2}$, $u$, and $s\circ t^{2}\circ u$.
Since each fixed locus of $s$, $t^{2}$, and  $u$ is not an empty set by Theorem \ref{thm:5}, if $G$ has a non-symplectic automorphism of order 2 whose fixed locus is an empty set, then that is $s\circ t^{2}\circ u$.

We assume that Fix$(s\circ t^{2}\circ u)$ is an empty set.
Then Fix$(t^{2})$ is $p^{-1}(B_{1,0})\cup p^{-1}(B_{1,2})$, and  Fix$(u)$ is $p^{-1}(B^{1}_{1,0})\cup p^{-1}(B^{2}_{1,0})$.
Since $(B_{1,3}\cdot B^{1}_{0,1}+B^{2}_{0,1})=4$, we get that $X/(G_{1,3}\oplus G^{1}_{0,1})$ is smooth, 
and the branch divisor of the quotient map $f:X/(G_{1,3}\oplus G^{1}_{0,1})\rightarrow X/G\cong{\mathbb F}_{2}$ is $4B_{1,0}+4B_{1,2}$, and the Galois group is ${\mathbb Z}/4{\mathbb Z}$, which is induced by $t$.
Furthermore, since $(B_{1,3}\cdot B_{1,0}+B_{1,2})=4$ and $(B_{1,3}\cdot B^{1}_{0,1}+B^{2}_{0,1})=4$, $G/\langle s,t^{2},u\rangle$ is smooth, and the branch divisor of the quotient map $g:X/\langle s,t^{2},u\rangle\rightarrow X/G\cong{\mathbb F}_{2}$ is $2B_{1,0}+2B_{1,2}$, and the Galois group is isomorphic to ${\mathbb Z}/2{\mathbb Z}$ as a group.
Let $E_{1,0}$ and $E_{1,2}$ be the support of $g^{\ast}B_{1,0}$ and $g^{\ast}B_{1,2}$ respectively. 
Then $g^{\ast}B_{1,0}=2E_{1,0}$ and $g^{\ast}B_{1,2}=2E_{1,2}$.
Moreover, by Theorem \ref{thm:12},
there is the double cover $h:X/(G_{1,3}\oplus G^{1}_{0,1})\rightarrow X/\langle s,t^{2},u\rangle$ such that $f=g\circ h$ and the branch divisor is $2E_{1,0}+2E_{1,2}$.
Since $X/(G_{1,3}\oplus G^{1}_{0,1})$ and $X/\langle s,t^{2},u\rangle$ are smooth, we get $\frac{E_{1,0}+E_{1,2}}{2}\in$Pic$(X/\langle s,t^{2},u\rangle)$. 
Since $g^{\ast}B_{1,2}=g^{\ast}B_{1,0}+2g^{\ast}F$ in Pic$(X/\langle s,t^{2},u\rangle)$,
\[ 2E_{1,2}=2E_{1,0}+2g^{\ast}F\  {\rm in\ Pic}(X/\langle s,t^{2},u\rangle). \]
Since $X/\langle s,t^{2},u\rangle$ is a smooth rational surface, Pic$(X/(G_{1,3}\oplus G^{1}_{0,1}))$ is torsion free.
Therefore, we get 
\[ E_{1,2}=E_{1,0}+g^{\ast}F\ {\rm in\ Pic}(X/\langle s,t^{2},u\rangle), \]
and hence  
\[ E_{1,2}+E_{1,0}=2E_{1,0}+g^{\ast}F\ {\rm in\ Pic}(X/\langle s,t^{2},u\rangle).\]
Since $\frac{E_{1,2}+E_{1,0}}{2}\in$Pic$(X/\langle s,t^{2},u\rangle)$, we get 
\[ \frac{g^{\ast}F}{2}\in{\rm Pic}(X/\langle s,t^{2},u\rangle).\]
Since  $(B_{1,0}\cdot F)=1$, the degree of $g$ is two, and $\frac{g^{\ast}B_{1,0}}{2}$ and $\frac{g^{\ast}F}{2}$ are elements of Pic$(X/\langle s,t^{2},u\rangle)$, this is a contradiction.
Therefore, there are not an Enriques surface $E$ and a finite abelian subgroup $H$ of Aut$(E)$ such that $E/H\cong{\mathbb F}_{1}$ and the numerical class of the branch divisor of $H$ is (\ref{208}).\\

We assume that the numerical class of $B$ is (\ref{250}).
We denote $B$ by $2B_{1,0}+2B^{1}_{1,4}+2B^{2}_{1,4}+2B^{3}_{1,4}$.
By Corollary \ref{pro:12}, $G=\oplus_{i=1}^{3}G^{i}_{1,4}$,
Let $s_{i}\in G^{i}_{1,4}$ be a generator for $i=1,2,3$. 
Then the non-symplectic automorphisms of $G$ are $s_{i}$ and $s_{1}\circ s_{2}\circ s_{3}$ where $i=1,2,3$.
Since each fixed locus of $s_{i}$ is not an empty set for each $i=1,2,3$ by Theorem \ref{thm:5}, if $G$ has a non-symplectic automorphism of order 2 whose fixed locus is an empty set, then that is $s_{1}\circ s_{2}\circ s_{3}$.
We assume that Fix$(s_{1}\circ s_{2}\circ s_{3})$ is an empty set.
Then we may assume that Fix$(s_{1})$ is $p^{-1}(B_{1,0})\cup p^{-1}(B^{1}_{1,4})$. 
Since $(B_{1,0}+B^{1}_{1,4}\cdot B_{1,4})=4$, we get that $X/(G^{1}_{1,4}\oplus G^{2}_{1,4})$ is smooth, and 
the branch divisor of the quotient map $X/(G^{1}_{1,4}\oplus G^{1}_{1,4})\rightarrow X/G\cong{\mathbb F}_{4}$ is $2B^{3}_{1,4}$. 
This is a contradiction as the degree of the quotient map is 2.  
Therefore, there are not an Enriques surface $E$ and a finite abelian subgroup $H$ of Aut$(E)$ such that $E/H\cong{\mathbb F}_{4}$ and the numerical class of the branch divisor of $H$ is (\ref{250}).
\end{proof}
By Theorem \ref{thm:25}, we get Theorem \ref{thm:24}.
\begin{thm}\label{thm:26}
Let $E$ be an Enriques surface and $H$ be a finite abelian subgroup of ${\rm Aut}(E)$.
If $E/H$ is smooth, then $H$ is isomorphic to one of ${\mathcal AG}(E)$ as a group.
\end{thm}
\begin{proof}
Let $X$ be the $K3$-cover of $E$, $G:=\{ s\in{\rm Aut}(X):s\ {\rm is\ a\ lift\ of\ some}\ h\in H\}$, and $p:X\rightarrow X/G$ be the quotient map.
Then $G$ is a finite abelian group, $X/G=E/H$, and the branch divisor of $p$ is $B$. 
We classified $H$ for the case of $E/H\cong{\mathbb F}_{n}$ in Theorem \ref{thm:25}. 
From here, we assume that $E/H$ is smooth and $E/H\not\cong{\mathbb F}_{n}$ or ${\mathbb P}^{2}$.
Since $G$ does not have a non-symplectic automorphism whose order is odd ([\ref{bio:10}]), 
by Theorem \ref{thm:5} and \ref{thm:1}, 
$G\cong{\mathbb Z}/2{\mathbb Z}^{\oplus s}\oplus{\mathbb Z}/4{\mathbb Z}^{\oplus t}\oplus {\mathbb Z}/8{\mathbb Z}^{\oplus u}$ where $s,t,u\geq0$.
By the assumption that $G$ has a non-symplectic automorphism of order 2 such that whose fixed locus is an empty set, and the fact that $G$ is generated by non-symplectic automorphisms whose fixed locus have a curve, we get $s+t+u\geq3$.
Therefore,  $G$ is  one of the following as a group:
\[\{{\mathbb Z}/2{\mathbb Z}^{\oplus a},\ {\mathbb Z}/4{\mathbb Z}^{\oplus 3},\ 
\ {\mathbb Z}/2{\mathbb Z}^{\oplus f}\oplus{\mathbb Z}/4{\mathbb Z}^{\oplus g},\ {\mathbb Z}/2{\mathbb Z}\oplus{\mathbb Z}/4{\mathbb Z}\oplus{\mathbb Z}/8{\mathbb Z}:\]
\[ 3\leq a\leq 5,\ (f,g)=(1,2),(2,1),(3,1) \}.\]
If $G$ is one of 
 \[\{{\mathbb Z}/2{\mathbb Z}^{\oplus a},\ 
\ {\mathbb Z}/2{\mathbb Z}^{\oplus f}\oplus{\mathbb Z}/4{\mathbb Z}^{\oplus g}:3\leq a\leq 5,\ (f,g)=(1,2),(2,1),(3,1) \}\]
as a group, 
then quotient group $G/K$ of $G$ by a subgroup $K\cong{\mathbb Z}/2{\mathbb Z}$ is one of 
\[\{{\mathbb Z}/2{\mathbb Z}^{\oplus a},\ {\mathbb Z}/4{\mathbb Z}^{\oplus 2},\ {\mathbb Z}/2{\mathbb Z}^{\oplus f}\oplus{\mathbb Z}/4{\mathbb Z}: a=2,3,4\ f=1,2 \}\subset{\mathcal AG}(E)\]
as a group.
Let $f:X/G\rightarrow {\mathbb F}_{n}$ be the birational morphism.
We assume that $G\cong{\mathbb Z}/4{\mathbb Z}^{\oplus 3}$. 
By the assumption that $G\cong{\mathbb Z}/4{\mathbb Z}^{\oplus 3}$ and Theorem \ref{thm:5}, 
the numerical class of $f_{\ast}B$ is only (\ref{137}).
We denote $B$ by $2B_{1,0}+4B^{1}_{1,4}+4B^{2}_{1,4}+4B^{1}_{0,1}+4B^{2}_{0,1}+\sum_{i=1}^{n}b'_{i}B'_{i}$ where $f_{\ast}B_{1,0}=C$, $f_{\ast}B^{i}_{1,4}=C+4F$, $f_{\ast}B^{i}_{0,1}=F$, and $f_{\ast}B'_{i}=0$ in Pic$({\mathbb F}_{4})$.
Since $G\cong{\mathbb Z}/4{\mathbb Z}^{\oplus 3}$, by Theorem \ref{thm:5}, 
we get that $G=G^{1}_{1,4}\oplus G^{2}_{1,4}\oplus G^{1}_{0,1}$.
Let $s\in G^{1}_{1,4}$, $t\in G^{2}_{1,4}$, and $u\in G^{1}_{0,1}$ be generators respectively.
The non-symplectic involutions of $G$ are $s^{2}$, $t^{2}$, $u^{2}$, and $s^{2}\circ t^{2}\circ u^{2}$.
Since each fixed locus of $s^{2}$, $t^{2}$, and  $u^{2}$ is not an empty set, 
if $G$ has a non-symplectic automorphism of order 2 whose fixed locus is an empty set, then that is $s^{2}\circ t^{2}\circ u^{2}$.
If the fixed locus of $s^{2}\circ t^{2}\circ u^{2}$ is an empty set, then the fixed locus of $s\circ t\circ u$ is an empty set.
By [\ref{bio:9}], this is a contradiction.
Therefore, $G$ is not ${\mathbb Z}/4{\mathbb Z}^{\oplus 3}$ as a group.

We assume that $G\cong{\mathbb Z}/2{\mathbb Z}\oplus {\mathbb Z}/4{\mathbb Z}\oplus{\mathbb Z}/8{\mathbb Z}$.
By Theorem \ref{thm:5}, the numerical class of $f_{\ast}B$ is only (\ref{141}).
By the proof of Theorem \ref{thm:66}, $f$ is an isomorphism, i.e. $X/G\cong{\mathbb F}_{1}$. 
By Theorem \ref{thm:25}, we get the claim.
\end{proof}
By Theorem \ref{thm:25} and \ref{thm:26}, we get Theorem \ref{thm:23}.
\section{Non-Abelian case}
Let $S$ be a smooth rational surface such which is neither $\mathbb P^2$ nor an Enriques surface.
In Theorem \ref{thm:47}, we will describe the existence of a $K3$ surface $X'$ and a finite subgroup $G'$ of Aut$(X')$ such that $X'/G'\cong S$ by the result of $[\ref{bio:5}$
$,\,{\rm Proposition}\,5.1$
$\,{\rm and\ Theorem}\,5.2]$.
It does not assume that $G'$ is an abelian group, but it does not elucidate the group structure of $G'$.
First, we prepare a little.
Next, we state Theorem \ref{thm:47}.
Let 
\[{\mathcal AS}(K3):=\left\{
\begin{aligned}
&{\mathbb Z}/n{\mathbb Z},\ {\mathbb Z}/m{\mathbb Z}^{\oplus 2},\ {\mathbb Z}/2{\mathbb Z}\oplus{\mathbb Z}/k{\mathbb Z},\ {\mathbb Z}/2{\mathbb Z}^{\oplus l}\\
&n=2,\ldots,8,\ m=2,3,4,\ k=4,6,\ l=3,4 
\end{aligned}
\right\}
\]
\begin{thm}$([\ref{bio:12}])$
Let $X$ be a $K3$ surface, and $G$ be a finite abelian subgroup of Aut$(X)$.
If $G$ is symplectic, then $G\in {\mathcal AS}(K3)$ as a group.
Conversely,
for each an abelian group $G'\in {\mathcal AS}(K3)$, there is a $K3$ surface $X'$ such that $G'$ acts faithfully on $X'$ as symplectic.
\end{thm}
\begin{thm}$([\ref{bio:5},\,{\rm Proposition}\,5.1])$
Let $X$ be a $K3$ surface, and $G$ be a finite abelian subgroup of Aut$(X)$ such that $G$ acts symplectically on $X$.
Let $Y$ be a $K3$ surface such that $Y$ is a minimal resolution of the quotient space $X/G$, and $E_G$ be the lattices spanned by the curves on $Y$ arising from the resolution of the singularities of
$X/G$.

Then $E_G$ is one of the following root
lattices:
\begin{table}[h]
\label{table:data_type}
\centering
\renewcommand{\arraystretch}{1.3}
\begin{tabular}{|c||c|c|c|c|c|c|c|}
\hline
$G$&${\mathbb Z}/2{\mathbb Z}$&${\mathbb Z}/3{\mathbb Z}$&${\mathbb Z}/4{\mathbb Z}$&${\mathbb Z}/5{\mathbb Z}$&${\mathbb Z}/6{\mathbb Z}$&${\mathbb Z}/7{\mathbb Z}$&${\mathbb Z}/8{\mathbb Z}$\\ \hline 
$E_G$&$A^8_1$&$A^6_2$&$A_3^4\oplus A_1^2$&$A_4^4$&$A_5^2\oplus A_2^2\oplus A_1^2$&$A_6^3$&$A_7^2\oplus A_3\oplus A_1$
\\ \hline
\end{tabular}
\label{table:data_type}
\centering
\renewcommand{\arraystretch}{1.3}
\begin{tabular}{|c||c|c|c|c|}
\hline
$G$&${\mathbb Z}/2{\mathbb Z}^{\oplus 2}$&${\mathbb Z}/2{\mathbb Z}^{\oplus 3}$&${\mathbb Z}/2{\mathbb Z}^{\oplus 4}$&${\mathbb Z}/2{\mathbb Z}\oplus \mathbb Z/4\mathbb Z$\\ \hline 
$E_G$&$A^{12}_1$&$A^{14}_1$&$A_1^{15}$&$A_3^4\oplus A_1^4$
\\ \hline
\end{tabular}
\label{table:data_type}
	\centering
	\renewcommand{\arraystretch}{1.3}
	\begin{tabular}{|c||c|c|c|}
		\hline
		$G$&${\mathbb Z}/2{\mathbb Z}\oplus {\mathbb Z}/6\mathbb Z$&${\mathbb Z}/3{\mathbb Z}^{\oplus 2}$&${\mathbb Z}/4{\mathbb Z}^{\oplus 2}$\\ \hline 
		$E_G$&$A^{3}_5\oplus A_1^3$&$A^8_2$&$A_3^6$\\ \hline
	\end{tabular}
\end{table}

Let $M_G$ be the minimal primitive sublattice of ${\rm NS}(Y)$ which contains $E_G$. 
Then $M_G$ is an overlattice of finite index ${\rm r}_G$ of $E_G$ and its properties are the
followings
\begin{table}[h]
	\label{table:data_type}
	\centering
	\renewcommand{\arraystretch}{1.3}
	\begin{tabular}{|c||c|c|c|c|c|}
\hline
$G$&${\mathbb Z}/2{\mathbb Z}$&${\mathbb Z}/3{\mathbb Z}$&${\mathbb Z}/4{\mathbb Z}$&${\mathbb Z}/5{\mathbb Z}$&${\mathbb Z}/6{\mathbb Z}$\\ \hline 
${\rm r}_G$&$2$&$3$&$4$&$5$&$6$
\\ \hline
${\rm rank}M_G$&$8$&$12$&$14$&$16$&$16$
\\ \hline
$M_G^{\vee}/M_G$&${\mathbb Z}/2{\mathbb Z}^{\oplus 6}$&${\mathbb Z}/3{\mathbb Z}^{\oplus 4}$&$({\mathbb Z}/2{\mathbb Z}\oplus \mathbb Z/4\mathbb Z)^{\oplus 2}$&${\mathbb Z}/5{\mathbb Z}^{\oplus 2}$&${\mathbb Z}/6{\mathbb Z}^{\oplus 2}$
\\ \hline
\end{tabular}
	\label{table:data_type}
	\centering
	\renewcommand{\arraystretch}{1.3}
	\begin{tabular}{|c||c|c|c|c|c|}
		\hline
		$G$&${\mathbb Z}/7{\mathbb Z}$&${\mathbb Z}/8{\mathbb Z}$&${\mathbb Z}/2{\mathbb Z}^{\oplus 2}$&${\mathbb Z}/2{\mathbb Z}^{\oplus 3}$&${\mathbb Z}/2{\mathbb Z}^{\oplus 4}$\\ \hline 
		${\rm r}_G$&$7$&$8$&$2^2$&$2^3$&$2^4$
		\\ \hline
		${\rm rank}M_G$&$18$&$18$&$12$&$14$&$15$
		\\ \hline
		$M_G^{\vee}/M_G$&${\mathbb Z}/7{\mathbb Z}$&${\mathbb Z}/2{\mathbb Z}\oplus {\mathbb Z}/4{\mathbb Z}$&${\mathbb Z}/2{\mathbb Z}^{\oplus 8}$&${\mathbb Z}/2{\mathbb Z}^{\oplus 8}$&${\mathbb Z}/2{\mathbb Z}^{\oplus 7}$
		\\ \hline
	\end{tabular}
	\label{table:data_type}
	\centering
	\renewcommand{\arraystretch}{1.3}
	\begin{tabular}{|c||c|c|c|c|}
		\hline
		$G$&${\mathbb Z}/2{\mathbb Z}\oplus \mathbb Z/4\mathbb Z$&${\mathbb Z}/2{\mathbb Z}\oplus {\mathbb Z}/6\mathbb Z$&${\mathbb Z}/3{\mathbb Z}^{\oplus 2}$&${\mathbb Z}/4{\mathbb Z}^{\oplus 2}$\\ \hline 
		${\rm r}_G$&$8$&$12$&$3^2$&$4^2$
		\\ \hline
		${\rm rank}M_G$&$16$&$18$&$16$&$18$
		\\ \hline
		$M_G^{\vee}/M_G$&$({\mathbb Z}/2{\mathbb Z}\oplus {\mathbb Z}/6{\mathbb Z})^{\oplus 2}$&${\mathbb Z}/2{\mathbb Z}\oplus {\mathbb Z}/6{\mathbb Z}$&${\mathbb Z}/3{\mathbb Z}^{\oplus 4}$&${\mathbb Z}/4{\mathbb Z}^{\oplus 2}$
		\\ \hline
	\end{tabular}
\end{table}
 \end{thm}
\begin{thm}$([\ref{bio:5},\,{\rm Theorem}\,5.2])$
For $G\in{\mathcal AS}(K3)$ and a $K3$ surface $X$ such that $G$ acts faithfully on $X$ as symplectic,
  	a $K3$ surface $Y$ is a minimal resolution of the quotient space $X/G$ if and only if $M_{G}$ is primitively embedded in ${\rm NS}(Y)$.
  \end{thm}
In what follows, we explain Theorem \ref{thm:47}.
Let $X$ be a $K3$ surface, and $G$ be a finite subgroup of Aut$(X)$ such that $X/G$ is smooth.
If $X/G$ is neither $\mathbb P^2$ nor an Enriques surface, then $X/G$ is a Hirzebruch surface or its blow-ups.
Let $p:X\rightarrow X/G$ be the quotient morphism, and $B:=\sum_{i=1}^kb_iB_i$ be the branch divisor of $p:X\rightarrow X/G$.
Since a Hirzebruch surface $\mathbb F_n$ has the unique fibre structure,
there is the fibre structure $f:X/G\rightarrow \mathbb P^1$ induced by a Hirzebruch surface.
Then the general fibre of $f\circ p:X\rightarrow \mathbb P^1$ is a disjoint union of elliptic curves.

Let $G_s$ be the symplectic subgroup of $G$ generated by symplectic automorphisms of $G$.
We set $G_h:=\{g\in G_s\,|\,g\ \ {\rm does\ not\ replace\ irreducible\ components\ of\ }(f\circ p)^{-1}(u)\ \ {\rm for\ any\ }u\in\mathbb P^1.\}$.
\begin{lem}
The group $G_h$ is an abelia group.
\end{lem}
\begin{proof}
Since $G_h$ is a finite symplectic group, $\bigcup_{g\in G_h}{\rm Fix}(g)$ is a finite set.
There is a point $u\in\mathbb P^1$ such that $G_s$ acts freely on $(f\circ p)^{-1}(u)$.
Since the general fibre of $f\circ p:X\rightarrow \mathbb P^1$ is a disjoint union of elliptic curves, we get that $G_h$ is an abelian group. 
\end{proof}
Let $X_s:=X/G_s$ be the quotient space, $p_1:X_s\rightarrow X/G=X_s/(G/G_s)$ be the quotient map, and $f_1:=f\circ p_1:X_s\rightarrow \mathbb P^1$ be the fibre structure.
Then $p_1:X_s\rightarrow X/G=X_s/(G/G_s)$ is a cyclic cover whose branch divisor is $B$.
Let $X_h:=X/G_h$ be the quotient space, $p_2:X_h\rightarrow X_s=X_h/(G_s/G_h)$ be the quotient map.

For $f\circ p_1\circ p_2:X_h\rightarrow\mathbb P^1$,
by the Stein factorization, there are the fibre structure $h:X_h\rightarrow \mathbb P^1$ and a finite map $k:\mathbb P^1\rightarrow \mathbb P^1$ such that
$f\circ p_1\circ p_2=k\circ h$, i.e. the following diagram is commutative:
$$
\xymatrix{
	X_s\ar[r]^{f\circ p_1}&\mathbb P^1 \\
	X_h\ar[u]^{p_2}\ar[r]_{h}&\mathbb P^1\ar[u]_{k}.
}
$$
Then $X_h$ is the normalization of the fibre product of $f\circ p_1:X_s\rightarrow \mathbb P^1$ and $k:\mathbb P^1\rightarrow \mathbb P^1$.
By the definition of $G_h$, we get that $k:\mathbb P^1\rightarrow \mathbb P^1$ is a Galois cover whose Galois group is isomorphic to $G_s/G_h$ as a group.
Recall that $G_h$ is an abelian group, and by Theorem \ref{thm:43} $p_1:X_s\rightarrow X/G$ is a cyclic cover of order $b$ such that the branch divisor is $B$ and the Galois group is $G/G_s$ where $b$ is the least common multiple of $b_{1},\ldots,b_{k}$.
By the above, we have the following theorem. 
\begin{thm}\label{thm:47}
Let $S$ be a smooth rational surface which is neither $\mathbb P^2$ nor an Enriques surface.
We fix a fibre structure $f:S\rightarrow \mathbb P^1$, and an effective divisor $B:=\sum_{i=1}^kb_iB_i$ on $S$.
Let $b$ be the least common multiple of $b_{1},\ldots,b_{k}$.

Then there are a $K3$ surface $X$ and a finite subgroup $G$ of Aut$(X)$ such that $X/G=S$ and the branch divisor of the quotient morphism $p:X\rightarrow X/G$ is $B$ if and only if there are a Galois cover $k:\mathbb P^1\rightarrow\mathbb P^1$ and a cyclic cover $q:T\rightarrow S$ of degree $b$ whose the branch divisor is $B$
such that the minimal singular resolution $X'$ of the normalization of the fibre product of $f\circ q:T\rightarrow \mathbb P^1$ and $k:\mathbb P^1\rightarrow \mathbb P^1$ is a $K3$ surface, and $M_G'$ is primitively embedded in ${\rm NS}(X')$ for some $G'\in {\mathcal AS}(K3)$.
\end{thm}
\section{the list of a numerical class}
Here, we will give the list of a numerical class of an effective divisor $B=\sum_{i=1}^{k}b_{i}B_{i}$ on ${\mathbb F}_{n}$ such that $B_{i}$ is a smooth curve for each $i=1,\ldots,k$ and $K_{{\mathbb F}_{n}}+\sum_{i=1}^{k}\frac{b_{i}-1}{b_{i}}B_{i}=0$ in Pic$({\mathbb F}_{n})$. 

If there are a $K3$ surface $X$ and a finite subgroup $G$ of Aut$(X)$ such that $X/G={\mathbb F}_{0}\cong{\mathbb P}^{1}\times{\mathbb P}^{1}$, then by Theorem \ref{thm:6} the numerical class of $B$ is one of the following:
\begin{flalign}
	\label{1} 3(3C+3F)\ \ {\mathbb Z}/3{\mathbb Z}\\
	\label{14} 3C+3C+3(C+3F)\ \ {\mathbb Z}/3{\mathbb Z}^{\oplus 2}\\	
	\label{50} 3C+3C+3(C+F)+3F+3F\ \ {\mathbb Z}/3{\mathbb Z}^{\oplus 3}\\
	\label{2} 2(4C+4F)\ \ {\mathbb Z}/2{\mathbb Z}\\
	\label{16} 2C+2C+2(2C+4F)\ \ {\mathbb Z}/2{\mathbb Z}^{\oplus 2}\\
	\label{53} 2C+2C+2(2C+2F)+2F+2F\ \ {\mathbb Z}/2{\mathbb Z}^{\oplus 3}\\
	\label{15} 4C+4C+2(C+4F)\ \ {\mathbb Z}/2{\mathbb Z}\oplus {\mathbb Z}/4{\mathbb Z}
\end{flalign}
\begin{flalign}
	\label{51} 4C+4C+2(C+F)+4F+4F\ \ {\mathbb Z}/2{\mathbb Z}\oplus{\mathbb Z}/4{\mathbb Z}^{\oplus 2}\\
	\label{52} 4C+4C+2(C+2F)+2F+2F\ \ {\mathbb Z}/2{\mathbb Z}^{\oplus 2}\oplus{\mathbb Z}/4{\mathbb Z}\\
	\label{33} 2C+2C+2C+2(C+4F)\ \ {\mathbb Z}/2{\mathbb Z}^{\oplus 3}\\
	\label{66} 2C+2C+2C+2(C+F)+2F+2F+2F\ \ {\mathbb Z}/2{\mathbb Z}^{\oplus 5}\\
	\label{63} 2C+2C+2C+2(C+2F)+2F+2F\ \ {\mathbb Z}/2{\mathbb Z}^{\oplus 4}\\
	\label{62} 2C+2C+2C+2(C+F)+4F+4F\ \ {\mathbb Z}/2{\mathbb Z}^{\oplus 3}\oplus{\mathbb Z}/4{\mathbb Z}\\
	\label{3} 2(2C+2F)+2(2C+2F)\ \ {\mathbb Z}/2{\mathbb Z}^{\oplus 2}\\
	\label{34} 2C+2C+2(C+2F)+2(C+2F)\ \ {\mathbb Z}/2{\mathbb Z}^{\oplus 3}\\
	\label{64} 2C+2C+2(C+F)+2(C+F)+2F+2F\ \ {\mathbb Z}/2{\mathbb Z}^{\oplus 4}\\
	\label{12} 3(C+F)+3(C+F)+3(C+F)\ \ {\mathbb Z}/3{\mathbb Z}^{\oplus 2}\\  
	\label{13} 3C+3(C+F)+3(C+2F)\ \ {\mathbb Z}/3{\mathbb Z}^{\oplus 2}\\
	\label{31} 2(C+F)+2(C+F)+2(C+F)+2(C+F)\ \ {\mathbb Z}/2{\mathbb Z}^{\oplus 3}\\
	\label{32} 2C+2(C+F)+2(C+F)+2(C+2F)\ \ {\mathbb Z}/2{\mathbb Z}^{\oplus 3}\\
	\label{4} 2(C+F)+4(2C+2F)\\
	\label{5} 3(C+F)+3(2C+2F)\\
	\label{6} 3(C+2F)+3(2C+F)\\
	\label{7} 3C+3(2C+3F)\\
	\label{8} 2C+2(3C+4F)\\
	\label{9} 2(C+F)+2(3C+3F)\\
	\label{10} 2(C+2F)+2(3C+2F)\\
	\label{11} 2(C+3F)+2(3C+F)\\
	\label{17} 2(C+F)+3(C+F)+6(C+F)\\ 
	\label{18} 2(C+F)+4(C+F)+4(C+F)\\
	\label{19} 2C+4(2C+2F)+2F\\
	\label{20} 4C+2(2C+2F)+4F\\
	\label{23} 2C+2(3C+3F)+2F\\
	\label{22} 3C+6C+2(C+4F)\\
	\label{21} 4C+2(C+F)+4(C+2F)\\
	\label{24} 2C+2(C+F)+2(2C+3F)\\
	\label{25} 2C+2(C+2F)+2(2C+2F)\\
	\label{26} 2C+2(C+3F)+2(2C+F)\\
	\label{27} 3C+3(2C+2F)+3F\\
	\label{28} 2C+6C+3(C+3F)\\
	\label{29} 2(C+F)+2(C+F)+2(2C+2F)\\
	\label{30} 2(C+2F)+2(C+F)+2(2C+F)\\
	\label{39} 2C+4C+4(C+2F)+2F\\
	\label{40}  3C+6C+2(C+3F)+2F\\
	\label{41}  4C+4C+2(C+3F)+2F\\
	\label{42}  2C+2C+2(2C+3F)+2F
\end{flalign}
\begin{flalign}	
	\label{35} 2C+4(C+F)+4(C+F)+2F\\
	\label{36} 4C+2(C+F)+4(C+F)+4F\\
	\label{37} 2C+3(C+F)+6(C+F)+2F\\
	\label{38} 6C+2(C+F)+3(C+F)+6F\\
	\label{43} 2C+2(C+F)+2(2C+2F)+2F\\
	\label{44} 2C+2(C+2F)+2(2C+F)+2F\\
	\label{45} 3C+2(C+F)+6(C+F)+3F\\
	\label{46} 3C+3(C+F)+3(C+F)+3F\\
	\label{47} 3C+3C+3(C+2F)+3F\\
	\label{48} 2C+6C+3(C+2F)+3F\\
	\label{49} 2C+2C+2(C+F)+2(C+3F)\\
	\label{54} 2C+2(C+F)+2(C+F)+2(C+F)+2F\\
	\label{55} 2C+2C+2C+2(C+3F)+2F\\
	\label{56} 2C+2C+2(C+F)+2(C+2F)+2F\\
	\label{57} 2C+4C+4(C+F)+2F+4F\\
	\label{58} 2C+3C+6(C+F)+2F+3F\\
	\label{59} 2C+6C+3(C+F)+2F+6F\\
	\label{60} 3C+6C+2(C+F)+3F+6F\\
	\label{312} 3C+6C+2(C+F)+4F+4F\\
	\label{313} 2C+6C+3(C+F)+3F+3F\\
	\label{61} 3C+6C+2(C+2F)+2F+2F\\
	\label{65} 3C+6C+2(C+F)+2F+2F+2F\\
	\label{67} 2C+3C+6C+2F+3F+6F\\
	\label{67,1} 2C+3C+6C+2F+4F+4F\\
	\label{67,2} 2C+3C+6C+3F+3F+3F\\
	\label{68} 2C+4C+4C+2F+4F+4F\\
	\label{68,1} 2C+4C+4C+3F+3F+3F\\
	\label{69} 3C+3C+3C+3F+3F+3F\\
	\label{70} 2C+3C+6C+2F+2F+2F+2F\\
	\label{71} 2C+4C+4C+2F+2F+2F+2F\\
	\label{72} 3C+3C+3C+2F+2F+2F+2F\\
	\label{73} 2C+2C+2C+2C+2F+2F+2F+2F
\end{flalign}
If there are a $K3$ surface $X$ and a finite subgroup $G$ of Aut$(X)$ such that $X/G\cong{\mathbb F}_{1}$, then by Theorem \ref{thm:6} the numerical class of $B$ is one of the following:
\begin{flalign}
	\label{74} 2(4C+6F)\ \ {\mathbb Z}/2{\mathbb Z}\\
	\label{88} 2(2C+4F)+2(2C+2F)\ \ {\mathbb Z}/2{\mathbb Z}^{\oplus 2}
\end{flalign}
\begin{flalign}	
	\label{144} 2C+2(C+2F)+2(C+2F)+2(C+2F)\ \ {\mathbb Z}/2{\mathbb Z}^{\oplus 3}\\
	\label{163} 2(C+3F)+2(C+F)+2(C+F)+2(C+F)\ \ {\mathbb Z}/2{\mathbb Z}^{\oplus 3}\\
	\label{80} 3(3C+3F)+2F+2F\ \ {\mathbb Z}/2{\mathbb Z}\oplus{\mathbb Z}/3{\mathbb Z}\\
	\label{77} 3C+3(2C+2F)+6F+6F\ \ {\mathbb Z}/2{\mathbb Z}\oplus{\mathbb Z}/3{\mathbb Z}^{\oplus 2}\\
	\label{81} 2(4C+4F)+2F+2F\ \ {\mathbb Z}/2{\mathbb Z}^{\oplus 2}\\
	\label{79} 2C+2(3C+3F)+4F+4F\ \ {\mathbb Z}/2{\mathbb Z}\oplus {\mathbb Z}/4{\mathbb Z}\\
	\label{78} 2C+2(3C+3F)+2F+2F+2F\ \ {\mathbb Z}/2{\mathbb Z}^{\oplus 3}\\ 
	\label{93} 2C+2(C+F)+2(2C+3F)+2F+2F\ \ {\mathbb Z}/2{\mathbb Z}^{\oplus 3}\\ 
	\label{89} 2(2C+2F)+2(2C+2F)+2F+2F\ \ {\mathbb Z}/2{\mathbb Z}^{\oplus 3}\\ 
	\label{95} 2C+2(C+F)+2(2C+2F)+4F+4F\ \ {\mathbb Z}/2{\mathbb Z}^{\oplus 2}\oplus{\mathbb Z}/4{\mathbb Z}\\ 
	\label{96} 2C+2(C+F)+2(2C+2F)+2F+2F+2F\ \ {\mathbb Z}/2{\mathbb Z}^{\oplus 4}\\
	\label{145} 3(C+F)+3(C+F)+3(C+F)+2F+2F\ \ {\mathbb Z}/2{\mathbb Z}\oplus{\mathbb Z}/3{\mathbb Z}^{\oplus 2}\\
	\label{91} 3C+3(C+F)+3(C+F)+6F+6F\ \ {\mathbb Z}/2{\mathbb Z}\oplus{\mathbb Z}3{\mathbb Z}^{\oplus 3}\\
	\label{148} 2C+2(C+2F)+2(C+F)+2(C+F)+2F+2F\ \ {\mathbb Z}/2{\mathbb Z}^{\oplus 4}\\ 
	\label{90} 6C+2(C+F)+3(C+F)+12F+12F\ \ {\mathbb Z}/2{\mathbb Z}\oplus {\mathbb Z}/3{\mathbb Z}^{\oplus 2}\oplus{\mathbb Z}/4{\mathbb Z}\\
	\label{166} 2(C+F)+2(C+F)+2(C+F)+2(C+F)+2F+2F\ \ {\mathbb Z}/2{\mathbb Z}^{\oplus 4}\\
	\label{146} 2C+2(C+F)+2(C+F)+2(C+F)+4F+4F\ \ {\mathbb Z}/2{\mathbb Z}^{\oplus 3}\oplus{\mathbb Z}/4{\mathbb Z}\\
	\label{147} 2C+2(C+F)+2(C+F)+2(C+F)+2F+2F+2F\ \ {\mathbb Z}/2{\mathbb Z}^{\oplus 5}\\
	\label{75} 2C+4(2C+2F)+4F+4F\ \ {\mathbb Z}/4{\mathbb Z}^{\oplus 2}\\
	\label{76} 2C+4(2C+2F)+2F+2F+2F\ \ {\mathbb Z}/2{\mathbb Z}^{\oplus 2}\oplus{\mathbb Z}/4{\mathbb Z}\\
	\label{141} 4C+2(C+F)+4(C+F)+8F+8F\ \ {\mathbb Z}/2{\mathbb Z}\oplus{\mathbb Z}/4{\mathbb Z}\oplus{\mathbb Z}/8{\mathbb Z}\\
	\label{130} 3(2C+2F)+3(C+F)+2F+2F\\
	\label{97} 4(2C+2F)+2(C+3F)\\
	\label{100} 4(2C+2F)+2(C+2F)+2F\\
	\label{131} 4(2C+2F)+2(C+F)+2F+2F\\
	\label{149} 2(C+3F)+3(C+F)+6(C+F)\\
	\label{154} 2(C+2F)+3(C+F)+6(C+F)+2F\\  
	\label{160} 2(C+F)+3(C+F)+6(C+F)+2F+2F\\
	\label{150} 2(C+3F)+4(C+F)+4(C+F)\\ 
	\label{155} 2(C+2F)+4(C+F)+4(C+F)+2F\\
	\label{161} 2(C+F)+4(C+F)+4(C+F)+2F+2F\\
	\label{82} 2(4C+5F)+2F\\
	\label{99} 2(3C+aF)+2(C+(6-a)F), a\geq3\\
	\label{101} 2(3C+4F)+2(C+F)+2F\\	
	\label{102} 2(3C+3F)+2(C+2F)+2F\\
	\label{132} 2(3C+3F)+2(C+F)+2F+2F\\	
	\label{98} 2(2C+3F)+2(2C+3F)
\end{flalign}
\begin{flalign}	
	\label{103} 2(2C+3F)+2(2C+2F)+2F\\
	\label{151} 2(2C+4F)+2(C+F)+2(C+F)\\  
	\label{152} 2(2C+3F)+2(C+2F)+2(C+F)\\
	\label{156} 2(2C+3F)+2(C+F)+2(C+F)+2F\\	
	\label{153} 2(2C+2F)+2(C+2F)+2(C+2F)\\
	\label{157} 2(2C+2F)+2(C+2F)+2(C+F)+2F\\ 
	\label{162} 2(2C+2F)+2(C+F)+2(C+F)+2F+2F\\
	\label{164} 2(C+2F)+2(C+2F)+2(C+F)+2(C+F)\\
	\label{165} 2(C+2F)+2(C+F)+2(C+F)+2(C+F)+2F\\
	\label{86} 3C+3(2C+3F)+2F+2F\\ 
	\label{268} 3C+3(2C+2F)+2F+2F+3F\\
	\label{269} 3C+3(2C+2F)+4F+12F\\
	\label{83} 2C+4(2C+4F)\\
	\label{84} 2C+4(2C+3F)+4F\\
	\label{267} 2C+4(2C+2F)+3F+6F\\
	\label{106} 2C+3(C+2F)+6(C+2F)\\
	\label{104} 2C+3(C+2F)+6(C+F)+6F\\
	\label{105} 2C+3(C+F)+6(C+2F)+3F\\
	\label{133} 2C+3(C+F)+6(C+F)+4F+4F\\
	\label{270} 2C+3(C+F)+6(C+F)+3F+6F\\
	\label{134} 2C+3(C+F)+6(C+F)+2F+2F+2F\\
	\label{117} 2C+4(C+2F)+4(C+2F)\\
	\label{119} 2C+4(C+F)+4(C+3F)\\ 
	\label{118} 2C+4(C+F)+4(C+2F)+4F\\ 
	\label{137} 2C+4(C+F)+4(C+F)+4F+4F\\ 
	\label{279} 2C+4(C+F)+4(C+F)+3F+6F\\
	\label{138} 2C+4(C+F)+4(C+F)+2F+2F+2F\\
	\label{110} 3C+2(C+3F)+6(C+F)+3F\\
	\label{108} 3C+2(C+2F)+6(C+F)+2F+3F\\
	\label{109} 3C+2(C+F)+6(C+3F)\\ 
	\label{107} 3C+2(C+F)+6(C+2F)+6F\\ 
	\label{135} 3C+2(C+F)+6(C+F)+6F+6F\\
	\label{271} 3C+2(C+F)+6(C+F)+4F+12F\\
	\label{272} 3C+2(C+F)+6(C+F)+2F+2F+3F\\
	\label{283} 3C+3(C+F)+3(C+F)+4F+12F\\
	\label{92} 3C+3(C+2F)+3(C+F)+2F+2F\\
	\label{284} 3C+3(C+F)+3(C+F)+2F+2F+3F
\end{flalign}
\begin{flalign}
	\label{122} 4C+2(C+3F)+4(C+2F)\\
	\label{120} 4C+2(C+3F)+4(C+F)+4F\\
	\label{121} 4C+2(C+2F)+4(C+2F)+2F\\
	\label{139} 4C+2(C+2F)+4(C+F)+2F+4F\\
	\label{281} 4C+2(C+F)+4(C+F)+6F+12F\\
	\label{280} 4C+2(C+F)+4(C+F)+5F+20F\\
	\label{140} 4C+2(C+F)+4(C+2F)+2F+2F\\
	\label{282} 4C+2(C+F)+4(C+F)+2F+2F+4F\\
	\label{113} 6C+2(C+3F)+3(C+F)+6F\\
	\label{116} 6C+2(C+2F)+3(C+3F)\\
	\label{114} 6C+2(C+2F)+3(C+2F)+3F\\
	\label{136} 6C+2(C+2F)+3(C+F)+3F+3F\\
	\label{111} 6C+2(C+2F)+3(C+F)+2F+6F\\
	\label{115} 6C+2(C+F)+3(C+3F)+2F\\
	\label{112} 6C+2(C+F)+3(C+2F)+2F+3F\\
	\label{273} 6C+2(C+F)+3(C+F)+10F+15F\\
	\label{274} 6C+2(C+F)+3(C+F)+9F+18F\\
	\label{275} 6C+2(C+F)+3(C+F)+8F+24F\\
	\label{276} 6C+2(C+F)+3(C+F)+7F+42F\\
	\label{278} 6C+2(C+F)+3(C+F)+2F+3F+3F\\
	\label{277} 6C+2(C+F)+3(C+F)+2F+2F+6F\\
	\label{85} 2C+2(3C+6F)\\
	\label{85,1} 2C+2(3C+5F)+2F\\
	\label{87} 2C+2(3C+4F)+2F+2F\\
	\label{285} 2C+2(3C+3F)+3F+6F\\
	\label{126} 2C+2(C+4F)+2(2C+2F)\\
	\label{125} 2C+2(C+3F)+2(2C+3F)\\
	\label{124} 2C+2(C+2F)+2(2C+4F)\\
	\label{123} 2C+2(C+F)+2(2C+5F)\\
	\label{129} 2C+2(C+3F)+2(2C+2F)+2F\\
	\label{128} 2C+2(C+2F)+2(2C+3F)+2F\\ 
	\label{94} 2C+2(C+2F)+2(2C+2F)+2F+2F\\
	\label{127} 2C+2(C+F)+2(2C+4F)+2F\\ 
	\label{286} 2C+2(C+F)+2(2C+2F)+3F+6F\\
	\label{142} 2C+2(C+4F)+2(C+F)+2(C+F)\\
	\label{143} 2C+2(C+3F)+2(C+2F)+2(C+F)\\
	\label{159} 2C+2(C+3F)+2(C+F)+2(C+F)+2F\\ 
	\label{158} 2C+2(C+2F)+2(C+2F)+2(C+F)+2F\\ 
	\label{287} 2C+2(C+F)+2(C+F)+2(C+F)+3F+6F
\end{flalign}
If there are a $K3$ surface $X$ and a finite subgroup $G$ of Aut$(X)$ such that $X/G\cong{\mathbb F}_{2}$, then by Theorem \ref{thm:6} the numerical class of $B$ is one of the following:
\begin{flalign}
	\label{167} 3(3C+6F)\ \ {\mathbb Z}/3{\mathbb Z}\\
	\label{169} 2(4C+8F)\ \ {\mathbb Z}/2{\mathbb Z}\\
	\label{176} 2(2C+4F)+2(2C+4F)\ \ {\mathbb Z}/2{\mathbb Z}^{\oplus 2}\\
	\label{177} 2C+2(C+2F)+2(2C+6F)\ \ {\mathbb Z}/2{\mathbb Z}^{\oplus 2}\\
	\label{209} 3(C+2F)+3(C+2F)+3(C+2F)\ \ {\mathbb Z}/3{\mathbb Z}^{\oplus 2}\\ 
	\label{()} 2C+2(C+4F)+2(C+2F)+2(C+2F)\ \ {\mathbb Z}/2{\mathbb Z}^{\oplus 3}\\
	\label{216} 2(C+2F)+2(C+2F)+2(C+2F)+2(C+2F)\ \ {\mathbb Z}/2{\mathbb Z}^{\oplus 3}\\
	\label{170} 3C+3(2C+4F)+3F+3F\ \ {\mathbb Z}/3{\mathbb Z}^{\oplus 2}\\
	\label{171} 2C+2(3C+6F)+2F+2F\ \ {\mathbb Z}/2{\mathbb Z}^{\oplus 2}\\
	\label{183} 2C+2(C+2F)+2(2C+4F)+2F+2F\ \ {\mathbb Z}/2{\mathbb Z}^{\oplus 3}\\
	\label{180} 3C+3(C+3F)+3(C+3F)\ \ {\mathbb Z}/3{\mathbb Z}^{\oplus 2}\\
	\label{182} 3C+3(C+2F)+3(C+2F)+3F+3F\ \ {\mathbb Z}/3{\mathbb Z}^{\oplus 3}\\
	\label{212} 2C+2(C+2F)+2(C+2F)+2(C+2F)+2F+2F\ \ {\mathbb Z}/2{\mathbb Z}^{\oplus 4}\\
	\label{172} 2C+4(2C+4F)+2F+2F\ \ {\mathbb Z}/2{\mathbb Z}\oplus{\mathbb Z}/4{\mathbb Z}\\ 
	\label{208} 4C+2(C+3F)+4(C+2F)+2F+2F\ \ {\mathbb Z}/2{\mathbb Z}^{\oplus 2}\oplus{\mathbb Z}/4{\mathbb Z}\\
	\label{206} 4C+2(C+2F)+4(C+2F)+4F+4F\ \ {\mathbb Z}/2{\mathbb Z}\oplus{\mathbb Z}/4{\mathbb Z}^{\oplus 2}\\	
	\label{207} 4C+2(C+2F)+4(C+2F)+2F+2F+2F\ \ {\mathbb Z}/2{\mathbb Z}^{\oplus 3}\oplus{\mathbb Z}/4{\mathbb Z}\\
	\label{181} 6C+2(C+2F)+3(C+2F)+6F+6F\ \ {\mathbb Z}/2{\mathbb Z}^{\oplus 2}\oplus{\mathbb Z}/3{\mathbb Z}^{\oplus 2}\\
	\label{200} 3(C+2F)+3(2C+4F)\\
	\label{184} 2(C+2F)+4(2C+4F)\\
	\label{213} 2(C+2F)+3(C+2F)+6(C+2F)\\
	\label{214} 2(C+2F)+4(C+2F)+4(C+2F)\\ 
	\label{201} 2(3C+6F)+2(C+2F)\\ 
	\label{215} 2(2C+4F)+2(C+2F)+2(C+2F)\\
	\label{168} 3C+3(2C+6F)\\ 
	\label{173} 3C+3(2C+5F)+3F\\
	\label{288} 3C+3(2C+4F)+2F+6F\\
	\label{202} 2C+3(C+2F)+6(C+2F)+2F+2F\\
	\label{205} 2C+4(C+2F)+4(C+2F)+2F+2F\\
	\label{187} 3C+2(C+3F)+6(C+3F)\\
	\label{185} 3C+2(C+3F)+6(C+2F)+6F\\ 
	\label{186} 3C+2(C+2F)+6(C+3F)+2F\\
	\label{203} 3C+2(C+2F)+6(C+2F)+3F+3F\\ 
	\label{289} 3C+2(C+2F)+6(C+2F)+2F+6F\\
	\label{178} 3C+3(C+2F)+3(C+4F)
\end{flalign}
\begin{flalign}
	\label{179} 3C+3(C+2F)+3(C+3F)+3F\\
	\label{293} 3C+3(C+2F)+3(C+2F)+2F+6F\\
	\label{193} 4C+2(C+5F)+4(C+2F)\\  
	\label{195} 4C+2(C+4F)+4(C+2F)+2F\\ 
	\label{192} 4C+2(C+2F)+4(C+4F)\\ 
	\label{194} 4C+2(C+2F)+4(C+3F)+4F\\ 
	\label{292} 4C+2(C+2F)+4(C+2F)+3F+6F\\
	\label{191} 6C+2(C+4F)+3(C+3F)\\ 
	\label{189} 6C+2(C+4F)+3(C+2F)+3F\\ 
	\label{190} 6C+2(C+3F)+3(C+3F)+2F\\
	\label{188} 6C+2(C+3F)+3(C+2F)+2F+3F\\ 
	\label{204} 6C+2(C+2F)+3(C+3F)+2F+2F\\
	\label{175} 2C+2(3C+8F)\\
	\label{174} 2C+2(3C+7F)+2F\\
	\label{196} 2C+2(C+4F)+2(2C+4F)\\
	\label{197} 2C+2(C+3F)+2(2C+5F)\\
	\label{199} 2C+2(C+3F)+2(2C+4F)+2F\\
	\label{198} 2C+2(C+2F)+2(2C+5F)+2F\\
	\label{290} 6C+2(C+2F)+3(C+2F)+4F+12F\\
	\label{291} 6C+2(C+2F)+3(C+2F)+2F+2F+3F\\
	\label{210} 2C+2(C+3F)+2(C+3F)+2(C+2F)\\
	\label{211} 2C+2(C+3F)+2(C+2F)+2(C+2F)+2F
\end{flalign}
If there are a $K3$ surface $X$ and a finite subgroup $G$ of Aut$(X)$ such that $X/G\cong{\mathbb F}_{3}$, then by Theorem \ref{thm:6} the numerical class of $B$ is one of the following:
\begin{flalign}
	\label{217} 3C+3(2C+6F)+2F+2F\ \ {\mathbb Z}/2{\mathbb Z}\oplus {\mathbb Z}/3{\mathbb Z}\\
	\label{222} 3C+3(C+3F)+3(C+3F)+2F+2F\ \ {\mathbb Z}/2{\mathbb Z}\oplus{\mathbb Z}/3{\mathbb Z}^{\oplus 2}\\
	\label{220} 6C+2(C+3F)+3(C+3F)+4F+4F\ \ {\mathbb Z}/2{\mathbb Z}\oplus{\mathbb Z}/3{\mathbb Z}\oplus {\mathbb Z}/4{\mathbb Z}\\ 
	\label{221} 6C+2(C+3F)+3(C+3F)+2F+2F+2F\ \ {\mathbb Z}/2{\mathbb Z}^{\oplus 3}\oplus{\mathbb Z}/3{\mathbb Z}\\
	\label{298} 2C+4(2C+6F)+2F\\
	\label{299} 2C+3(C+3F)+6(C+3F)+2F\\
	\label{300} 2C+4(C+3F)+4(C+3F)+2F\\
	\label{223} 3C+2(C+5F)+6(C+3F)\\
	\label{224} 3C+2(C+4F)+6(C+3F)+2F\\
	\label{233} 3C+2(C+3F)+6(C+3F)+2F+2F\\
	\label{230} 4C+2(C+4F)+4(C+4F)
\end{flalign}
\begin{flalign}
	\label{228} 4C+2(C+4F)+4(C+3F)+4F\\
	\label{229} 4C+2(C+3F)+4(C+4F)+2F\\
	\label{295} 4C+2(C+3F)+4(C+3F)+2F+4F\\
	\label{227} 6C+2(C+6F)+3(C+3F)\\
	\label{226} 6C+2(C+5F)+3(C+3F)+2F\\
	\label{234} 6C+2(C+4F)+3(C+3F)+2F+2F\\
	\label{225} 6C+2(C+3F)+3(C+4F)+6F\\  
	\label{294} 6C+2(C+3F)+3(C+3F)+3F+6F\\
	\label{219} 2C+2(3C+10F)\\ 
	\label{301} 2C+2(3C+9F)+2F\\
	\label{231} 2C+2(C+4F)+2(2C+6F)\\
	\label{232} 2C+2(C+3F)+2(2C+7F)\\
	\label{302} 2C+2(C+3F)+2(2C+6F)+2F\\
	\label{235} 2C+2(C+4F)+2(C+3F)+2(C+3F)\\
	\label{303} 2C+2(C+3F)+2(C+3F)+2(C+3F)+2F
\end{flalign}
If there are a $K3$ surface $X$ and a finite subgroup $G$ of Aut$(X)$ such that $X/G\cong{\mathbb F}_{4}$, then by Theorem \ref{thm:6} the numerical class of $B$ is one of the following:
\begin{flalign}
	\label{236} 2C+2(3C+12F)\ \ {\mathbb Z}/2{\mathbb Z}\\
	\label{237} 2C+4(2C+8F)\ \ {\mathbb Z}/4{\mathbb Z}\\
	\label{239} 2C+2(C+4F)+2(2C+8F)\ \ {\mathbb Z}/2{\mathbb Z}^{\oplus 2}\\
	\label{245} 4C+2(C+6F)+4(C+4F)\ \ {\mathbb Z}/2{\mathbb Z}\oplus{\mathbb Z}/4{\mathbb Z}\\
	\label{249} 4C+2(C+4F)+4(C+4F)+2F+2F\ \ {\mathbb Z}/2{\mathbb Z}^{\oplus 2}\oplus{\mathbb Z}/4{\mathbb Z}
\end{flalign}
\begin{flalign}	
	\label{250} 2C+2(C+4F)+2(C+4F)+2(C+4F)\ \ {\mathbb Z}/2{\mathbb Z}^{\oplus 3}\\
	\label{240} 6C+2(C+4F)+3(C+4F)+3F+3F\ \ {\mathbb Z}/2{\mathbb Z}\oplus{\mathbb Z}/3{\mathbb Z}^{\oplus 2}\\
	\label{238} 3C+3(2C+9F)\\ 
	\label{304} 3C+3(2C+8F)+3F\\
	\label{241} 2C+3(C+4F)+6(C+4F)\\ 
	\label{248} 2C+4(C+4F)+4(C+4F)\\
	\label{305} 3C+2(C+4F)+6(C+4F)+3F\\
	\label{247} 3C+3(C+4F)+3(C+5)\\
	\label{306} 3C+3(C+4F)+3(C+4F)+3F\\
	\label{246} 4C+2(C+5F)+4(C+4F)+2F\\
	\label{242} 6C+2(C+5F)+3(C+4F)+6F\\  
	\label{244} 6C+2(C+4F)+3(C+6F)\\
	\label{243} 6C+2(C+4F)+3(C+5F)+3F\\   
	\label{296} 6C+2(C+4F)+3(C+4F)+2F+6F
\end{flalign}
If there are a $K3$ surface $X$ and a finite subgroup $G$ of Aut$(X)$ such that $X/G\cong{\mathbb F}_{5}$, then by Theorem \ref{thm:6} the numerical class of $B$ is one of the following:
\begin{flalign}
	\label{255} 4C+2(C+5F)+4(C+6F)\\
	\label{307} 4C+2(C+5F)+4(C+5F)+4F\\
	\label{254} 6C+2(C+6F)+3(C+6F)\\
	\label{252} 6C+2(C+6F)+3(C+5F)+3F\\
	\label{253} 6C+2(C+5F)+3(C+6F)+2F\\
	\label{297} 6C+2(C+5F)+3(C+5F)+2F+3F
\end{flalign}
If there are a $K3$ surface $X$ and a finite subgroup $G$ of Aut$(X)$ such that $X/G\cong{\mathbb F}_{6}$, then by Theorem \ref{thm:6} the numerical class of $B$ is one of the following:
\begin{flalign}
	\label{256} 3C+3(2C+12F)\ \ {\mathbb Z}/3{\mathbb Z}\\ 
	\label{257} 3C+3(C+6F)+3(C+6F)\ \ {\mathbb Z}/3{\mathbb Z}^{\oplus 2}\\
	\label{258} 6C+2(C+6F)+3(C+6F)+2F+2F\ \ {\mathbb Z}/2{\mathbb Z}^{\oplus 2}\oplus {\mathbb Z}/3{\mathbb Z}\\
	\label{260} 3C+2(C+6F)+6(C+6F)\\  
	\label{259} 4C+2(C+7F)+4(C+6F)\\
	\label{308} 4C+2(C+6F)+4(C+6F)+2F\\
	\label{261} 6C+2(C+8F)+3(C+6F)\\
	\label{262} 6C+2(C+7F)+3(C+6F)+2F
\end{flalign}
If there are a $K3$ surface $X$ and a finite subgroup $G$ of Aut$(X)$ such that $X/G\cong{\mathbb F}_{7}$, then by Theorem \ref{thm:6} the numerical class of $B$ is one of the following:
\begin{flalign}
	\label{309} 6C+2(C+7F)+3(C+7F)+6F
\end{flalign}
If there are a $K3$ surface $X$ and a finite subgroup $G$ of Aut$(X)$ such that $X/G\cong{\mathbb F}_{8}$, then by Theorem \ref{thm:6} the numerical class of $B$ is one of the following:
\begin{flalign}
	\label{263} 4C+2(C+8F)+4(C+8F)\ \ {\mathbb Z}/2{\mathbb Z}\oplus{\mathbb Z}/4{\mathbb Z}\\
	\label{264} 6C+2(C+8F)+3(C+9F)\\
	\label{310} 6C+2(C+8F)+3(C+8F)+3F
\end{flalign}
If there are a $K3$ surface $X$ and a finite subgroup $G$ of Aut$(X)$ such that $X/G\cong{\mathbb F}_{9}$, then by Theorem \ref{thm:6} the numerical class of $B$ is one of the following:
\begin{flalign}
	\label{265} 6C+2(C+10F)+3(C+9F)\\
	\label{311} 6C+2(C+9F)+3(C+9F)+2F 	
\end{flalign}
By Theorem \ref{thm:6} there are not a $K3$ surface $X$ and a finite subgroup $G$ of Aut$(X)$ such that $X/G\cong{\mathbb F}_{l}$ for $l=10,11$. 

If there are a $K3$ surface $X$ and a finite subgroup $G$ of Aut$(X)$ such that $X/G\cong{\mathbb F}_{12}$, then by Theorem \ref{thm:6} the numerical class of $B$ is the following:
\begin{flalign}
	\label{266} 6C+2(C+12F)+3(C+12F)\ \ {\mathbb Z}/2{\mathbb Z}\oplus {\mathbb Z}/3{\mathbb Z}
\end{flalign}

\end{document}